\documentclass{article}
\usepackage[utf8]{inputenc}
\usepackage{savesym}
\usepackage{amsthm,amssymb,amsfonts,amsmath,bm,xcolor}
\usepackage{newtxmath}
\usepackage{faktor}
\usepackage{blkarray}
\savesymbol{bigtimes}
\usepackage{enumitem}
\usepackage[margin=1in]{geometry}
\usepackage[textsize=tiny]{todonotes}
\usepackage{caption}
\usepackage{subcaption}
\usepackage{tikz-cd}
\usepackage{hyperref}
\usepackage{fix-cm}

\DeclareFontFamily{U}{mathx}{\hyphenchar\font45}
\DeclareFontShape{U}{mathx}{m}{n}{
      <5> <6> <7> <8> <9> <10> gen * mathx
      <10.95> mathx10 <12> <14.4> <17.28> <20.74> <24.88> mathx12
      }{}
\DeclareSymbolFont{mathx}{U}{mathx}{m}{n}
\DeclareFontSubstitution{U}{mathx}{m}{n}
\DeclareMathSymbol{\bigboxplus}   {1}{mathx}{"D0}

\newcommand\added[1]{{#1}}

\newtheorem*{theorem*}{Theorem}
\newtheorem{theorem}{Theorem}
\newtheorem{lemma}[theorem]{Lemma}
\newtheorem{proposition}[theorem]{Proposition}
\newtheorem{corollary}[theorem]{Corollary}
\theoremstyle{definition}
\newtheorem{definition}[theorem]{Definition}
\theoremstyle{remark}
\newtheorem{remark}[theorem]{Remark}

\newtheorem{question}[theorem]{Question}
\newtheorem{example}[theorem]{Example}

\numberwithin{theorem}{section}

\newcommand\cA{{\mathcal A}}

\newcommand\cC{{\mathcal C}}

\newcommand\cH{{\mathcal H}}
\newcommand\cI{{\mathcal I}}

\newcommand\cL{{\mathcal L}}
\newcommand\cM{{\mathcal M}}
\newcommand\cP{{\mathcal P}}
\newcommand\cR{{\mathcal R}}
\newcommand\cS{{\mathcal S}}
\newcommand\cV{{\mathcal V}}

\newcommand\CC{{\vmathbb C}}
\newcommand\FF{{\vmathbb F}}
\newcommand\HH{{\vmathbb H}}
\newcommand\KK{{\vmathbb K}}
\newcommand\NN{{\vmathbb N}}
\newcommand\PP{{\vmathbb P}}
\newcommand\QQ{{\vmathbb Q}}
\newcommand\RR{{\vmathbb R}}
\renewcommand\SS{{\vmathbb S}}
\newcommand\TT{{\vmathbb T}}
\newcommand\ZZ{{\vmathbb Z}}
\newcommand\HS{\cH^\succ}
\newcommand\cHS{\cH^\succeq}

\newcommand\bp{{\bm p}}
\newcommand\bq{{\bm q}}
\newcommand\br{{\bm r}}

\newcommand\bu{{\bm u}}
\newcommand\bv{{\bm v}}
\newcommand\bw{{\bm w}}
\newcommand\bx{{\bm x}}
\newcommand\by{{\bm y}}
\newcommand\bz{{\bm z}}
\newcommand\bX{{\bm X}}
\newcommand\balpha{{\bm \alpha}}
\newcommand\blambda{{\bm \lambda}}

\newcommand{\0}{{\vmathbb 0}}
\newcommand{\1}{{\vmathbb 1}}

\newcommand\SetOf[2]{\left\{\left.#1\vphantom{#2}\ \right|\ #2\vphantom{#1}\right\}}
\DeclareMathOperator{\conv}{conv}
\DeclareMathOperator{\cone}{cone}
\DeclareMathOperator{\lc}{lc}
\DeclareMathOperator{\lt}{lt}
\DeclareMathOperator{\lexp}{lp}
\DeclareMathOperator{\supp}{supp}
\DeclareMathOperator{\sgn}{sgn}

\DeclareMathOperator{\sep}{sep}

\DeclareMathOperator{\val}{val}
\DeclareMathOperator{\sval}{sval}
\DeclareMathOperator{\fval}{fval}

\title{Convex geometry over ordered hyperfields}
\author{James Maxwell\thanks{School of Mathematics, University of Bristol. E-mail: \texttt{james.maxwell@bristol.ac.uk}} \and Ben Smith\thanks{School of Mathematical Sciences, Lancaster University. E-mail: \texttt{b.smith9@lancaster.ac.uk}}}
\begin{document}

\maketitle

\begin{abstract}
 We initiate the study of convex geometry over ordered hyperfields.
 %We explore the connection between orders and order relations over hyperfields.
 %We propose the former as the most suitable candidate for developing a theory of convex geometry in this general algebraic setting.
 %In the case of ordered stringent hyperfields, we specialise the existing classification to give a complete characterisation.
 We define convex sets and halfspaces over ordered hyperfields, presenting structure theorems over hyperfields arising as quotients of fields.
 We prove hyperfield analogues of Helly, Radon and Carathéodory theorems.
 %This leads to a structure result for convex sets over quotient hyperfields.
 We also show that arbitrary convex sets can be separated via hemispaces.
 Comparing with classical convexity, we begin classifying hyperfields for which halfspace separation holds.
 In the process, we demonstrate many properties of ordered hyperfields, including a classification of stringent ordered hyperfields.
\end{abstract}

\section{Introduction}

Hyperfields are structures that generalise fields by allowing the addition operation to be multivalued, i.e. the sum $a\boxplus b$ may be a set rather than a singleton.
Multivalued operations were first considered by Marty~\cite{MAR} in the context of hypergroups, but it was Krasner who first introduced the notion of a hyperfield~\cite{KRA:57, KRA:83}.
They have since been studied within the context of a number of mathematical problems in which multivalued addition is required, including number theory~\cite{CC}, real fields~\cite{Mars} and more recently matroid theory~\cite{BB}.
They rose to prominence within tropical geometry via the articles of Viro~\cite{VIR,VIR2}, who noted that enriching the tropical semiring $(\RR, \max, +)$ with a hyperfield structure turns tropical varieties in genuine algebraic varieties.
Since then, hyperfields have proved to be an invaluable tool within tropical geometry~\cite{JJ,JJb,JSY,JM,BL}.

Convex geometry over the tropical semiring has been studied under the guise of `max-plus linear algebra' for over 30 years~\cite{cohen+gaubert+quadrat, butkovic}.
It has applications to a number of areas of mathematics, including commutative algebra~\cite{Develin+Yu:07, DJS, ADS}, phylogenetics~\cite{LSTY:17, Develin+sturmfels} and economics~\cite{Joswig:17}.
However, possibly its most striking application is its connection to computational complexity via two different problems.
The first is via the complexity of linear programming, where tropical methods have had a number of recent successes constructing pathological linear programs for various interior point methods~\cite{ABGJ:21, AGV:22}.
The second is via mean payoff games, known to sit in complexity class NP $\cap$ co-NP but for which no polynomial time algorithm is known. 
The feasibility problem for tropical linear programs is equivalent to mean payoff games~\cite{AGG:12}.
These two connections have lead to various generalised applications, such as semidefinite programming and stochastic mean payoff games~\cite{AGS:18,AGS:20}.

Despite these successes, there is much missing from the theory of tropical convexity that is present in classical convexity.
A key issue is that the tropical semiring has an implicit non-negativity constraint not present in classical convexity.
To rectify this, there has been a recent program of work to understand convex structures over the signed tropical numbers~\cite{LV,LS}.
While some of the theory carries over, signed tropical convexity exhibits some seemingly odd behaviour such as intersections of halfspaces not necessarily being convex.
This is the motivation for this work: we would like to understand what features of signed tropical convexity hold for general hyperfields, and what is unique to the signed tropical hyperfield.
Moreover, we believe that initiating this study will open the door to other `tropical-like' hyperfields that may be exploited in further applications.

\subsection{Our results}
The two classes of hyperfields we focus on throughout are stringent hyperfields and quotient hyperfields.
A hyperfield $\HH$ is \emph{stringent} if $a \boxplus b$ is a (non-singleton) set if and only if $a = -b$.
A \emph{quotient} hyperfield is a hyperfield that arises as the field quotient $\HH = \FF/U$ where $\FF$ is a field and $U \subseteq \FF^\times$ is a multiplicative subgroup of the group of units.
Both of these families of hyperfields have additional structure that make them easier to work with.
Moreover, they satisfy the containment given in Figure~\ref{fig:hyp+classes}.
\begin{figure}
    \[
\{\text{fields}\} \, \subsetneq \, \{\text{stringent hyperfields}\} \, \subsetneq \, \{\text{quotient hyperfields}\} \, \subsetneq \, \{\text{hyperfields}\}
\]
    \caption{Containment of classes of hyperfields.}
    \label{fig:hyp+classes}
\end{figure}

The key objects in our study are \emph{ordered hyperfields}, a hyperfield $\HH$ with an \emph{ordering} $\HH^+$, a distinguished set of positive elements satisfying
\[
\HH^+ \boxplus \HH^+ \subseteq \HH^+ \quad , \quad \HH^+ \odot \HH^+ \subseteq \HH^+ \quad , \quad \HH = \HH^+ \sqcup \{\0\} \sqcup \HH^- \, .
\]
These are the focus of Section~\ref{sec:ordered+hyperfields}.
Using results in~\cite{BS}, we prove a classification of ordered stringent hyperfields (Proposition~\ref{prop:ord+stringent+class}).
Explicitly, we show they can be written as a semi-direct product $\HH\rtimes G$, where $G$ is an ordered abelian group and $\HH$ is the sign hyperfield $\SS$ or an ordered field $\FF$.
We also show the strict containments in Figure~\ref{fig:hyp+classes} hold for ordered hyperfields as well, except the final case where it remains open to find an ordered hyperfield that does not arise as a quotient.
Note that we do not use an order-theoretic approach to ordered hyperfields; we discuss this approach in Appendix~\ref{sec:order+relations} and its shortcomings for convex geometry.

% We end the section with a comparison between orderings and order relations on hyperfields.
% If $\HH = \FF$ is a field, each ordering $\FF^+$ is in one-to-one correspondence with a strict total order $\prec$ that is compatible with the field operations, defined by $a \prec b$ if and only if $b-a \in \FF^+$.
% We show that for hyperfields, one can associate to each ordering $\HH^+$ a strict partial order $\prec_{\HH^+}$ that is compatible with multiplication, and that $\prec_{\HH^+}$ is total if the hyperfield is stringent (Proposition~\ref{prop:partial+order}).
% However, we also show multiple occurrences where this correspondence breaks down for more general hyperfields, as well as the difficulties of compatibility with multi-valued addition.

In Section~\ref{sec:convex+sets}, we introduce the notions of convex and conic sets over ordered hyperfields.
We prove that they satisfy many of the properties that one would want for convex sets, including being closed under intersections, projections and products.
We also give a structure theorem for convex sets over a quotient hyperfield $\HH=\FF/U$ as the union of lifts to convex sets over the ordered field $\FF$.
\begin{theorem*}[\ref{thm:quotient+structure}]
    Let $\HH = \FF/U$ be an ordered quotient hyperfield with quotient map $\tau\colon \FF \rightarrow \HH$.
    Then for any $T \subseteq \FF^d$, we have
    \[
    \conv(\tau(T)) = \bigcup_{\tau(T) = \tau(T')} \tau(\conv(T'))
    \]
\end{theorem*}
\noindent We close this section by proving three classical theorems in convex geometry for quotient hyperfields: Radon's Theorem, Helly's Theorem and Carath\'eodory's Theorem (Theorems~\ref{thm:radon}, \ref{thm:helly}, \ref{thm:caratheodory+hyp}).
We show these by pushing the results over ordered fields to quotient hyperfields via properties of the quotient map.

A crucial theorem within convex geometry is the Hyperplane Separation Theorem, which states given a polyhedron $P$ and a point $\bq$ not contained in $P$, there exists a hyperplane separating $P$ from $\bq$.
This theorem leads to external representations of polyhedra as intersections of closed halfspaces, as well as internal representations via convex hulls of vertices and rays.
This ability to jump between internal and external notions of convexity is very fruitful, and a natural tool to develop for hyperfield convexity.
Notions of external convexity, and their obstacles, are the focus of the remaining sections.

Given their prolific applications in convex geometry, we introduce and study halfspaces over hyperfields in Section~\ref{sec:halfspace}.
Given an ordered hyperfield $\HH$ and an affine polynomial $\phi \in \HH[X_1, \dots, X_d]$, we define the associated \emph{open} and \emph{closed halfspace} respectively as
\[
\HS(\phi):= \{ \bp \in \HH^d \, : \, \phi(\bp) \subseteq \HH^+\} \quad , \quad
\cHS(\phi) := \{ \bp \in \HH^d \, : \, \phi(\bp) \cap (\HH^+ \cup \{\0\} ) \neq \emptyset\} \, .
\]
We show that open halfspaces are convex, but that in general finitely generated convex sets are not equal to the intersection of the open halfspaces containing them, merely contained in them (Proposition~\ref{prop:conv+in+open} \& Example~\ref{ex:S-non-open-sep}).
This marks a departure even from signed tropical convexity.
We show that closed halfspaces behave even worse in general: they are not convex and do not decompose into an open halfspace and the variety defined by $\phi$.
In both cases, we give a structure theorem for halfspaces over quotient hyperfields in terms of their lifts over fields.
We show that open halfspaces over $\HH = \FF/U$ are equal to the image of the intersection of \emph{all} lifted open halfspaces over $\FF$ (Theorem \ref{thm:quotient+open}). 
We also show that closed halfspaces are equal to the image of \emph{any} lifted closed halfspace over $\FF$ (Theorem \ref{clsd_hs_eql}).
% \[
% \cHS(\tau_*(\phi)) = \tau(\cHS(\phi))
% \]

% \james{Incorporated this into the main body of text, possibly a little wordy.}

% \begin{theorem*}[\ref{thm:quotient+open} \& \ref{clsd_hs_eql}]
   %  Let $\HH = \FF/U$ be an ordered quotient hyperfield with quotient map $\tau\colon \FF \rightarrow \HH$.
%  For any affine polynomial $\phi \in \FF[X]$, the following holds
   % \[
   % \HS(\tau_*(\phi)) = \tau\left(\bigcap_{\tau_*(\phi) = \tau_*(\psi)} \HS(\psi)\right) \quad , \quad \cHS(\tau_*(\phi)) = \tau(\cHS(\phi))  \, .
   % \]
%\end{theorem*}

Given that halfspaces over hyperfields behave worse than over fields, we shift our focus to a different object for separation in Section~\ref{sec:kakutani}.
We consider \emph{hemispaces} over hyperfields, convex sets whose complement is also convex.
Our main theorem of this section is a hemispace separation theorem over quotient hyperfields.
\begin{theorem*}[\ref{thm:kakutani}]
    Let $\HH = \FF/U$ be an ordered quotient hyperfield.
    If $A,B \subseteq \HH^d$ are two disjoint convex sets, then there exists a hemispace $X \subseteq \HH$ such that $A \subseteq X$ and $B \subseteq \HH^d \setminus X$.
\end{theorem*}
%The proof of this theorem follows from the equivalent \emph{Pasch property}
\noindent An immediate corollary is that any convex set over a quotient hyperfield is equal to the intersection of the hemispaces containing it.
We end the section by characterising hemispaces over $\HH = \FF/U$ in terms of hemispaces over $\FF$ that are closed under an action of $U$.

While hemispace separation works over any quotient hyperfield, it was shown in~\cite{LV} that there are hyperfields for which the (open) Hyperplane Separation Theorem does hold, namely the signed tropical hyperfield.
We close our investigation by beginning to classify for which hyperfields the Hyperplane Separation Theorem holds.
We restrict our focus to stringent hyperfields, all of which are of the form $\HH \rtimes G$ where $\HH$ is either the sign hyperfield $\SS$ or an ordered field $\FF$.
\begin{theorem*}[\ref{thm:separation}]
    Let $\HH \rtimes G$ be an ordered stringent hyperfield with a dense ordering.
    Consider a finitely generated convex set $\conv(T) \subseteq \HH^d$ and point $\bp \notin \conv(T)$.
    There exists an open halfspace $\HS(\phi)$ such that $\conv(T) \subseteq \HS(\phi)$ and $\bp \notin \HS(\phi)$ if either
    \begin{itemize}
        \item $\HH = \SS \rtimes G$,
        \item $\HH = \FF \rtimes G$, and $T$ and $\bp$ are sufficiently generic.
    \end{itemize}
\end{theorem*}
\noindent The proof techniques used are similar to those in~\cite{LV} for the signed tropical semiring.
We develop Fourier-Motzkin elimination for stringent hyperfields. 
We then use this to prove a Farkas-type result for hyperfields of the form $\SS \rtimes G$, which provably holds over $\FF \rtimes G$ in generic cases.
It remains an open question whether one can remove this genericity assumption.
We end with some possible applications and open questions, including vector axioms for matroids over ordered hyperfields.

\subsection*{Acknowledgments}
We thank Jeff Giansiracusa and Georg Loho for helpful conversations, and Mateusz Skomra for communicating related work.
We also thank Oliver Lorscheid and Nelly Villamizar for helpful comments on an earlier draft that appeared in JM's thesis, and Christos Massouros for a number of comments including Remark~\ref{rem:vector+space}.
\added{We also thank an anonymous reviewer for detailed feedback and comments.}
This work was partially carried out during JM's PhD, where they were supported by Swansea University, College of Science Bursary EPSRC DTP (EP/R51312X/1).
BS and JM were both also supported by the Heilbronn Institute for Mathematical Research. 

\section{Background on hyperfields}\label{sec:background}

Given a set $\HH$, we denote the set of non-empty subsets of $\HH$ by $\cP(\HH)^*$.
A \emph{hyperoperation} on $\HH$ is a map $\boxplus\colon\HH\times\HH\rightarrow \cP(\HH)^*$ that sends two elements $a,b \in \HH$ to a non-empty subset $a \boxplus b$ of $\HH$.
This map can be extended to strings of elements, or sums of subsets, via:
\begin{align*}
    a_1 \boxplus a_2 \boxplus \dots \boxplus a_k &:= \bigcup_{a' \in a_2 \boxplus \dots \boxplus a_k} a_1 \boxplus a' \quad , \quad a_i \in \HH \, , \\
    A \boxplus B &:= \bigcup_{a \in A, b \in B} a \boxplus b \quad , \quad A, B \subseteq \HH \, .
\end{align*}
The hyperoperation $\boxplus$ is called \emph{associative} if it satisfies
\[
(a \boxplus b) \boxplus c = a \boxplus (b \boxplus c) \quad , \quad \forall a,b,c \in \HH \, , 
\]
and is called \emph{commutative} if it satisfies
\[
a \boxplus b = b \boxplus a \quad \forall a,b \in \HH \, .
\]
\begin{definition}

A canonical hypergroup is a tuple $(\HH, \boxplus, \0)$ where $\boxplus$ is an associative and commutative hyperoperation satisfying:
\begin{itemize}
    \item (Identity) $\0 \boxplus a = \{a\}$ for all $x \in \HH$,
    \item (Inverses) For all $a \in \HH$, there exists a unique $(-a) \in \HH$ such that $\0 \in a \boxplus -a$,
    \item (Reversibility) $a \in b \boxplus c$ if and only if $c \in a \boxplus (-b)$.
\end{itemize}
\end{definition}
We note that the reversibility axiom is equivalent to $a \in b \boxplus c \Leftrightarrow -a \in -b \boxplus -c$.
In the literature, canonical hypergroups are sometimes called abelian hypergroups.

 It will also be convenient for us to define the inverse of a set: given $A \subseteq \HH$, its inverse is the set of inverses of its elements, i.e.
\[
-A = \left\{-a \mid a \in A\right\} \, .
\]
Note that by the alternative reversibility axiom, we have
\[
-(a \boxplus b) = -a \boxplus -b \, .
\]
% \ben{This seemed useful to spell out and comes up a lot in compatibility.}

\begin{definition}
A \emph{hyperring} is a tuple $(\HH, \boxplus, \odot, \0, \1)$ such that,
\begin{itemize}
\item $(\HH, \boxplus, \0)$ is a canonical hypergroup,
\item $(\HH, \odot, \1)$ is a commutative monoid,
\item $a\odot(b \boxplus c) = a\odot b \boxplus a \odot c$ for all $a,b,c \in \HH$,
\item $\0 \odot a = \0$ for all $a \in \HH$. 
\end{itemize}

A \emph{hyperfield} is a hyperring such that $(\HH, \odot, \1)$ is an abelian group, i.e. for every $a \in \HH^{\times} = \HH \setminus \{\0\}$, there exists a unique element $a^{-1}$ such that $a \odot a^{-1} = \1$.
\end{definition}

\begin{example}
A field $\FF$ can be viewed as a hyperfield in a trivial manner, where the hyper-addition is defined as $a \boxplus b = \{a+b\}$. 
\end{example}

\begin{example}
The \emph{Krasner hyperfield} is the set $\KK := \{\0,\1\}$ with the standard multiplication and hyper-addition defined as
\[
\0 \boxplus \0 = \0 \quad , \quad \0 \boxplus \1 = \1 \boxplus \0 = \1 \quad , \quad \1 \boxplus \1 = \{ \0 , \1\} \, .
\]
\end{example}

\begin{example}
The \emph{sign hyperfield} is the set $\SS :=\{-\1,\0,\1\}$ with standard multiplication and hyper-addition defined as
\[
\0 \boxplus a = a \quad , \quad a \boxplus a = a \quad , \quad \1 \boxplus -\1 = \{\0 , \1 , -\1\} \quad \forall a \in \SS .
\]
\end{example}

\begin{example}
The \emph{tropical hyperfield} is the set $\TT := \RR \cup \{-\infty\}$ with hyperfield operations
\begin{align*}
a \boxplus b &= 
\begin{cases}
\max(a,b)  &\text{if} \quad a \neq b\\
\{ c \, | \, c \le a\} \cup \{-\infty\} &\text{if} \quad a = b
\end{cases} \, , \\
a \odot b &= a + b \, .
\end{align*}
The additive and multiplicative identity elements are $\0 = -\infty$ and $\1 = 0$.
Note that we can obtain an isomorphic hyperfield by replacing $\max$ with $\min$ and $-\infty$ with $\infty$.
\end{example}

\begin{example}
The \emph{signed tropical hyperfield} (or real tropical hyperfield) is the set $\TT\RR:= (\{\pm 1\} \times \RR) \cup \{-\infty\}$, where $\{-1\} \times \RR$ is a `negative' copy of the tropical numbers.
Its hyperfield operations are
\begin{align*}
    (a_1,b_1) \boxplus (a_2,b_2) &=
    \begin{cases}
        (a_1,b_1), & \text{if} \, b_1 > b_2,\\
        (a_2,b_2), & \text{if} \, b_2 > b_1,\\
        (a_1,b_1), & \text{if} \, a_1 = a_2, \, \text{and} \, b_1 = b_2,\\
        \{(\pm 1, c) \mid c \leq b_1\} \cup \{-\infty\}, & \text{if} \, a_1 = -a_2, \, \text{and} \, b_1 = b_2 ,
    \end{cases} \\
    (a_1,b_1) \odot (a_2,b_2) &= (a_1 \cdot a_2, b_1 + b_2) \, .
\end{align*}
The additive and multiplicative identity elements are $\0 = -\infty$ and $\1 = (1,0)$.
As with $\TT$, we can obtain an isomorphic hyperfield by replacing $\max$ with $\min$ and $-\infty$ with $\infty$.
\end{example}
%\ben{Maybe some extra flavour text about tropical and signed tropical, depends what goes in the intro.}

We now introduce two key classes of hyperfields that will be central to our study.
\begin{definition}
%A hyperfield $\HH$ is \emph{stringent} if $|a \boxplus b| > 1$ implies that $a = -b$, or equivalently, $\0 \in a \boxplus b$.
A hyperfield $\HH$ is \emph{stringent} if $a \boxplus b$ is a (non-singleton) set if and only if $a = -b$.
\end{definition}

The multivalued nature of hyper-addition can be made controllable by restricting to the class of stringent hyperfields, as sets only arise when summing additive inverses. As such, stringency is equivalent to $a \boxplus b$ being a non-singleton set if and only if it contains the element $\0$.
Certain notions in classical convexity are easier to extend to stringent hyperfields than general hyperfields, as we can work with single-valued addition except is specific circumstances.

\begin{definition}
A \emph{quotient hyperfield} is a hyperfield $\HH = \FF/U$ arising as the quotient of a field $(\FF,+,\cdot)$ by a multiplicative subgroup $U \subseteq \FF^{\times}$.
The elements of $\HH$ are cosets $\bar{a} := a \cdot U = \{ a \cdot u \mid u \in U\}$, and the operations are inherited from the field operations:
\begin{align*}
\bar{a} \boxplus \bar{b} &= \{ \bar{c} \mid c \in \bar{a}+\bar{b}\} \, ,\\
\bar{a} \odot \bar{b} &= \overline{a \cdot b}
\end{align*}
%     Let $(\FF,+,\cdot)$ be a field and take $U \subseteq \FF^{\times}$ a multiplicative subgroup of the units. Then, the quotient is defined as $\FF/U := \FF^{\times}/U \cup \{0\}$, which has a hyperfield structure due to the following operations. Elements of $\FF/U$ are cosets, defined as $\bar{a} := a \cdot U = \{ a \cdot u \, : \, u \in U\}$. The multiplication is inherited from the field $\FF$, $\bar{a} \odot \bar{b} = \overline{a \cdot b}$, and the multivalued addition is defined as; 
% $$ \bar{a} \boxplus \bar{b} : = \{ \bar{c} \, : \, c = a \cdot u + b \cdot v \, , \, u,v \in U\}.$$
\end{definition}
Quotient hyperfields, which were introduced by~\cite{KRA:83}, allow one to use field structure to control the multivalued nature of hyper-addition.
We note that one can generalise this construction by replacing $\FF$ with a general hyperfield.
However one loses the benefits of field structure in this process, so we will restrict to this case.

We shall see in Section~\ref{sec:class-ord-strngt} that all stringent hyperfields are quotient hyperfields, but the converse does not hold.
There are also examples of hyperfields that do not arise as quotient hyperfields, the first of which was due to~\cite{MAS}.
\begin{example}\label{ex:quotient+hyperfields}
Note that all of the examples we have seen so far, $\FF, \KK, \SS, \TT$ and $\TT\RR$, are stringent.
For instance, the only sets we can obtain in each example are the following:
\begin{align*}
\KK &\colon \1 \boxplus \1 = \{\0,\1\} & \SS &\colon \1 \boxplus -\1 = \{-\1, \0, \1\} \\
\TT &\colon a \boxplus a = \{c \mid c \leq a\} \cup \{-\infty\} & \TT\RR &\colon (1,b) \boxplus (-1,b) = \{(\pm 1, c) \mid c \leq b\} \cup \{-\infty\}
\end{align*}
In each case, this occur by adding additive inverses together, as each set contains the zero element of its respective hyperfield.

Each of the examples we have see are also quotient hyperfields.
For instance, we can realise the following as the quotients
\[
\FF \cong \FF/\1 \quad , \quad \KK \cong \FF/\FF^{\times} \quad , \quad \SS \cong \RR/\RR_{>0}  \quad ,
\] 
where $\FF\neq \FF_2$ is any field.
The quotient construction for $\TT$ and $\TT\RR$ is slightly more involved: we shall see how to obtain them as quotients of fields of Hahn series in Section~\ref{sec:class-ord-strngt}.
%See \cite{TG} for a detailed description of how both $\TT$ and $\TT\RR$ are constructed as quotients of the field of Hahn series. Note that when $\HH = \FF_2$ the Krasner hyperfield is not isomorphic to $\FF_2/\FF_2^\times$.
\end{example}

\begin{example}
The \emph{phase hyperfield} $\PP = \SetOf{z \in \CC}{|z| = 1} \cup \{0\}$ is the hyperfield with operations
\begin{align*}
    z_1 \odot z_2 &= z_1\cdot z_2 \\
    z_1 \boxplus z_2 &= \begin{cases}
        z_1 & z_1 = z_2 \\
        \{z_1, 0, -z_1\} & z_2 = -z_1 \\
        \text{shortest open arc between } z_1, z_2 & \text{otherwise}
    \end{cases}
\end{align*}
It is not a stringent hyperfield, as the sum of two generic elements will be an arc.
However, it is a quotient hyperfield: it can be constructed as the quotient $\PP = \CC/\RR_{>0}$ where one works with the complex numbers but forgetting the modulus of an element.
\end{example}

\begin{definition}
A map between hyperfields $f\colon\HH_1\rightarrow\HH_2$ is a \emph{hyperfield homomorphism} if it satisfies
\begin{enumerate}
    \item $f(a \boxplus_1 b) \subseteq f(a) \boxplus_2 f(b)$ \, ,
    \item $f(a \odot_1 b) = f(a) \odot_2 f(b)$ \, ,
    \item $f(\1_1) = \1_2$ \, ,
    \item $f(\0_1) = \0_2$ \, .
\end{enumerate}
\end{definition}
As with field homomorphisms, the only element in the preimage of $\0_2$ is $\0_1$.
To see this, let $a\neq \0_1$ be an element such that $f(a) = \0_2$. Then
\[
f(\1_1) = f(a\odot_1 a^{-1}) = f(a)\odot_2 f(a^{-1}) = \0_2 \odot_2 f(a^{-1}) = \0_2 \, ,
\]
resulting in a contradiction.

\begin{example}
Every hyperfield has a trivial homomorphism to the Krasner hyperfield, given by
\begin{align*}
    t:\HH &\rightarrow \KK \quad , \quad t(a) =
    \begin{cases}
        \1, \quad \text{if} \quad a\neq \0_{\HH}. \\
        \0, \quad \text{if} \quad a=\0_{\HH}.
    \end{cases}
\end{align*}
\end{example}
\begin{example}\label{ex:sign+hom}
Both $\RR$ and $\TT\RR$ have a homomorphism to the sign hyperfield that takes the positive elements to $\1$ and negative elements to $-\1$:
 \begin{align*}
 \sgn&: \RR \rightarrow \SS, & \sgn(a) &= 
 \begin{cases}
 \1, & \text{if} \quad  a \in \RR_{>0}. \\
 -\1, & \text{if} \quad a \in \RR_{<0}. \\
 \0, & \text{if} \quad a=0.
 \end{cases} \\
     \sgn&:\TT\RR \rightarrow \SS, & \sgn(a) &= 
     \begin{cases}
      \1, & \text{if} \quad a = (1,b). \\
     -\1, & \text{if} \quad a = (-1,b). \\
     \0, & \text{if} \quad a = -\infty.
     \end{cases}
 \end{align*}
 The existence of this homomorphism is due to $\RR$ and $\TT\RR$ being \emph{ordered hyperfields}.
 We shall see in Section~\ref{sec:ordered+hyperfields} that the existence of such a homomorphism characterises ordered hyperfields.
 \end{example}
 \begin{example}
 Given a quotient hyperfield $\HH=\FF/U$, there is a natural hyperfield homomorphism
 \begin{equation}
     \tau:\FF \rightarrow \HH \quad , \quad a \mapsto \bar{a} \, .
 \end{equation}
 \added{We call this map the \emph{quotient map}.}
 This map will be a key example of a hyperfield homomorphism throughout. 
\end{example}

\begin{definition}
The set of polynomials in $d$ variables over a hyperfield $\HH$ will be denoted $\HH[\bX] = \HH[X_1 \, , \dots , \, X_d]$, where elements of this set are defined as
\begin{equation}\label{hyperpolynomial}
\phi(X_1 \, , \dots , \, X_d) :=  \bigboxplus_{I \in \cI} c_I \odot X_1^{\odot i_1} \odot \, \cdots \, \odot X_d^{\odot i_d} = \bigboxplus_{I\in \cI} c_I \odot \bX^{\odot I},
\end{equation}
where $\cI \subseteq \ZZ_{\geq 0}^d$ finite and $c_I \in \HH$.
Each polynomial defines a function from $\HH^d$ to \added{the non-empty subsets of} $\HH$ given by evaluation.
\end{definition}
We will be primarily concerned with \emph{affine} and \emph{linear} polynomials; recall that a polynomial is affine if it is of degree one, and linear if it is affine with no constant term.

\begin{remark}
    The set of polynomials over a field $\FF[\bX]$ has the structure of a ring.
    Over a general hyperfield, we cannot imbue $\HH[\bX]$ with any additional algebraic structure: attempting to do so results in multivalued multiplication that is sometimes non-associative~\cite[Appendix A]{BL}.
\end{remark}

Let $f:\HH_1 \rightarrow \HH_2$ be a hyperfield homomorphism. 
This induces a map of polynomials $f_* : \HH_1[\bX] \rightarrow \HH_2[\bX]$ defined as
\[
\phi = \bigboxplus_1  c_I \odot_1 \bX^I \longmapsto f_*(\phi)  = \bigboxplus_2 f(c_I) \odot_2 \bX^I \, .
\]
We call $f_*(\phi)$ the \emph{push-forward} of $\phi$.
%The quotient map can be extended in this way and will be denoted $\tau_* : \FF[\bX] \rightarrow \FF/U[\bX]$.
By properties of hyperfield homomorphisms, we observe that
\begin{equation}\label{eq:pushforward+hom}
f(\phi(\bp)) = f\left(\bigboxplus_1  c_I \odot_1 \bp^I\right) \subseteq \bigboxplus_2 f(c_I) \odot_2 (f(\bp))^I = f_*(\phi)(f(\bp)) \, .
\end{equation}
This observation will be regularly exploited throughout.

Finally, we note that we will be considering geometric objects in the space $\HH^d$.
Its elements are $d$-tuples, that we denote in bold to avoid confusion.
One can extend the hyper-addition $\boxplus$ to tuples, and multiplication $\odot$ to scalars acting on tuples as follows:
\begin{align*}
\boxplus\colon\HH^d \times \HH^d &\rightarrow \cP(\HH^d)^* & \bp \boxplus \bq &= \bigcup_{r_i \in p_i \boxplus q_i} (r_1,\dots,r_d) \, , \\
\odot \colon \HH \times \HH^d &\rightarrow \HH^d & a \odot {\bp} &= (a\odot p_1, \dots, a \odot p_d) \, .
\end{align*}
Moreover, it has a zero `vector' $\underline{\0} = (\0, \dots, \0)$ that acts as an additive identity.
%Formally, $\HH^d$ has the structure of a vector space over a hyperfield, as defined in \cite{TSS}.

We also note that given a hyperfield homomorphism $f\colon\HH_1 \rightarrow \HH_2$, we can extend it to map $f\colon\HH_1^d \rightarrow \HH_2^d$ by applying it entrywise.
It is trivial to check that $f$ has the following desirable properties:
\[
f(\underline{\0}_1) = \underline{\0}_2 \quad , \quad f(\bp \boxplus \bq) \subseteq f(\bp) \boxplus f(\bq) \quad , \quad f(a \odot \bp) = f(a) \odot f(\bp) \, .
\]

\begin{example}
 Throughout the work, we will use the space $\SS^2$ as a domain to visualise examples.
 %These will include; convex sets, halfspaces and images of these objects from $\RR^2$.
 Figure~\ref{fig:S2} presents a schematic of $\SS^2$, with corresponding labels.
 To reduce the complication of the schematic for the remainder of the work, the labels are omitted.
 \begin{figure}
 \centering
     \includegraphics[]{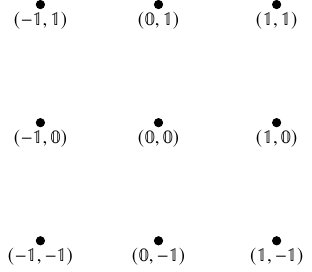}
     \caption{Schematic representation of the space $\SS^2$.}
     \label{fig:S2}
 \end{figure}
\end{example}

\begin{remark}\label{rem:vector+space}
While it is tempting to consider $\HH^d$ as a vector space over $\HH$, it does not satisfy distributivity of scalars.
Explicitly, for any $a,b \in \HH$ and $\bp \in \HH^d$, we have $(a \boxplus b) \odot \bp \subseteq a \odot \bp \boxplus b \odot \bp$, but they are not equal in general.
We refer to~\cite[Theorem 16]{MM} and the surrounding discussion for further details.
\end{remark}

\section{Ordered hyperfields}\label{sec:ordered+hyperfields}
In this section, we introduce orderings on hyperfields and give examples of the various classes of ordered hyperfields we shall be interested in.
%We end with a comparison between orderings and order relations on hyperfields.

\subsection{Orderings on hyperfields}

We begin by introducing orderings on hyperfields.
\added{We use a definition from ordered field theory and real algebraic geometry rather than in terms of binary relations.
We note that this definition is not a novel concept, and has already been used and studied in a number of places within the hyperfield literature~\cite{Mars,LAJD,KLS}.
A comparison between this definition and the potential binary relation definition is given in Appendix~\ref{sec:order+relations}.}

% Unlike many classical structures, hyperfields have a couple of distinct notions of an ordering.
% We recap the necessary ones here and compare them to existing notions.

\begin{definition}
An \emph{ordering} on a hyperfield $\HH$ is a subset $\HH^+ \subseteq \HH$ satisfying:
\begin{itemize}
    \item $\HH^+ \boxplus \HH^+ \subseteq \HH^+$,
    \item $\HH^+ \odot \HH^+ \subseteq \HH^+$,
    \item $\HH = \HH^+ \cup \{\0\} \cup \HH^-$ and $\HH^+ \cap \HH^- = \emptyset$, where $\HH^- = -\HH^+ = \{-a \mid a \in \HH^+\}$.
    %\item $\HH = \HH^+ \cup \{\0\} \cup \HH^-$ where $\HH^- = -\HH^+ = \{-a \mid a \in \HH^+\}$,
\end{itemize}
We let $\chi(\HH)$ denote the set of orderings on $\HH$.
We say that $\HH$ is \emph{real} if its set of orderings is non-empty, i.e. $\chi(\HH) \neq \emptyset$. %(See \cite{Mars} for further results on real hyperfields)
\added{An \emph{ordered hyperfield} is a real hyperfield $\HH$ with a distinguished ordering $\HH^+$.}
\end{definition}

One can think of the subset $\HH^+$ as the `positive' elements of $\HH$.
Note that this choice in not unique in general.
%From now on, we will use the term \emph{ordered hyperfield} to be a real hyperfield $\HH$ with a distinguished ordering $\HH^+$.
%There are discussions of equivalent notions of orderings for hyperfields in \cite{LinziStoj} and \cite{KLS} and for multirings in \cite{Mars}.
%By the properties of a hyperring homomorphism and $\HH^+ = f^{-1}(\1_\SS)$, it can be seen that $\1 \in \HH^+$ for any ordering.
%There are no self inverses over ordered hyperfields. Suppose that $x \in \HH$ is self inverse, then $0 \in x \boxplus x$. When taking the image under the morphism to $\SS$, this gives,
%$$\0 \in \mathrm{sgn}(x\boxplus x) \subseteq \mathrm{sgn}(x) \boxplus \mathrm{sgn}(x) = \mathrm{sgn}(x) \notni \0  .$$
%Demonstrating a contradiction. \\

\begin{example}\label{ex:field+ordering}
If $\HH = \RR$, then there is precisely one ordering, namely $\RR^+ = \RR_{> 0}$.
This also defines an order relation $\prec$ on $\RR$ given by $a \prec b$ if and only if $b-a \in \RR^+$.

More generally, if $\HH = \FF$ is a \emph{real} field, then its orderings $\FF^+$ are in bijection with order relations that are compatible with the field operations.
Explicitly, one can define an order relation given by $a \prec b$ if and only if $b - a \in \FF^+$: this is a strict total order that is compatible with the field operations:
\begin{align*}
    &\text{(Compatible with addition)} & a \prec b \, &\Rightarrow \, a + c \prec b + c \quad \forall c \in \FF \, , \\
    &\text{(Compatible with multiplication)} & a \succ 0 \, , \, b\succ 0 \, &\Rightarrow a\cdot b \succ 0 \, .
\end{align*}
Conversely, every ordering is defined as the set of elements $a \succ 0$ for some strict total order that is compatible with the field operations.
We show in Appendix~\ref{sec:order+relations} that defining order relations on hyperfields has many issues, as such we prefer to work with orderings. 
\end{example}

\begin{example}
The simplest example of a \emph{real} hyperfield is $\SS$, whose single ordering \emph{is} $\SS^+ = \{\1\}$. Note that $\SS^+$ is closed under addition and multiplication, and that $\SS$ decomposes into $\SS^+ \sqcup \SS^- \sqcup \{\0\}$.

As a non-example, consider $\SS^- = \{-\1\}$.
While closed under addition, it is not closed under multiplication and so does not form an ordering.
\end{example}

\begin{example}
The signed tropical hyperfield $\TT\RR$ has a single ordering, namely
\[
\TT\RR^+ = \left\{(1,a) \mid a \in \RR\right\} \, .
\]
%Consider $x, y \in \TT_\pm$ such that $x < y$.
%Then for any $z \in \TT_\pm$ such that $|z| > |x|,|y|$, we have
%\[
%(y \boxplus z) \boxplus (\ominus z \boxplus \ominus x) = z \boxplus \ominus z \nsubseteq \TT_+ \, ,
%\]
%as $x \boxplus z = y \boxplus z = z$.
\end{example}

\begin{example}
    None of the other hyperfields introduced so far have an ordering.
    As an example, suppose that the tropical hyperfield $\TT$ had some ordering $\TT^+$.
    For any $a \in \TT^+$, we would obtain $-\infty \in a \boxplus a$.
    This is a contradiction as $\TT^+$ must be closed under addition, but the zero element cannot be contained in $\TT^+$.
\end{example}

The next result describes how \emph{orderings} on a field $\FF$ are connected to \emph{orderings} on the quotient hyperfield $\FF/U$.

\begin{theorem}\label{quotientorder}\cite[Theorem 3.4]{KLS}
    Let $\HH = \FF/U$ be a quotient hyperfield \added{with corresponding quotient map $\tau\colon \FF \rightarrow \HH$}.
    Then $\FF^+$ is an ordering on $\FF$ with $U \subseteq \FF^+$ if and only if $\tau(\FF^+)$ is an ordering on $\HH$.
\end{theorem}

An immediate corollary of Theorem~\ref{quotientorder} is that an ordered quotient hyperfield can only arise as the quotient of an ordered field.
Moreover, the number of orderings on $\HH$ is equal to $|\chi(\FF|U)|$, the number of orderings on $\FF$ that contain $U$.
We can use this to count the number of orderings on quotient hyperfields, and construct hyperfields with a specific number of orderings.

\begin{example}\label{ex:uncountable}
The following is a quotient hyperfield with uncountably many orders, constructed as a quotient of a field with the same property.
Consider the field $\FF=\QQ(X)$, the field of rational functions over $\QQ$.
For any transcendental $\alpha \in \RR$, one can define an ordering $\FF^+_{\alpha} = \{f \, \mid f(\alpha) >0 \}$.
Furthermore, $\FF^+_{\alpha} = \FF^+_{\beta}$ if and only if $\alpha =\beta$.
This implies that $\chi(\QQ(X))$ is uncountable, as the transcendental numbers in $\RR$ are uncountable.

For $f \in \QQ(X)^\times$, let $\langle f \rangle = \{f^k \, \mid \, k \in \ZZ \}$ be the multiplicative subgroup generated by $f$.
If $f \in \FF^+_{\alpha}$, then $\langle f \rangle \subset\FF^+_{\alpha}$.
Consider the quotient hyperfield $\HH = \QQ(X)/ \langle f \rangle$.
As the set of transcendental $\alpha$ such that $f(\alpha) > 0$ is uncountable, it follows that $\chi(\QQ(X)\mid \langle f \rangle)$ is also uncountable, and that $\chi(\HH)$ is uncountable.
\end{example}

\begin{example}
\added{One should note that an ordering $\FF^+ \subseteq \FF$ will not give rise to an ordering $\HH^+ = \tau(\FF^+)$ on $\HH = \FF/U$ if $U \nsubseteq \FF^+$.}
%One should note that $\HH = \FF/U$ will not have an ordering if $U \nsubseteq \FF^+$.
As an example, consider $\FF = \RR$ and the multiplicative subgroup $U = \{1, -1\}$.
Observe that $\RR$ has a unique ordering $\RR_{>0}$, and $U$ is not contained in it.
The resulting quotient hyperfield $\HH = \RR/\{1, -1\}$ is not real; hyper-addition is given by $\bar{a} \boxplus \bar{b} = \{\overline{a+b}, \overline{a-b}\}$, and so each element is self-inverse.
This implies any possible ordering would satisfy $\HH^+ = \HH^-$, a contradiction.
\end{example}

We end this subsection by recalling the notion of an ordering preserving homomorphism between ordered hyperfields.

\begin{definition}
Let $\HH_1, \HH_2$ be ordered hyperfields.
A homomorphism $f\colon\HH_1 \rightarrow \HH_2$ \added{is} \emph{order preserving} if $f(\HH_1^+) \subseteq \HH_2^+$. %This will be referred to as \emph{order preserving}. 
%Hyperfield homomorphisms satisfy $f(-x) = -f(x)$, so this can be defined equivalently as $f(\HH_1^-) \subseteq \HH_2^-$.
\end{definition}

\begin{example}
Let $\HH = \FF/U$ be an ordered quotient hyperfield.
As a corollary of Theorem~\ref{quotientorder}, we observe that the quotient map $\tau\colon \FF \rightarrow \HH$ is an order preserving homomorphism.
\end{example}
\begin{example}
Extending Example~\ref{ex:sign+hom}, for any ordered hyperfield there exists an order preserving homomorphism to the sign hyperfield
\[
\sgn \colon \HH \rightarrow \SS \quad , \quad \sgn(a) = \begin{cases}
    \1 & a \in \HH^+, \\
    -\1 & a \in \HH^-, \\
    \0 & a = \0_\HH,
\end{cases} \, 
\]
\added{The following lemma shows that converse is also true.}
%: a hyperfield is real if and only if there exists a homomorphism $f\colon\HH\rightarrow \SS$, where $\HH^+ = f^{-1}(\1_\SS)$~\cite[Remark 3.2]{KLS}.
%This gives an alternative characterisation of real hyperfields as those with a homomorphism to the sign hyperfield.
\end{example}

\added{
\begin{lemma}{\cite[Remark 3.2]{KLS}} \label{lem:ordered+hom}
    Let $\HH$ be a hyperfield.
    $\HH^+$ is an ordering on $\HH$ if and only if there exists a hyperfield homomorphism $f \colon \HH \rightarrow \SS$ with $\HH^+ = f^{-1}(\1)$.
\end{lemma}
}

\subsection{Classes of ordered hyperfields} \label{sec:class-ord-strngt}
In this subsection, we prove the inclusion of classes of ordered hyperfields given in Figure~\ref{fig:classes}.
Along the way, we will classify ordered stringent hyperfields, and give numerous examples of the pathological behaviour that can arise.

\begin{figure}
\[
\{\text{ord. fields}\} \, \subsetneq \, \{\text{stringent ord. hyperfields}\} \, \subsetneq \, \{\text{quotient ord. hyperfields}\} \, \subseteq \, \{\text{ord. hyperfields}\}
\]
\caption{The classes of ordered hyperfields and their inclusion relations}
\label{fig:classes}
\end{figure}
Bowler and Su~\cite{BS} present a complete classification of stringent (skew) hyperfields.
Their classification utilizes a construction in which one `extends' a hyperfield by a totally ordered abelian group.
This construction had already been utilised to study tropical extensions of semirings~\cite{AGG:09,AGG:14}, and more recently tropical extensions of idylls~\cite{TG:22}.
We recall this construction for hyperfields as follows.

Let $(\HH,\boxplus_\HH,\odot_\HH)$ be a hyperfield and $(G,+,0)$ a totally ordered abelian group.
We define $\HH \rtimes G$ to be the hyperfield with ground set $(\HH^\times \times G) \cup \{\0\}$, where $\0$ is some new element acting as the additive identity.
We define the multiplication by
\[
(a,g) \odot (b,h) = (a\odot_\HH b, g+ h) \quad , \quad (a,g) \odot \0 = \0 \odot (a,g) = \0 \, ,
\]
where $\1 := (\1_\HH, 0)$ is the multiplicative identity.
The addition is defined in a more involved way as follows:
\[
(a,g) \boxplus (b,h) = 
\begin{cases}
(a,g), & g > h \\
(b,h), & g < h \\
\{(c,g) \mid c \in (a \boxplus_\HH b) \cap \HH^\times\}, & g = h \, , \, \0_\HH \notin a \boxplus_\HH b  \\
\{(c,g) \mid c \in (a \boxplus_\HH b) \cap \HH^\times\} \cup \{(a, g') \mid a \in \HH^\times \, , \, g' < g\} \cup \{\0\}, & g = h \, , \, \0_\HH \in a \boxplus_\HH b 
\end{cases}
\]
One can check that $\HH \rtimes G$ is a hyperfield, and that it is stringent if and only if $\HH$ is stringent also.
Furthermore, the following classification theorem shows all stringent hyperfields can be written this way.

\begin{theorem}\label{thm: stringent+classification}\cite[Theorem 4.10]{BS}
Every stringent hyperfield has the form $\HH \rtimes G$, where $\HH$ is either the Krasner hyperfield $\KK$, the sign hyperfield $\SS$ or a field $\FF$.
\end{theorem}

We use this classification of stringent hyperfields to completely classify ordered stringent hyperfields.
This classification will be useful for a number of technical lemmas throughout, where properties of stringent hyperfields are proved more directly via case analysis.

\begin{proposition}\label{prop:ord+stringent+class}
Every ordered stringent hyperfield has the form $\HH \rtimes G$, where $\HH$ is either the sign hyperfield $\SS$ or an ordered field $\FF$.
\end{proposition}
\begin{proof}
We show that $\HH \rtimes G$ has an ordering if and only if $\HH$ has an ordering.
As the sign hyperfield and ordered fields are the only possibilities for $\HH$ that are ordered, this completes the proof.

It is easy to check that if $\HH^+$ is an ordering of $\HH$, then $\{(a, g) \mid a \in \HH^+ \, , \, g \in G \}$ satisfies the properties of an ordering on $\HH \rtimes G$.
Conversely, suppose $\HH \rtimes G$ has an ordering: \added{by Lemma~\ref{lem:ordered+hom} we consider this as a hyperfield homomorphism $\varphi \colon \HH \rtimes G \rightarrow \SS$.}
There is an injective homomorphism $i \colon \HH \rightarrow \HH \rtimes G$ that sends $\0_\HH \mapsto \0$ and all nonzero elements $a \mapsto (a,0)$.
Therefore the composition $\varphi \circ i$ is a hyperfield homomorphism from $\HH$ to $\SS$ and so $\HH$ is also ordered.
\end{proof}

\begin{example}
The signed tropical hyperfield $\TT\RR$ is ordered and stringent, and so is covered by this classification.
Explicitly it has the form $\SS \rtimes \RR$ where $(\RR,+)$ is viewed as a ordered additive abelian group.

More generally, we can view any hyperfield of the form $\SS \rtimes G$ as a signed tropical hyperfield of higher rank.
By Hahn's embedding theorem, any ordered abelian group $G$ can be embedded inside $(\RR^k,+)$ ordered via the lexicographical ordering, where the smallest such $k$ is the \emph{rank} of $G$.
Consequently, we can embed $\SS \rtimes G$ inside the \emph{rank k} signed tropical hyperfield $\TT\RR^{k} = \SS \rtimes \RR^k$.
Tropical geometry and convexity of higher rank has been of recent interest~\cite{Foster+Ranganathan:16,Joswig+Smith:18,AminiIriarte:2022}, and so this is a natural hyperfield to study.
\end{example}

% \begin{remark}
%     In \cite{BS}, the classification of stringent (skew) hyperfields is first done by introducing a short exact sequence construction.
%     This viewpoint is further developed in \cite{TG:22} where the notion of a stringent hyperfield is generalised to that of a \emph{tropical extension} is the category of idylls.
% \end{remark}

It was also shown in~\cite{BS} that every (skew) stringent hyperfield is the quotient of a (skew) field.
We recall their result specialising to ordered stringent hyperfields; note that removing the skew condition simplifies the construction considerably.

Let $\FF[[t^G]]$ be the field of \emph{Hahn series} with value group $G$ and coefficients in some arbitrary field $\FF$.
This is the field of formal power series 
\[
\gamma = \sum_{g \in G} c_gt^g \, , \, c_g \in \FF
\]
whose support $\supp(\gamma) = \{g \in G \mid c_g \neq 0\}$ is well-ordered and has a well-defined maximum.
The \emph{leading term} $\lt(\gamma)$ of $\gamma$ is the term with largest exponent,
\[
\lt(\gamma) = \left\{c_gt^g \mid g \geq g' \, \forall g' \in \supp(\gamma) \right\} \, .
\]
The \emph{leading coefficient} $\lc(\gamma)$ is the coefficient of the leading term, and the \emph{leading power} $\lexp(\gamma)$ is the exponent of the leading term.
Addition and multiplication are the usual operations on power series:
\begin{align*}
\left(\sum_{g \in G} c_gt^g\right) + \left(\sum_{g \in G} d_gt^g\right) &= \sum_{g \in G} (c_g+d_g)t^g \\
\left(\sum_{g \in G} c_gt^g\right) \cdot \left(\sum_{g \in G} d_gt^g\right) &= \sum_{g \in G} \left(\sum_{\substack {h,h' \in G \\ h +_{G} h' = g}} c_h\cdot d_{h'}\right)t^g \\
\end{align*}
If $\FF$ is an ordered field, we can define an ordering on $\FF[[t^G]]$ given by power series whose leading coefficient is positive, i.e.
\[
\FF[[t^G]]^+ := \{\gamma \mid \lc(\gamma) \in \FF^+\} \, .
\]

\begin{theorem}\cite[Theorem 7.5]{BS}\label{thm:stringent+quotient}
Every stringent hyperfield can be realised as a quotient of a field. Explicitly,
\begin{align*}
\KK \rtimes G &\simeq \FF[[t^G]]/U & U &= \left\{\gamma \mid \lt(\gamma) = c_0 \in \FF \right\} \\
\SS \rtimes G &\simeq \FF[[t^G]]/V  & V &= \left\{\gamma \mid \lt(\gamma) = c_0 \in \FF^+ \right\} \, , \\
\FF \rtimes G &\simeq \FF[[t^G]]/W & W &= \left\{\gamma \mid \lt(\gamma) = c_0 = 1 \right\}  \, .
\end{align*}
\end{theorem}
% \ben{How many orders does each stringent hyperfield have?} \james{The structure of the proof to prop.3.12 gives that the semidirect product has a order if H has an order, so given an ordered field with more than one order, you can construct a stringent hyperfield with non-unique order. Maybe this would be best placed as a remark under prop 3.12? Let me know if this is precise and if you want me to do it. }
% \ben{I think it will be more work than worth it: its quite easy to show $\SS \rtimes G$ and $\FF \rtimes G$ have at least as many orders as $\SS$ or $\FF$, but precisely equal sounds quite a bit of effort}

\begin{example} \label{ex:tropical+quotients}
The tropical hyperfield $\TT = \KK\rtimes \RR$ can be realised as a quotient of the field $\FF[[t^{\RR}]]$, also known as the \emph{generalised Puiseux series}~\cite{Markwig:07}.
Usually in tropical geometry, one considers the image of objects in the valuation map
\begin{equation}\label{eq:val}
\val\colon \FF[[t^{\RR}]] \rightarrow \TT \quad , \quad \gamma \mapsto \lexp(\gamma) \, .
\end{equation}
We can connect this perspective to Theorem~\ref{thm:stringent+quotient} by observing that $U = \val^{-1}(\1_\TT) = \val^{-1}(0)$.

The signed tropical hyperfield $\TT\RR = \SS\rtimes \RR$ can be realised as a quotient of the field $\FF[[t^{\RR}]]$ where $\FF$ is an ordered field.
The valuation map~\eqref{eq:val} can be enriched to the \emph{signed valuation}, which records the sign of a Hahn series:
\begin{equation}\label{eq:sval}
\sval\colon \FF[[t^{\RR}]] \rightarrow \TT\RR \quad , \quad \gamma \mapsto (\sgn(\lc(\gamma)), \lexp(\gamma)) \, .
\end{equation}
%\james{should it be sgn(lc(gamma))above??}
%\ben{Good shout, changed}
Studying objects in the image of this map is referred to as \emph{real tropicalization}~\cite{JSY}.
Again, we can connect this perspective to Theorem~\ref{thm:stringent+quotient} by observing that $V = \sval^{-1}(\1_{\TT\RR}) = \sval^{-1}((1,0))$.

Extrapolating the previous two examples suggests the introduction of the \emph{fine valuation}, a further enrichment of the valuation map that recalls the whole leading term:
\begin{equation}\label{eq:fval}
\fval\colon \FF[[t^{\RR}]] \rightarrow \FF\rtimes\RR \quad , \quad \gamma \mapsto (\lc(\gamma), \lexp(\gamma)) \, .
\end{equation}
This map is defined so that $W = \fval^{-1}(\1_{\FF\rtimes\RR}) = \fval^{-1}((1,0))$.
To our knowledge, this fine valuation has not been studied within tropical geometry.
\end{example}

\begin{example}
Consider the hyperfield $\RR \rtimes \ZZ$, a schematic viewpoint is given in Figure~\ref{fig:RxZ}.
Displayed are the elements
\[
a = (1,1) \, , \, b = (3,0) \, , \, c = (-1,0) \, .
\]
We see that $a \boxplus b = a \boxplus c = a$, as $a$ has a strictly larger value in the group $G = \ZZ$.
However, both $b$ and $c$ have the same value in $\ZZ$, and so their sum is $b \boxplus c = (3 + (-1), 0) = (2,0)$.
If we consider a point $d = (k,0)$ and its additive inverse $-d = (-k,0)$, their sum is the set
\[
d \boxplus -d = \SetOf{(\ell, g)}{g \in \ZZ_{<0}} \, .
\]

As discussed in Example~\ref{ex:tropical+quotients}, we can view $\RR \rtimes \ZZ$ as $\RR[[t^\ZZ]]$ in which we only remember the leading term.
When we take the sum of two series $\gamma, \gamma'$, the leading term $\lt(\gamma + \gamma')$ is completely determined by $\lt(\gamma)$ and $\lt(\gamma')$ if and only if $\lt(\gamma') \neq -\lt(\gamma)$.
If the leading terms cancel, $\lt(\gamma + \gamma')$ could be any term whose exponent is smaller. 
As the hyperfield does not see this information about $\gamma$ and $\gamma'$, the output is the set of all possibilities.
The same viewpoint can be taken for $\TT$ and $\TT\RR$: in $\TT$ the hyperfield only sees the leading power of $\gamma$, while in $\TT\RR$ the hyperfield only sees the sign and the leading power of $\gamma$.
% More concretely, take $\gamma = c_gt^g + \cdots $ and  $\gamma' = -c_gt^g + \cdots$, then $\gamma +\gamma' = (c_g - c_g)t^g + \cdots$.
% The leading terms cancel and $\lt(\gamma +\gamma')$ depends on the lower order terms of both $\gamma$ and $\gamma'$, thus could but any term where the exponent it strictly less than $g$.
% Although, if the leading terms do not cancel, say $\gamma = c_gt^g + \cdots $ and $\gamma'' = c_{g+1}t^{g+1} + \cdots$ for example, then $\lt(\gamma + \gamma'') = \lt(\gamma)$ is exactly determined.  
\end{example}

\begin{figure}
    \centering
    \includegraphics{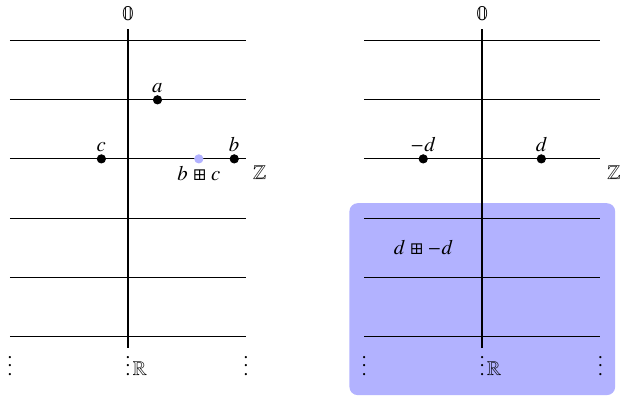}
    \caption{A schematic representation of $\RR \rtimes \ZZ$, with examples of the hyper-addition.}
    \label{fig:RxZ}
\end{figure}

Theorem~\ref{thm: stringent+classification} gives examples of ordered stringent hyperfields that are not ordered fields.
Moreover, we have seen examples of ordered quotient hyperfields that do not arise as ordered stringent hyperfields: see Example~\ref{ex:uncountable} and the upcoming Example~\ref{ex:finite+ordered}.
This gives the inclusion of classes in Figure~\ref{fig:classes}.
However, not all of these inclusions are strict: we do not know of any ordered hyperfields that cannot arise as ordered quotient hyperfields.
There does exist examples of hyperfields that are not quotient hyperfields (see~\cite{MAS,Baker+Jin:21}), but to our knowledge none of them are real.

\begin{question}
    Does there exist an ordered hyperfield that cannot be realised as the quotient of an ordered field?
\end{question}

We end this section by noting that certain standard behaviour of ordered fields need not hold for ordered hyperfields.
An ordered field must be infinite, however the sign hyperfield proves that this need not be the case for ordered hyperfields.
The following example shows that this is not the only one, by constructing an infinite family of ordered hyperfields of finite size.

\begin{example} \label{ex:finite+ordered}
    For proofs of the claimed properties in the following, see~\cite[Prop 8.3.9]{Lam:73}.
    Consider the field of real formal Laurent series $\RR((t))$ and its set of square elements $S$
    \[
    \RR((t)) := \SetOf{\sum_{i = m}^\infty a_i t^i}{m \in \ZZ \, , \, a_i \in \RR} \quad , \quad S = \SetOf{\sum_{i = m}^\infty a_i t^i}{m \in 2\ZZ\, , \, a_m >0} \, .
    \]
    The squares $S$ are closed under addition (as $\RR((t))$ is a \emph{Pythagorean field}) and forms a multiplicative subgroup of $\RR((t))^\times$.
    Moreover, $S$ forms an index four subgroup of $\RR((t))^\times$, as we can decompose $\RR((t))$ into the cosets
    \[
    \RR((t)) = S \sqcup t\cdot S \sqcup -S \sqcup -t\cdot S \sqcup 0\cdot S \, .
    \]
    Using these facts, we can deduce that the quotient hyperfield $\HH = \RR((t))/S = \{\0, \1, -\1, \bar{t}, -\bar{t}\}$ has additive structure:
\begin{equation} \label{eq:addition+table}
\begin{tabular}{|c|c|c|c|c|c|}
 \hline
  $\boxplus$ & $\0$ & $\1$ & $\bar{t}$ & $-\bar{t}$ & $-\1$\\
  \hline
  $\0$ & $\{\0\}$ & $\{\1\}$ & $\{\bar{t}\}$ & $\{-\bar{t}\}$ & $\{-\1\}$\\
  \hline
  $\1$ & $\{\1\}$ & $\{\1\}$ & $\{\1,\bar{t}\}$ & $\{\1,-\bar{t}\}$ & $\HH$\\
  \hline
  $\bar{t}$ & $\{\bar{t}\}$  & $\{\1,\bar{t}\}$ & $\{\bar{t}\}$ & $\HH$ & $\{\bar{t},-\1\}$\\
  \hline
  $-\bar{t}$ & $\{-\bar{t}\}$ & $\{\1,-\bar{t}\}$ & $\HH$ & $\{-\bar{t}\}$ & $\{-\bar{t},-\1\}$\\
  \hline
  $-\1$ & $\{-\1\}$ & $\HH$ & $\{\bar{t},-\1\}$ & $\{-\bar{t},-\1\}$ & $\{-\1\}$\\
  \hline
\end{tabular}
\end{equation}
and multiplicative structure $\HH^\times \cong C_2 \times C_2$ generated by $\bar{t}$ and $-\1$.

    As $\RR$ is a real field with a unique ordering, $\RR((t))$ is a real field with exactly two orderings: one where $t > 0$ and one where $t < 0$.
    These orderings are inherited by $\HH$ in the following way:
    \[
    \RR((t))^+ = \begin{cases}
        S \sqcup t\cdot S & t > 0 \\
        S \sqcup -t\cdot S & t < 0 \\
    \end{cases}
    \quad \Rightarrow \quad 
    \HH^+ = \begin{cases}
    \{\1, \bar{t}\}  & t > 0\\
    \{\1, -\bar{t}\} & t < 0\\
    \end{cases}
    \]
    This shows that $\chi(\HH) = \chi(\RR((t))) = 2$.
    Moreover, one can check from the classification in~\cite{ZLfin} that $\HH$ is the only real hyperfield of order five (it is the hyperfield $\SS^{\Uparrow 5}$ in their notation).

    This construction can be generalised to $\FF_n := \RR((t_1))\dots((t_n))$.
    Let $S_n$ be the set of squares in $\FF_n$; as above $S_n$ is a multiplicative subgroup of index $2^{n+1}$.
    This gives a quotient hyperfield 
    \[
    \HH_n = \FF_n/S_n = \left\{\pm\bar{t}_1^{k_1}\cdots\bar{t}_n^{k_n} \mid k_i \in \{0,1\}\right\} \cup \{\0\}
    \]
    of order $2^{n+1} + 1$.
    Moreover, these hyperfields are real and have exactly $2^n$ orders, as $\FF_n$ has twice as many orders as $\FF_{n-1}$ determined by whether $t_n > 0$ or $t_n < 0$.
    This gives an infinite family of ordered hyperfields of finite order.
\end{example}

\section{Convex and conic sets over hyperfields} \label{sec:convex+sets}

In this section, we introduce conic and convex sets over ordered hyperfields.
Throughout, we shall assume all of our hyperfields are ordered.

\subsection{Basics of convex and conic sets}

\begin{definition}
A subset $S \subseteq \HH^d$ is \emph{conic} if it satisfies
\begin{equation} \label{eq:conic}
a \odot \bp \boxplus b \odot \bq \subseteq S \, ,
\end{equation}
for all $\bp, \bq \in S$ and $a,b \in \HH^+$.
Given a set $T \subseteq \HH^d$, its \emph{conic hull} $\cone(T)$ is the (unique) minimal conic set containing $T$.
\end{definition}
One can also define cones to be closed under non-negative scalars $\HH^+ \cup\{\0\}$ rather than just positive scalars.
This will ensure the vector $\underline{\0}$ is always contained in a conic set.
It will be convenient for us later to not enforce this, as we will have conic sets that we may not want to contain $\underline{\0}$.

Given a finite multiset $\{\bp_1, \dots, \bp_k\}\subseteq \HH^d$ of %(not necessarily distinct) 
points, a \emph{conic combination} is an expression of the form
\[
\bigboxplus_{i=1}^k a_i \odot \bp_i \subseteq \HH^d \quad , \quad a_i \in \HH^+ \, .
\]
Note that unlike over fields, we must allow for repetition in the set of points we are taking a combination of, as $a_1 \odot \bp \boxplus a_2 \odot \bp = (a_1 \boxplus a_2) \odot \bp$ where $a_1 \boxplus a_2$ may be a subset of $\HH^+$ rather than an element.
If $\HH$ is stringent, then this is not an issue as $a_1 \boxplus a_2$ is always a singleton.

Given a set $T \subseteq \HH^d$, the following lemma shows a more concrete characterisation of $\cone(T)$.
The proof is a simpler version of the proof of Lemma~\ref{lem:conv+comb}, and so we defer it until then.

\begin{lemma}\label{lem:conic+comb}
The conic hull of $T$ is equal to the set of all finite conic combinations of points of $T$, i.e.
\[
\cone(T) = \SetOf{\bq \in \bigboxplus_{i=1}^k a_i \odot \bp_k }{ a_i \in \HH^+ \, , \, \{\bp_1, \dots, \bp_k\} \text{ finite multiset of } T }
\]
\end{lemma}

\begin{definition}
A subset $S \subseteq \HH^d$ is \emph{convex} if it satisfies
\begin{equation} \label{eq:convex}
a \odot \bp \boxplus b \odot \bq \subseteq S \, ,
\end{equation}
for all $\bp, \bq \in S$ and $a,b \in \HH^+$ such that $\1 \in a \boxplus b$.
Given a set $T \subseteq \HH^d$, the \emph{convex hull} $\conv(T)$ of $T$ is the (unique) minimal convex set containing $T$.
\end{definition}
Given a finite multiset $\{\bp_1, \dots, \bp_k\}\subseteq \HH^d$ of (not necessarily distinct) points, a \emph{convex combination} is an expression of the form
\[
\bigboxplus_{i=1}^k a_i \odot \bp_i \subseteq \HH^d \quad , \quad a_i \in \HH^+ \, , \, \1 \in \bigboxplus_{i=1}^k a_i \, .
\]

\begin{lemma}\label{lem:conv+comb}
The convex hull of $T$ is equal to the set of all finite convex combinations of points of $T$, i.e.
\[
\conv(T) = \SetOf{\bq \in \bigboxplus_{i=1}^k a_i \odot \bp_k }{ a_i \in \HH^+ \, , \, \1 \in \bigboxplus_{i=1}^k a_i \, , \, \{\bp_1, \dots, \bp_k\} \text{ finite multiset of } T }
\]
\end{lemma}
\begin{proof}
We begin by showing that $\conv(T)$ must contain all finite convex combinations of elements of $T$.
Let $\{\bp_1, \dots, \bp_k\} \subseteq T$ be a finite set of points, possibly with repetition.
We claim via induction on $k$ that $\bigboxplus_{i=1}^k a_i \odot \bp_i \subseteq \conv(T)$, where $a_i \in \HH^+$ and $\1 \in \bigboxplus_{i=1}^k a_i$.
For the base case $k=2$, this holds by the definition of convex sets.

Assume that the claim holds for $k-1$.
As
\[
\1 \in a_1 \boxplus \cdots \boxplus a_k = \bigcup_{a \in a_1 \boxplus \cdots \boxplus a_{k-1}} a \boxplus a_k \, ,
\]
there exists some $b \in a_1 \boxplus \cdots \boxplus a_{k-1}$ such that $\1 \in b \boxplus a_k$.
Furthermore, we note that $b\in \HH^+$ as $\HH^+$ is closed under hyper-addition and multiplication, and so $b^{-1} \in \HH^+$ exists.
Then
\begin{align} \label{eq:conv+ind+step}
    \bigboxplus_{i=1}^k a_i \odot \bp_i = b \odot \left(\bigboxplus_{i=1}^{k-1} a_i \odot b^{-1} \odot \bp_i\right) \boxplus (a_k \odot \bp_k) \, ,
\end{align}
where $a_i \odot b^{-1} \in \HH^+$ for all $i$.
Furthermore
\[
\1 = b \odot b^{-1} \in \left(\bigboxplus_{i=1}^{k-1} a_i\right) \odot b^{-1} = \bigboxplus_{i=1}^{k-1} a_i b^{-1} \, ,
\]
and so by the induction hypothesis, we have $\bigboxplus_{i=1}^{k-1} a_i \odot b^{-1} \odot \bp_i \subseteq \conv(T)$.
By definition of convex sets, equation~\eqref{eq:conv+ind+step} must also be in $\conv(T)$.

Conversely, we show that the set of finite convex combinations of $T$ forms a convex set.
Let $\bp' \in \bigboxplus_{i=1}^k a_i \odot \bp_i$ and $\bq' \in \bigboxplus_{j=1}^\ell b_j \odot \bq_j$ be finite convex combinations of points $\bp_i, \bq_j \in T$.
Then for any $\alpha, \beta \in \HH^+$, we have
\begin{align*}
    \alpha \odot \bp' \boxplus \beta \odot \bq' &\in \alpha \odot \left(\bigboxplus_{i=1}^k a_i \odot \bp_i\right) \boxplus \beta \odot \left(\bigboxplus_{j=1}^\ell b_i \odot \bq_j\right) \\
    &= \left(\bigboxplus_{i=1}^k \alpha \odot a_i \odot \bp_i\right) \boxplus \left(\bigboxplus_{j=1}^\ell\beta \odot b_i \odot \bq_j\right) \, .
\end{align*}
Note that as $\HH^+$ is closed under multiplication, each $\alpha \odot a_i, \beta \odot b_j \in \HH^+$.
Moreover, if $\1 \in \alpha \boxplus \beta$, we also have
\begin{align}
    \1 \in \alpha \boxplus \beta \subseteq \alpha \odot \left(\bigboxplus_{i=1}^k a_i\right) \boxplus \beta \odot \left(\bigboxplus_{j=1}^\ell b_i\right) = \left(\bigboxplus_{i=1}^k \alpha \odot a_i\right) \boxplus \left(\bigboxplus_{j=1}^\ell  \beta \odot b_i\right) \, .
\end{align}
Therefore, the set of finite convex combinations of elements of $T$ is itself a convex set.
By minimality of $\conv(T)$, this completes the proof.
\end{proof}

Classic convex and conic sets over fields are intimately related, as each convex set in $d$-dimensional space can be described by a conic set in $(d+1)$-dimensional space.
The following lemma shows the same correspondence holds over hyperfields.
Given a set $T \subseteq \HH^d$, we define its \emph{homogenisation} as
\[
\widetilde{T} = \SetOf{(\bp, \1)\in \HH^{d+1}}{\bp \in T} \, .
\]
\begin{lemma}\label{lem:homog}
    Let $T \subseteq \HH^d$.
    The point $\bq \in \HH^d$ is contained in $\conv(T)$ if and only if the point $(\bq,\1) \in \HH^{d+1}$ is contained in $\cone(\widetilde{T})$.
\end{lemma}
\begin{proof}
    Let $\bq \in \conv(T)$.
    By Lemma~\ref{lem:conv+comb} we can write it as a finite convex combination of points in $T$,
    \begin{align*}
    \bq \in \bigboxplus_{i=1}^k a_i \odot \bp_i \quad , \quad \1 \in \bigboxplus_{i=1}^k a_i \quad \Rightarrow \quad  (\bq, \1) \in \bigboxplus_{i=1}^k a_i \odot (\bp_i, \1)
    \end{align*}
    immediately implying that $(\bq, \1) \in \conv(\widetilde{T}) \subseteq \cone(\widetilde{T})$.

    Conversely, let $(\bq, \1) \in \cone(\widetilde{T})$, by Lemma~\ref{lem:conic+comb} we can write it as a finite conic combination of points in $\widetilde{T}$,
    \begin{align*}
        (\bq, \1) \in \bigboxplus_{i=1}^k a_i \odot (\bp_i, \1) \, , \, a_i \in \HH^+ \, .
    \end{align*}
    Considering the final coordinate shows that $\1 \in \bigboxplus_{i=1}^k a_i$, implying that this must also be a convex combination.
    Restricting to the first $d$ coordinates gives us $\bq \in \conv(T)$.
\end{proof}
As an immediate corollary, one can take a conic set in $\HH^{d+1}$ and intersect it with the ordinary hyperplane $\{\bq \mid q_{d+1} = \1\}$ to obtain a convex set in $\HH^d$.
This will be useful for showing properties of conic sets pass to convex sets.

We now give some examples of convex sets over various hyperfields.
For explicit examples of convex sets over $\TT\RR$ we refer the reader to~\cite{LV}.

\begin{example}
    Figure~\ref{fig:S+Convex+Set} presents six examples of convex sets in $\SS^2$, using the same coordinates as Figure~\ref{fig:S2}.
    Note that convex and conic sets are equivalent over $\SS$ as all sums of positive elements contain $\1$.
    \begin{figure}[ht]
     \centering
     \begin{subfigure}[b]{0.25\textwidth}
         \centering
         \includegraphics[width=\textwidth]{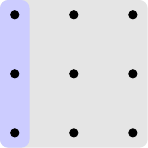}
         \caption{}
         \label{fig:S+Convex-1}
     \end{subfigure}
     \hfill
     \begin{subfigure}[b]{0.25\textwidth}
         \centering
         \includegraphics[width=\textwidth]{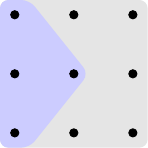}
         \caption{}
         \label{fig:S+Convex-2.pdf}
     \end{subfigure}
     \hfill
     \begin{subfigure}[b]{0.25\textwidth}
         \centering
         \includegraphics[width=\textwidth]{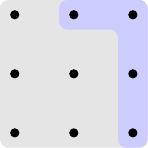}
         \caption{}
         \label{fig:S+Convex-3.pdf}
     \end{subfigure}
     \hfill
     \begin{subfigure}[b]{0.25\textwidth}
         \centering
         \includegraphics[width=\textwidth]{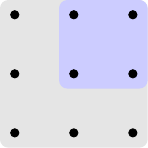}
         \caption{}
         \label{fig:S+Convex-4.pdf}
     \end{subfigure}
     \hfill
     \begin{subfigure}[b]{0.25\textwidth}
         \centering
         \includegraphics[width=\textwidth]{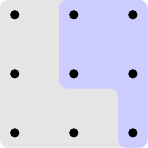}
         \caption{}
         \label{fig:S+Convex-5.pdf}
     \end{subfigure}
     \hfill
     \begin{subfigure}[b]{0.25\textwidth}
         \centering
         \includegraphics[width=\textwidth]{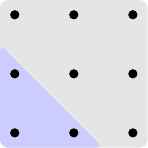}
         \caption{}
         \label{fig:S+Convex-6.pdf}
     \end{subfigure}
        \caption{Six examples of convex sets in blue over $\SS^2$.}
        \label{fig:S+Convex+Set}
\end{figure}
\end{example}

    \begin{figure}[ht]
        \centering
        \includegraphics[width=0.45\textwidth]{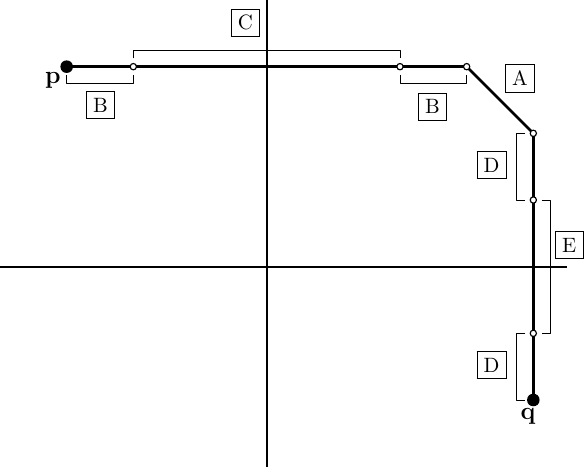} \hfill
        \includegraphics[width=0.45\textwidth]{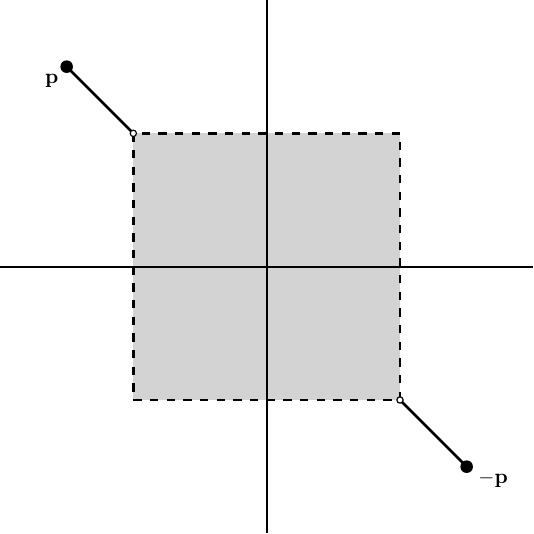}
        \caption{Convex hull of $\bp$ and $\bq$, and $\bp$ and $-\bp$ from Example~\ref{ex:RxZ2}.}
        \label{fig:RxZ2}
    \end{figure}

\begin{example}\label{ex:RxZ2}
    Let $\HH = \RR \rtimes \ZZ$ with unique ordering $\HH^+ = \RR_{>0} \times \ZZ$.
    Recall that $\1 = (1,0)$; we observe that if $a, b \in \HH^+$ such that $\1 \in a \boxplus b$, the possibilities for $a,b$ are
    \begin{align*}
    \begin{cases}
    a = (k, 0) \, , \, b = (1 - k, 0) & k \in (0,1) \subseteq  \RR_{>0} \, , \\
    a = (1, 0) \, , \, b = (k, g) & k \in \RR_{>0} \, , \, g \in \ZZ_{<0} \, , \\
    a = (k, g) \, , \, b = (1,0) & k \in \RR_{>0} \, , \, g \in \ZZ_{<0} \, .
    \end{cases}
    \end{align*}
    We use this to compute some convex sets in $\HH^2$.
    Consider the points $\bp = [(-1, 0) , (1, 0)]$ and $\bq = [(3, 1), (-2, -1)]$.
    A straightforward calculation shows that the possible values for $a \odot\bp \boxplus b \odot \bq$ are
\begin{enumerate}[label=\fbox{\Alph*}]
    \item{\quad\makebox[6cm][l]{$\begin{pmatrix} (3-3k, 1) \\ (k, 0) \end{pmatrix}$}  $a = (k, 0) \, , \, b = (1 - k, 0)  \, , \, k \in (0,1)$} 
    \item{\quad\makebox[6cm][l]{$\begin{pmatrix} (3k-1, 0) \\ (1, 0) \end{pmatrix}$} $a = (1, 0) \, , \, b = (k, -1) \, , \, k \neq 1/3$}
    \item{\quad\makebox[6cm][l]{$\SetOf{\begin{pmatrix} (k', g') \\ (1, 0) \end{pmatrix}}{g' < -1, k' \in \RR_{>0}}$} $a = (1, 0) \, , \, b = (1/3,-1)$}
    \item{\quad\makebox[6cm][l]{$\begin{pmatrix} (3, 1) \\ (k-2, -1) \end{pmatrix}$} $a = (k, -1) \, , \, b = (1,0) \, , \, k \neq 2$}
    \item{\quad\makebox[6cm][l]{$\SetOf{\begin{pmatrix} (3, 1) \\ (k', g') \end{pmatrix}}{g' < -1, k' \in \RR_{>0}}$}  $a = (2, -1) \, , \, b = (1,0)$}
\end{enumerate}
    with any other choices for $a, b$ returning either $\bp$ or $\bq$.
    % \begin{align*}
    %     (k, 0) \odot \begin{pmatrix} (-1, 0) \\ (1, 0) \end{pmatrix} \boxplus (1-k,0) \odot \begin{pmatrix} (3, 1) \\ (-2, -1) \end{pmatrix} &= \begin{pmatrix} (3-3k, 1) \\ (k, 0) \end{pmatrix} \\
    %     (1, 0) \odot \begin{pmatrix} (-1, 0) \\ (1, 0) \end{pmatrix} \boxplus (k,g) \odot \begin{pmatrix} (3, 1) \\ (-2, -1) \end{pmatrix} &=
    %     \begin{cases}
    %         \begin{pmatrix} (3k-1, 0) \\ (1, 0) \end{pmatrix} & g = -1, k \neq 1/3 \\
    %         \SetOf{\begin{pmatrix} (k', g') \\ (1, 0) \end{pmatrix}}{g' < -1, k' \in \RR_{>0}} & g = -1, k \neq 1/3 \\
    %         \begin{pmatrix} (-1, 0) \\ (1, 0) \end{pmatrix} & \text{otherwise}
    %     \end{cases}\\
    %     (k, g) \odot \begin{pmatrix} (-1, 0) \\ (1, 0) \end{pmatrix} \boxplus (1,0) \odot \begin{pmatrix} (3, 1) \\ (-2, -1) \end{pmatrix} &=
    %     \begin{cases}
    %         \begin{pmatrix} (3, 1) \\ (k-2, -1) \end{pmatrix} & g = -1, k \neq 2 \\
    %         \SetOf{\begin{pmatrix} (3, 1) \\ (k', g') \end{pmatrix}}{g' < -1, k' \in \RR_{>0}} & g = -1, k = 2 \\
    %         \begin{pmatrix} (3, 1) \\ (-2, 1) \end{pmatrix} & \text{otherwise}
    %     \end{cases}
    % \end{align*}
    A schematic representation of $\conv(\bp,\bq)$ is given in Figure~\ref{fig:RxZ2}.

    To highlight certain pathological behaviour convex sets over $\HH$ can have, we consider the convex hull of $\bp$ and $-\bp$.
    A similar calculation shows that all possible convex combinations of $\bp$ and $-\bp$ are
    \[
    \begin{cases}
        \begin{pmatrix} (1-2k, 0) \\ (2k - 1, 0) \end{pmatrix} & k \neq \frac{1}{2} \\
        \SetOf{\begin{pmatrix} (\ell, g) \\ (\ell', g') \end{pmatrix}}{\ell,\ell' \in \RR^\times \, , \, g,g' \in \ZZ_{<0}} & k = \frac{1}{2}
    \end{cases}
    \]
    In particular, the output is always a singleton except when $k = 1/2$: in this case we get cancellation in both coordinates, resulting in a set.
    A schematic representation of $\conv(\bp,-\bp)$ is given in Figure~\ref{fig:RxZ2}.
    While we have deliberately introduced no topological properties to $\HH$, there is the sense that $\conv(\bp,-\bp)$ should be one-dimensional but appear two-dimensional.
    This behaviour can occur for convex sets over any ordered hyperfield besides fields.
\end{example}

\added{
\begin{example}
    We give one further example to exhibit pathological behaviour of convex sets over hyperfields, namely that one point sets may not be convex.
    Recall the quotient hyperfield $\HH := \RR((t))/S$ introduced in Example~\ref{ex:finite+ordered}, and fix the ordering $\HH^+ := \{\1, \bar{t}\}$.
    The one point set $\bp = (\1) \in \HH^1$ is not convex:
    using \eqref{eq:addition+table}, we see that we can pick the scalars $a = \1$ and $b = \bar{t}$ such that $a \boxplus b \ni \1$, and
    \[
    \1\odot \bp \boxplus \bar{t}\odot \bp = \{\1, \bar{t}\} \nsubseteq \bp \, .
    \]
    The intuition behind this is that its preimage $S = \tau^{-1}(\1)$ in the quotient map is not convex over $\RR((t))$.
    We shall make this precise in Lemma~\ref{lem:pre-conv-conv}.
\end{example}
}

%We end with some nice properties of convex sets over hyperfields that are shared with both classically convex and tropically convex sets~\cite[Theorem 2]{Develin+sturmfels}.
%
We end by showing that convex sets over hyperfields satisfy certain properties one would expect in a theory of convexity.
Analogous statements hold for conic sets.
\begin{proposition}\label{prop:intconv-conv}
The intersection of convex sets over $\HH^d$ is convex. 
\end{proposition}
\begin{proof}
Let $S_1, \dots ,S_n \subseteq \HH^d$ be convex sets.
If $\bigcap_{i=1}^n S_i = \emptyset$, then it is trivially convex.
Assume that $\big | \bigcap_{i=1}^n S_i \big| \geq 1$, and take $\bp, \bq \in \bigcap_{i=1}^n S_i$.
Then $\lambda \odot \bp \boxplus \mu \odot \bq \subseteq S_i$ for all $\lambda,\mu \in \HH^+$ such that $\lambda \boxplus \mu \ni \1$. This implies that $\lambda \odot \bp \boxplus \mu \odot \bq \subseteq \bigcap_{i=1}^n S_i$, and hence is convex.
\end{proof}

\begin{proposition}
    The Cartesian product of two convex sets is convex. 
\end{proposition}
\begin{proof}
    Let $S \subseteq \HH^n$ and $T\subseteq\HH^m$ be convex sets. Then their Cartesian product is
    \[
    S \times T := \{(\bp,\bq) \, : \, \bp \in S, \, \bq \in T \} \subseteq \HH^{n+m}.
    \]
    Let $(\bp_1,\bq_1),(\bp_2,\bq_2) \in S \times T$ and $a,b \in \HH^+$ such that $a\boxplus b \ni \1$, then the convex combination formed by these elements is,
    \[ 
    a \odot (\bp_1,\bq_1) \boxplus b \odot (\bp_2,\bq_2) = (a \odot \bp_1 \boxplus b \odot \bp_2, a \odot \bq_1 \boxplus b \odot \bq_2). 
    \]
    As both $S$ and $T$ are convex, the first and second coordinates are subsets of $S$ and $T$ respectively. Thus, the tuple is a subset of $S \times T$ and hence convex.
\end{proof}

\begin{proposition}
    The projection of a convex set to a coordinate hyperplane is convex.
\end{proposition}
\begin{proof}
    Let $S \subseteq \HH^d$ be a convex set and consider the coordinate projection $\pi_i : \HH^d \rightarrow \HH^{d-1}$ that projects onto the $i$th coordinate hyperplane.
    We show the statement for $\pi:= \pi_1$% \colon (p_1, \dots, p_d) \mapsto (p_2, \dots, p_d)$
    , other projections proved identically.
    Take $\bp,\bq \in S$ and $a,b \in \HH^+$, then the convex combination is preserved under the projection:
    \begin{align}
        \pi(a \odot \bp \boxplus b \odot \bq) & = \pi(a \odot p_1 \boxplus b \odot q_1, \, \dots \, ,a \odot p_d \boxplus b \odot q_d) \nonumber \\
        & = (a \odot p_2 \boxplus b \odot q_2, \, \dots \, ,a \odot p_d \boxplus b \odot q_d) \nonumber \\
        & = a \odot (p_2, \, \dots \, ,p_d) \boxplus b \odot (q_2, \, \dots \, ,q_d) \nonumber \\
        & = a \odot \pi(\bp) \boxplus b \odot \pi(\bq). \nonumber
    \end{align}
    By the convexity of $S$, this implies that $a \odot \pi(\bp) \boxplus b \odot \pi(\bq) = \pi(a \odot \bp \boxplus b \odot \bq) \subseteq \pi(S)$. Hence, the projection $\pi(S)$ is convex.
\end{proof}

\begin{lemma}\label{lem:pre-conv-conv}
Let $f:\HH_1 \rightarrow \HH_2$ be an order preserving hyperfield homomorphism.
\added{If $S \subseteq \HH_2^d$ is convex then $f^{-1}(S) \subseteq \HH_1^d$ is convex.}
\end{lemma}
\begin{proof}
\added{Let $\bp, \bq \in f^{-1}(S)$.
For any $a,b \in \HH_1^+$ satisfying $a\boxplus b \ni \1_1$, we have} 
\begin{equation} \label{eq:convexity}
f(a \odot \bp \boxplus b \odot \bq) \subseteq f(a) \odot f(\bp) \boxplus f(b) \odot f(\bq).    
\end{equation}
\added{As $f$ is order preserving, we have $f(a), f(b) \in \HH^+_2$.
Moreover, $\1_1 \in a \boxplus b$ implies that $\1_2 = f(\1_1) \in f(a \boxplus b) \subseteq f(a) \boxplus f(b)$. 
By the convexity of $S$ and \eqref{eq:convexity}, this implies that $f(a \odot \bp \boxplus b \odot \bq) \subseteq S$.
Hence, $f^{-1}(S)$ is convex.
}
\end{proof}

\subsection{Convex and conic sets over quotient hyperfields}

The aim of this subsection to show the following structure theorem for convex and conic sets over quotient hyperfields.
\begin{theorem}\label{thm:quotient+structure}
    Let $\HH = \FF/U$ be an ordered quotient hyperfield with quotient map $\tau\colon \FF \rightarrow \HH$.
    Then for any $T \subseteq \FF^d$, we have
    \[
    \cone(\tau(T)) = \bigcup_{\tau(T) = \tau(T')} \tau(\cone(T')) \quad , \quad \conv(\tau(T)) = \bigcup_{\tau(T) = \tau(T')} \tau(\conv(T'))
    \]
\end{theorem}
This theorem generalises~\cite[Proposition 2.1]{Develin+Yu:07} and~\cite[Theorem 3.14]{LV} for (signed) tropically convex sets, as the (signed) valuation map can be seen as a quotient map by Example~\ref{ex:tropical+quotients}.
It will be a consistently useful tool throughout for investigating convexity over quotient hyperfields.

One direction of this containment is not special to the quotient map, it holds for any order preserving homomorphism.
\begin{proposition} \label{prop:cone+containment}
Let $f\colon\HH_1\rightarrow\HH_2$ be an order preserving homomorphism between ordered hyperfields.
For any set $T \subseteq \HH_1^d$, we have 
\[
f(\cone(T)) \subseteq \cone(f(T)) \quad , \quad f(\conv(T)) \subseteq \conv(f(T)) \, .
\]
\end{proposition}
\begin{proof}
Let $f(\bq) \in f(\cone(T))$ where $\bq \in \cone(T)$.
By Lemma~\ref{lem:conic+comb}, $\bq$ is contained in a finite conic combination,
\begin{align*}
\bq &\in \bigboxplus_{i=1}^k a_i \odot \bp_i \quad , \quad a_i \in \HH_1^+ \\
\Rightarrow f(\bq) &\in f(\bigboxplus_{i=1}^k a_i \odot \bp_i) \subseteq \bigboxplus_{i=1}^k f(a_i) \odot f(\bp_i) \, ,
\end{align*}
where the final step follows from properties of hyperfield homomorphisms.
Moreover, as $f$ is order preserving, it follows that $f(a_i) \in \HH_2^+$, and so $f(\bq) \in \cone(f(T))$.

The proof of the convex claim is identical, with the additional observation that
\[
\1_1 \in \bigboxplus_{i=1}^k a_i  \quad \Rightarrow \quad \1_2 = f(\1_1) \in f\left(\bigboxplus_{i=1}^k a_i\right) \subseteq \bigboxplus_{i=1}^k f(a_i) \, .
\]
\end{proof}

\begin{proof}[Proof of Theorem~\ref{thm:quotient+structure}]
Proposition~\ref{prop:cone+containment} gives one containment, it remains to show the other direction.

Let $\bar{\bq} \in \cone(\tau(T))$, we can write $\bar{\bq}$ as a finite conic linear combination $\bar{\bq} \in \bigboxplus_{i=1}^k \bar{a}_i \odot \bar{\bp}_i$ where $\bp_i \in T$.
In the $j$-th coordinate, this gives
\[
 \bar{q}_j \in \bigboxplus_{i=1}^k \bar{a}_i \odot \bar{p}_{ij} = \bigboxplus_{i=1}^k \overline{a_i \cdot p_{ij}} = \SetOf{\overline{\sum_{i=1}^k a_i \cdot p_{ij}\cdot u_{ij}}}{u_{ij} \in U} \, .
\]
Let $\bq = (q_1, \dots, q_d)$ be a representative in $\FF^d$ of the coset $\bar{\bq}$, then $q_j = \sum_{i=1}^k a_i \cdot p_{ij}\cdot v_{ij}$ for some choice of $v_{ij} \in U$.
Define $\br_i = (p_{i1}\cdot v_{i1}, \dots, p_{id}\cdot v_{id}) \in \FF^d$, then $\bq = \sum_{i=1}^k a_i \cdot \br_i \in \cone(\br_1, \dots, \br_k)$.
Moreover, $\tau(\br_i) = \tau(\bp_i)$ and so the reverse inclusion holds

For the convex claim, if $\1 \in \bigboxplus \bar{a}_i$ then there exists $w_i \in U$ such that $1 = \sum a_i\cdot w_i$.
Repeating the previous proof with $\br_i = (p_{i1}\cdot v_{i1}w_{i}^{-1}, \dots, p_{id}\cdot v_{id}w_{i}^{-1})$ gives $\bq = \sum_{i=1}^k a_i\cdot w_i \cdot \br_i \in \conv(\br_1, \dots, \br_k)$.
\end{proof}
\begin{example}\label{ex:conv-not-pf}
Consider the set $S = \{(\1,-\1), (-\1, \1)\} \subseteq \SS^2$.
The smallest convex set containing $S$ is $\conv(S) = \SS^2$, as the sum of the two points equals the whole of $\SS^2$.

Recall that $\SS = \RR/\RR_{> 0}$.
Let $T = \{(2,-1),(-1,2) \} \subseteq \RR^2$ be a choice of points such that $\tau(T) = S$.
Then, $\conv(T)$ is the line segment connecting $(2,-1)$ and $(-1,2)$ passing through the upper right quadrant.
This pushes forward to $\tau(\conv(T)) = \{(-\1,\1),(\0,\1),(\1,\1),(\1,\0),(\1,-\1)\}$, a subset of $\conv(S)$.
This is depicted in Figure \ref{fig:conv-not-pf}.
\begin{figure}
    \centering
        \includegraphics{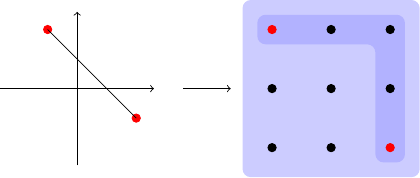}
    \caption{First case in Example~\ref{ex:conv-not-pf}, showing $\conv((2,-1),(-1,2)) \subseteq \RR^2$ and its image in the quotient map $\tau(\conv((2,-1),(-1,2))) \subseteq \conv((\1,-\1),(-\1,\1)) = \SS^2$}
    \label{fig:conv-not-pf}
\end{figure}

Note that it is not possible to find two points in $\RR^2$ such that their convex hull pushes forward to $\conv(S)$.
However, we can find a set of points $T' \subseteq \RR^2$ with larger cardinality than two such that $\conv(S) = \tau(\conv(T'))$ and $\tau(T) = S$.
Consider the set $T' = \{(-1,2), (-2,1),(1,-2),(2,-1)\}$: this pushes forward to $\tau(T') = S$ as before, just not one-to-one.
However, $\conv(T')$ is a quadrilateral that covers all quadrants of $\RR^2$, and so $\conv(S) = \tau(\conv(T'))$. This is depicted in Figure~\ref{Fig:box-conv-sgn-pf}.
\begin{figure}
    \centering
    \includegraphics{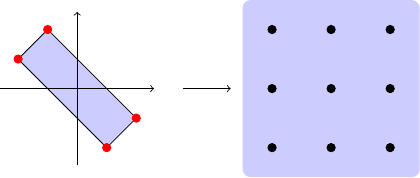}
    \caption{Second case in Example~\ref{ex:conv-not-pf}, showing $\conv(T') \subseteq \RR^2$ and its image in the quotient map $\conv(S) = \tau(\conv(T')) = \SS^2$ }
    \label{Fig:box-conv-sgn-pf}
\end{figure}
\end{example}

\subsection{Radon, Helly and Carath\'{e}odory Theorems for quotient hyperfields}
\label{sec:helly}

This section will prove generalisations of three theorems of classical convex geometry to ordered quotient hyperfields: namely Radon, Helly and Carath\'{e}odory's theorems. 
The main technique will be to lift back to an ordered field where the manipulation is more concrete, then push-forward back to the quotient hyperfield via the quotient map.

We recall the Radon, Helly and Carath\'{e}odory theorems over ordered fields.
For proofs of the statements, we refer to~\cite{Matousek:2002}: while the proofs are stated over $\RR$, they only make use of the fact that $\RR$ is an ordered field.
%We also note that this is unusual for theorems in convex geometry, as it makes no assumption on the existence of a metric.
Throughout the following, we let $\FF$ be an arbitrary ordered field.

\begin{theorem}[Radon's Theorem]\label{thm:rad-ord-field}
Let $\FF$ be an ordered field and $\{\bp_1, \dots, \bp_{d+2}\} \subseteq \FF^d$.
Then there exists a point $\bq \in \FF^d$ and a nonempty subset $I \subsetneq [d+2]$ such that \[
\bq \in \conv(\bp_i \mid i \in I) \cap \conv(\bp_j \mid j \notin I) \, .
\]
\end{theorem}

% \begin{proof}
% Let $a_1, \dots, a_{d+2} \in \FF$ be a non-zero solution to the following system of $d+1$ linear equations in $d+2$ unknowns:
% \[
% \sum_{i = 1}^{d+2} a_i = 0 \quad , \quad \sum_{i=1}^{d+2}a_i\bp_{ij} = 0 \, , \, 1 \leq j \leq d \, .
% \]
% Let $I = \SetOf{i}{a_i > 0}$ and set 
% \[
% \bq = \sum_{i \in I} \frac{a_i}{b} \bp_i = \sum_{j \notin I} \frac{-a_j}{b} \bp_j \, ,
% \]
% where $b = \sum_{i \in I} a_i = - \sum_{j \notin I} a_j$.
% Note that $b$ is positive as the sum of positive elements, and so the coefficients $a_i /b$ are all positive for $i \in I$ and sum to one.
% Similarly, $a_j$ are all negative for $j \notin I$ and so $-a_j/b$ are all positive and sum to one.
% Therefore $\bq$ suffices as the point in the statement.
% \end{proof}

\begin{theorem}[Helly's Theorem]\label{thm:hell-ord-field}
Let $S_1, \dots, S_n$ be a finite collection of convex sets in $\FF^d$, with $n \geq d+1$.
The intersection of each $d+1$ collection of sets is non-empty if and only if the intersection of all the sets is non-empty.
\end{theorem}
% \begin{proof}
% Note that one direction is trivially true, it suffices to show each $d+1$ intersection is non-empty implies the whole intersection is non-empty.
% We proceed by induction on $n$; note that $n = d+1$ is trivially true, so we take $n = d+2$ as the base case.
% For $1 \leq j \leq d+2$, there exists some $\bp_j \in \bigcap_{i \neq j} S_i$.
% By Radon's theorem, we can find $I \subseteq [d+2]$ such that 
% \[
% \bq \in \conv(\bp_i \mid i \in I) \cap \conv(\bp_j \mid j \notin I) \, .
% \]
% If $i \in I$, then $p_j \in S_i$ for all $j \notin I$, and so $\conv(\bp_j \mid j \notin I) \subseteq S_i$ by convexity.
% The same argument shows if $j \notin I$ then $\conv(\bp_i \mid i \in I) \subseteq S_j$.
% This implies $\bq \in S_i$ for all $1 \leq i \leq d+2$, and so $\bq \in \bigcap_{i=1}^{d+2} S_i$.

% For the induction step, we can replace $S_{n-1}$ and $S_n$ with $S_{n-1} \cap S_n$: this set is non-empty and convex and so the induction hypothesis completes the proof.
% \end{proof}

\begin{theorem}[Carath\'{e}odory's Theorem]\label{thm:car-ord-field}
Let $T \subseteq \FF^d$.
If $\bq \in \conv(T)$, then $\bq$ can be written as a convex combination of at most $d+1$ points in $T$.
\end{theorem}
% \begin{proof}
% As $\bq \in \conv(S)$, we can write it as a finite convex combination.
% Suppose that we can express it as the convex combination $\bq = \sum_{i=1}^k a_i\bp_i$ where $\bp_i \in S$ and $k > d+1$ is minimal, i.e. $a_i > 0$ for all $i$.
% The vectors $\bp_2 - \bp_1, \dots, \bp_k - \bp_1$ must be linearly dependent, therefore there exist $b_2, \dots, b_k \in \FF$ not all zero such that
% \[
% \sum_{i=2}^kb_i(\bp_i - \bp_1) = 0 \, .
% \]
% Setting $b_1 = -\sum_{i=2}^k b_i$ gives us a linear dependence $\sum_{i=1}^k b_i\bp_i = 0$ such that the sum of the scalars is zero.
% As $b_i$ are not all zero, there exists at least one such that $b_i > 0$.
% Define
% \[
% c = \min_{1\leq i\leq k}\left\{\frac{a_i}{b_i} \mid b_i >0\right\} = \frac{a_\ell}{b_{\ell}} > 0 \, .
% \]
% Note that
% \begin{equation} \label{eq:bx+convex+comb}
% \bq =\sum_{i=1}^k a_i\bp_i - c\left(\sum_{i=1}^k b_i\bp_i\right) = \sum_{i=1}^k(a_i - c b_i)\bp_i \, ,
% \end{equation}
% where the sum of the coefficients $a_i - c b_i$ is one.
% Furthermore
% \[
% a_i - c b_i \geq a_i - \frac{a_i}{b_i}b_i = 0 \quad , \quad 1 \leq i \leq k
% \]
% with $a_\ell - cb_\ell = 0$.
% Therefore we can remove the $\ell$-th term from~\eqref{eq:bx+convex+comb}, giving $\bq$ as a convex combination of $k-1$ points, a contradiction.
% \end{proof}

\begin{remark}
Over $\RR$, each of these theorems have many variants and strengthenings.
For example, Helly's theorem can be strengthened to a collection of infinite sets provided that the convex sets are also compact.
These variants usually involve topological arguments, which are more intricate over arbitrary ordered fields: in general, we cannot place a metric on $\FF$ as we do in $\RR$, see~\cite{dobbs+2000}.
We therefore stick to the simplest form of each statement that require no topological concerns.
\end{remark}

We now use these theorems over ordered fields to prove analogues over ordered quotient hyperfields.

\begin{theorem}\label{thm:radon}
Let $\HH$ be an ordered quotient hyperfield, and $\{\bp_1, \dots, \bp_{d+2}\} \subseteq \HH^d$.
Then there exists a point $\bq \in \HH^d$ and a nonempty subset $I \subsetneq [d+2]$ such that \[
\bq \in \conv(\bp_i \mid i \in I) \cap \conv(\bp_j \mid j \notin I) \, .
\]
\end{theorem}
\begin{proof}
Let $\HH = \FF/U$ with quotient map $\tau\colon\FF \rightarrow \HH$.
For each $\bp_i$, take some lift $\tilde{\bp}_i \in \tau^{-1}(\bp_i)$.
By Theorem~\ref{thm:rad-ord-field}, there exists some non-empty $I \subsetneq [d+2]$ and $\tilde{\bq} \in \FF^d$ such that
\begin{align*}
\tilde{\bq} &\in \conv(\tilde{\bp}_i \mid i \in I) \cap \conv(\tilde{\bp}_j \mid j \notin I) \, .
\end{align*}
Applying $\tau$ and using Proposition \ref{prop:cone+containment} gives, 
\begin{align*}
     \bq = \tau(\tilde{\bq}) &\in \tau(\conv(\tilde{\bp}_i \mid i \in I)) \cap \tau(\conv(\tilde{\bp}_j \mid j \notin I)) \\
     &\subseteq \conv(\tau(\tilde{\bp}_i) \mid i \in I) \cap \conv(\tau(\tilde{\bp}_j) \mid j \notin I) \\
     & = \conv(\bp_i \mid i \in I) \cap \conv(\bp_j \mid j \notin I).
\end{align*}
\end{proof}

\begin{theorem} \label{thm:helly}
Let $\HH$ be an ordered quotient hyperfield, and $S_1, \dots, S_n$ be a finite collection of convex sets in $\HH^d$ with $n \geq d+1$.
The intersection of each $d+1$ collection of sets is non-empty if and only if the intersection of all the sets is non-empty.
\end{theorem}
\begin{proof}
Let $\HH = \FF/U$ with quotient map $\tau\colon\FF \rightarrow \HH$.
The pre-image of a convex set is convex by Lemma~\ref{lem:pre-conv-conv}, thus $T_i = \tau^{-1}(S_i)$, are each convex sets over $\FF^d$.
If the intersection of each $d+1$ collection of $S_1, \dots, S_n$ is non-empty over $\HH^d$, then this is also the case for $T_1, \dots ,T_n$ as $\tau$ is surjective.
Then by Theorem \ref{thm:hell-ord-field}, the intersection of the complete collection is non-empty, i.e.
\begin{align*}
\emptyset \neq \bigcap_{i=1}^n T_i \quad \Rightarrow \quad \emptyset &\neq \tau\left (\bigcap_{i=1}^n T_i\right ) \subseteq \bigcap_{i=1}^n \tau(T_i) = \bigcap_{i=1}^n S_i \, ,
\end{align*}
where the last step is deduced via Proposition~\ref{prop:cone+containment}.
\end{proof}

\begin{theorem}\label{thm:caratheodory+hyp}
Let $\HH$ be an ordered quotient hyperfield, and $T \subseteq \HH^d$.
If $\bq \in \conv(T)$, then $\bq$ \added{is contained in} a convex combination of at most $d+1$ points in $T$.
\end{theorem}
\begin{proof}
Let $\HH = \FF/U$ with quotient map $\tau\colon\FF \rightarrow \HH$.
By Theorem~\ref{thm:quotient+structure}, there exists $\tilde{T} \subseteq \FF^d$ and $\tilde{\bq} \in \conv(\tilde{T})$ such that $\tau(\tilde{\bq}) = \bq$ and $\tau(\tilde{T}) = T$.
By Theorem \ref{thm:car-ord-field}, we can write $\tilde{\bq}$ as a finite convex combination of $d+1$ points $\tilde{\bq} = \sum_{j=1}^{d+1} a_j \cdot \tilde{\bp}_j$ where $\tilde{\bp}_j \in \tilde{T}$.
Applying $\tau$ gives
\[
\bq = \tau(\tilde{\bq}) = \tau \left(\sum_{j=1}^{d+1} a_j \cdot \tilde{\bp}_{j}\right) \subseteq \bigboxplus_{j=1}^{d+1} \tau(a_j) \odot \tau(\tilde{\bp}_{j}) \, .
\]
Note that $\tau(\tilde{\bp}_{j}) \in T$, and $1 = \sum_{j=1}^{d+1} a_j$ implies $\1 \in \boxplus_{j=1}^{d+1} \tau(\gamma_j)$.
Thus, $\bq$ can be written as a convex combination of at most $d+1$ points in $T$.
\end{proof}

\section{Halfspaces and hyperplanes}\label{sec:halfspace}
Let $\FF$ be an ordered field and $\phi \in \FF[\bX] = \FF[X_1,\dots, X_d]$ an affine polynomial.
There are three geometric objects associated with this polynomial: its \emph{hyperplane} $V(\phi)$, its \emph{open halfspace} $\HS(\phi)$ and its \emph{closed halfspace} $\cHS(\phi)$, defined respectively as
\begin{align*}
V(\phi) &:= \{ \bp \in \FF^d \, : \, \phi(\bp) = 0 \} \, , \\
\HS(\phi) &:= \{ \bp \in \FF^d \, : \, \phi(\bp) \in \FF^+\} \, , \\
\cHS(\phi) &:= \{ \bp \in \FF^d \, : \, \phi(\bp) \in \FF^+ \cup \{0\} \} \, . 
\end{align*}
It is immediate from the definition that $\cHS(\phi) = \HS(\phi) \sqcup V(\phi)$.
Each of these objects are convex, and are the building blocks of classical convex geometry.
A key example is the following classical theorem.

\begin{theorem}[Hyperplane Separation Theorem]\label{thm:hyp+sep+thm}
    Let $\conv(T) \subseteq \RR^d$ be a finitely generated convex set and $\bp \notin \conv(T)$.
    There exists a halfspace $\cHS(\phi)$ such that $\conv(T) \subseteq \cHS(\phi)$ and $\bp \notin \cHS(\phi)$.
\end{theorem}
There are various generalisations and variants of this theorem: we refer to~\cite{Rockafellar} for further details.
A key corollary is that one can characterise polyhedra as (finite) intersections of closed halfspaces, giving rise to notions such as the face lattice of a polyhedron.
This has led to multiple important results and applications throughout convex geometry and other areas of mathematics.

Given this, halfspaces and hyperplanes are natural candidates for studying convex geometry over hyperfields.
Moreover, we would ideally want an analogous result to Theorem~\ref{thm:hyp+sep+thm} over hyperfields.
However, we shall see many nice properties enjoyed over fields break down in the transition to hyperfields.

\begin{definition}\label{def:halfspace}
Let $\HH$ be an ordered hyperfield and $\phi \in \HH[\bX]$ an affine polynomial.
The associated \emph{hyperplane} $V(\phi)$, \emph{open halfspace} $\HS(\phi)$ and \emph{closed halfspace} $\cHS(\phi)$ are defined respectively as
\begin{align*}
V(\phi)&:= \{ \bp \in \HH^d \, : \, \phi(\bp) \ni \0 \} \, , \\
\HS(\phi)&:= \{ \bp \in \HH^d \, : \, \phi(\bp) \subseteq \HH^+\} \, , \\
\cHS(\phi) &:= \{ \bp \in \HH^d \, : \, \phi(\bp) \cap (\HH^+ \cup \{\0\} ) \neq \emptyset\} \, .
\end{align*}
\end{definition}
These objects were studied in~\cite{LV, JSY} over $\TT\RR$, the latter generalising to semialgebraic sets where one replaces the affine polynomial with an arbitrary one.

\subsection{Properties of open halfspaces}
We begin by studying properties of open halfspaces over arbitrary ordered hyperfields.
\begin{proposition}\label{>conic}
Let $\HH$ be an ordered hyperfield.
For a linear polynomial $\phi\in\HH[X_1, \dots, X_d]$, the open halfspace $\HS(\phi) \subseteq \HH^d$ is a conic set.
\end{proposition}
\begin{proof}
Let $\phi = \bigboxplus_{i=1}^d c_i \odot X_i$ and let $\bp, \bq \in \HS(\phi)$.
Consider $\br \in \lambda \odot \bp \boxplus \mu \odot \bq$ where $\lambda, \mu \in \HH^+$ are arbitrary scalars, suffices to show $\br \in \HS(\phi)$.
\begin{align*}
    \phi(\br) = \bigboxplus_{i=1}^d c_i \odot r_i &\subseteq \bigboxplus_{i=1}^d c_i \odot (\lambda \odot p_i \boxplus \mu \odot q_i) \\
    &= \bigboxplus_{i=1}^d (c_i \odot \lambda \odot p_i) \boxplus  \bigboxplus_{i=1}^d (c_i \odot \mu \odot q_i)\\
    &= \lambda \odot \left(\bigboxplus_{i=1}^d c_i \odot  p_i\right) \boxplus \mu \odot\left(\bigboxplus_{i=1}^d c_i \odot q_i\right) =
   \lambda \odot \phi(\bp) \boxplus \mu \odot \phi(\bq) \, .
\end{align*}
As $\HH^+$ closed under addition and multiplication, and $\phi(\bp), \phi(\bq) \in \HH^+$, this implies $\phi(\br) \in \HH^+$ and so $\br \in \HS(\phi)$.
\end{proof}
As a immediate corollary, we obtain that $\HS(\phi)$ is also convex.

\begin{corollary}\label{lem:aff-HS-conv}
Let $\HH$ be an ordered hyperfield.
For an affine polynomial $\phi\in\HH[X_1, \dots, X_d]$, the open halfspace $\HS(\phi)\subseteq \HH^d$ is a convex set.
\end{corollary}
\begin{proof}
Let $\phi = c_0\boxplus c_1\odot X_1 \boxplus \cdots c_d \odot X_d$, and define its homogenisation 
\[
\tilde{\phi} = \bigboxplus_{i=0}^d c_i \odot X_i \in \HH[X_0,X_1, \dots, X_d]\, .
\]
Note that $\phi(\bp) \subseteq \HH^+$ if and only if $\tilde{\phi}((\bp, \1)) \subseteq \HH^+$, or equivalently that $\bp \in \HS(\phi) \subseteq \HH^d$ if and only if $(\bp, \1) \in \HS(\tilde{\phi}) \subseteq \HH^{d+1}$.
By Proposition~\ref{>conic}, $\HS(\tilde{\phi})$ is a conic set, and so applying Lemma~\ref{lem:homog} implies $\HS(\phi)$ is a convex set.
\end{proof}

\begin{proposition}\label{prop:conv+in+open}
Let $\HH$ be an ordered hyperfield, and $T \subseteq \HH^d$.
Then $\conv(T)$ is contained in the intersection of all open halfspaces containing $T$, i.e.
\[
\conv(T) \subseteq \bigcap_{T\subseteq \HS(\phi)} \HS(\phi) \, .
\]
\end{proposition}
\begin{proof}
Corollary~\ref{lem:aff-HS-conv} implies each such $\HS(\phi)$ is a convex set containing $T$. 
By Proposition \ref{prop:intconv-conv}, the intersection of these halfspaces is convex. 
By definition, $\conv(T)$ is the smallest convex set containing $T$, hence giving the inclusion.
\end{proof}
The converse of this result holds over $\RR$, and so we get equality.
However, even over ordered fields the converse is not guaranteed if $T$ is infinite~\cite{RR:91}.
% Can show (0,0) not seperable from $\{(x,y) \in \QQ^2 \mid y = \sqrt{2} x\{

If we restrict to finitely generated convex sets, then the converse of this result is the (open) Hyperplane Separation Theorem: for any $\bp \in \HH^d$, either $\bp \in \conv(T)$ or there exists an open halfspace separating $\bp$ from $\conv(T)$.
This result has been proved when $\HH$ is an ordered field~\cite{Cernikov} and when $\HH = \TT\RR$~\cite[Theorem 5.1]{LV}.
A natural question is whether this holds for all ordered hyperfields.
The following example shows this is not true.
\begin{example}\label{ex:S-non-open-sep}
    Take the convex set $$T = \big \{ (-\1,\1),(\0,\0),(\0,\1),(\1,\0) ,(\1,\1) \big \} \subseteq \SS^2$$ and the point $(-\1,\0) \notin T$, as seen in Figure~\ref{Fig:S+non+open+sep}. The only open halfspace which contains $T$ is defined by $\phi = X_2 \boxplus \1$. Although, the point $(-\1,\0) \notin T$ also belongs to this open halfspace, i.e. $\phi(-\1,\0) = \0 \boxplus \1 = \{\1\} = \SS^+$.
    \begin{figure}
        \centering
        \includegraphics[scale=1.3]{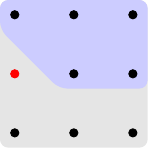}
        \caption{The convex set $T \subseteq \SS^2$ and the red point $(-\1,\0)$ which can not be separated by an open halfspace, as in Example~\ref{ex:S-non-open-sep}}
        \label{Fig:S+non+open+sep}
    \end{figure}
\end{example}
Classifying hyperfields for which the (open) hyperplane separation theorem holds appears to be difficult in general;
some results towards this will be given in Section~\ref{sec:farkas}.
We also note that separation via closed halfspaces does hold over $\SS$, but this is even harder to classify for general hyperfields; we discuss this as possible further work in Section~\ref{sec:exterior}.

We end this subsection by showing that if $\HH = \FF/U$ is a quotient hyperfield, we can describe $\HS(\phi)$ in terms of open halfspaces in $\FF$.

\begin{theorem}\label{thm:quotient+open}
    Let $\HH = \FF/U$ be an ordered quotient hyperfield with quotient map $\tau$.
    For any affine polynomial $\phi \in \FF[\bX]$, the following holds
    \[
    \HS(\tau_*(\phi)) = \tau\left(\bigcap_{\tau_*(\phi) = \tau_*(\psi)} \HS(\psi)\right) \, .
    \]
\end{theorem}
\begin{proof}
    Assume that $\tau(\bp) \in \HS(\tau_*(\phi))$, i.e. $\tau_*(\phi)(\tau(\bp)) \subseteq \HH^+$.
    Recall that $\tau(\phi(\bp)) \subseteq \tau_*(\phi)(\tau(\bp))$ for all $\bp \in \FF^d$ by hyperfield homomorphism properties.
    Therefore, we get that $\tau(\phi(\bp)) \in \HH^+$ and so $\bp \in \HS(\phi)$ by the order preserving property of $\tau$.
    Moreover, this holds for any $\psi \in \FF[\bX]$ such that $\tau_*(\psi) = \tau_*(\phi)$, giving the containment $\subseteq$.
    
    For the opposite containment, assume that $\tau(\bp) \notin \HS(\tau_*(\phi))$.
    This implies there exists some element $a \in \tau_*(\phi)(\tau(\bp)) \cap (\HH^- \cup \{\0\})$.
    Letting $\phi = c_0 + c_1X_1 + \cdots + c_dX_d$, we note that the quotient structure of $\HH$ implies there exists $\tilde{a} \in \tau^{-1}(a)$ and $u_i \in U$ such that
    \[
    c_0\cdot u_0 + \sum_{i=1}^d c_i\cdot p_i \cdot u_i = \tilde{a} \in \FF^- \cup \{0\} \, .
    \]
    Letting $\psi = c_0 u_0 + \sum c_i u_i X_i$, we observe that $\bp \notin \HS(\psi)$ and that $\tau_*(\psi) = \tau_*(\phi)$.
\end{proof}
\begin{example}\label{ex:sgn-open-HS-notIncld}
We consider the sign hyperfield as the quotient hyperfield $\SS = \RR/{\RR_{>0}}$ with quotient map $\tau\colon\RR\rightarrow\SS$.
Consider the polynomial $\phi = X_1 + X_2 \in \RR[X_1,X_2]$, the set $\HS(\phi)$ is the open halfspace in Figure~\ref{openHSnotinclusion}.
Its image in the quotient map is 
\[
\tau(\HS(\phi)) = \{(-\1,\1),(\0,\1), (\1,\1), (\1,\0), (\1,-\1)\} \subseteq \SS^2 \, .
\]
Note that this is not a convex set, as Example~\ref{ex:conv-not-pf} demonstrates.
However, the push-forward of polynomial $\tau_*(\phi) = X_1 \boxplus X_2 \in \SS[X_1,X_2]$ defines the open halfspace 
\[
\HS(\tau_*(\phi)) = \{(\0,\1), (\1,\1), (\1,\0)\} \, ,
\]
which is only a subset of $\tau(\HS(\phi))$.
One can verify this is a conic (and therefore convex) set.
\begin{figure}
    \centering
        \includegraphics{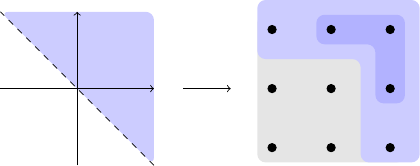}
    \caption{The open halfspace $\HS(X_1+X_2) \subseteq \RR^2$ and its image $\tau(\HS(X_1+X_2)) \subseteq \SS^2$.
    The open halfspace $\HS(X_1 \boxplus X_2) \in \SS^2$ is denoted in a darker blue, and a strict subset of $\tau(\HS(X_1+X_2))$.}
    \label{openHSnotinclusion}
\end{figure}

To obtain the characterisation in Theorem~\ref{thm:quotient+open}, consider any point $(a,b) \in \tau^{-1}((-\1,\1))$.
The polynomial $\psi = bX_1 - aX_2$ satisfies $(a,b) \notin \HS(\psi)$, and moreover $\tau_*(\psi) = \tau_*(\phi)$.
An identical condition holds for $(\1,-\1)$: for all points in the preimage there exists some halfspace that does not contain it whose polynomial pushes forward to $\tau_*(\phi)$.
Figure~\ref{fig:open-HS-pf-int} gives a schematic viewpoint of this.
\begin{figure}
    \centering
    \includegraphics{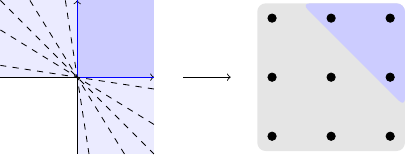}
    \caption{Demonstrating the intersection of open halfspaces, as in Theorem~\ref{thm:quotient+open}.}
    \label{fig:open-HS-pf-int}
\end{figure}
\end{example}

\subsection{Properties of closed halfspaces}
The properties of closed halfspaces are far less well-behaved over general hyperfields than open halfspaces.
Unlike over $\FF$, they are not convex and do not decompose into $V(\phi) \sqcup \HS(\phi)$ in general.
However, they do behave well with the quotient structure.
We first show that closed halfspaces over quotient hyperfields can be described in terms of closed halfspaces in $\FF$, analogously to Theorem~\ref{thm:quotient+open}.
\begin{theorem}\label{clsd_hs_eql}
Let $\HH=\FF/U$ be an ordered quotient hyperfield with quotient map $\tau\colon\FF\rightarrow \HH$.
For any affine polynomial $\phi \in \FF[\bX]$, the following holds
\[
\tau(\cHS(\phi)) = \cHS(\tau_*(\phi)) \, .
\]
\end{theorem}
\begin{proof}
For the inclusion $\tau(\cHS(\phi)) \subseteq \cHS(\tau_*(\phi))$, take $\bp \in \cHS(\phi)$.
By definition, one has $\phi(\bp) \in \FF^+ \cup \{0\}$, and so $\tau(\phi(\bp)) \in \HH^+ \cup \{\0\}$.
By hyperfield homomorphism properties~\eqref{eq:pushforward+hom}, we have 
\[
\tau(\phi(\bp)) \subseteq \tau_*(\phi)(\tau(\bp)) \, ,
\]
therefore $\tau_*(\phi)(\tau(\bp)) \cap (\HH^+ \cup \{\0\})$ is non-empty.
We conclude that $\tau(\bp) \in \cHS(\tau_*(\phi))$.

For the opposite inclusion, take $\tau(\bp) \in \cHS(\tau_*(\phi))$, implying there exists $a \in \tau_*(\phi)(\tau(\bp))\cap (\HH^+ \cup\{\0\})$.
Letting $\phi = c_0 + c_1X_1 + \cdots + c_dX_d$, we note that the quotient structure of $\HH$ implies there exists $\tilde{a} \in \tau^{-1}(a)$ and $u_i \in U$ such that
\[
    c_0\cdot u_0 + \sum_{i=1}^d c_i\cdot p_i \cdot u_i = \tilde{a} \in \FF^+ \cup \{0\} \, .
\]
Moreover, we can multiply through by $u_0^{-1} \in U \subseteq \FF^+$ to obtain
\[
c_0 + \sum_{i=1}^d c_i\cdot p_i \cdot u_iu_0^{-1} = \tilde{a}u_0^{-1} \in \FF^+ \cup \{0\} \, .
\]
Letting $\bq = (p_1\cdot u_1u_0^{-1}, \dots, p_d\cdot u_1u_0^{-1})$, we see that $\phi(\bq) \in \FF^+ \cup \{0\}$, and moreover that $\tau(\bq) = \tau(\bp)$.
This gives the opposite inclusion.
\end{proof}
Note that unlike open halfspaces, an immediate corollary of this result is that $\tau(\cHS(\phi)) = \tau(\cHS(\psi))$ whenever $\tau_*(\phi) = \tau_*(\psi)$.

It is immediate from Definition~\ref{def:halfspace} that the closed halfspace $\cHS(\phi)$ contains the open halfspace $\HS(\phi)$ and hyperplane $V(\phi)$ for any ordered hyperfield $\HH$.
Moreover, the open halfspace and hyperplane are disjoint.
From classical convex geometry, we would expect the closed halfspace to be exactly equal to $\cHS(\phi) = \HS(\phi) \sqcup V(\phi)$.
While this is not true in general for hyperfields, it is true for stringent hyperfields.
\begin{proposition}
    Let $\HH$ be an ordered stringent hyperfield.
    For any affine polynomial $\phi \in \HH[\bX]$, the closed halfspace $\cHS(\phi)$ decomposes into
    \[
    \cHS(\phi) = \HS(\phi) \sqcup V(\phi) \, .
    \]
\end{proposition}
\begin{proof}
Let $\bp \in \cHS(\phi)$, then $\phi(\bp) \cap (\HH^+ \cup \{\0\}) \neq \emptyset$.
By~\cite[Lemma 39]{BP}, $\phi(\bp)$ is a singleton unless it contains zero.
If $\phi(\bp)$ is a singleton then in must be contained in $\HH^+ \cup \{\0\}$ and hence $\bp$ is an element of $\HS(\phi) \sqcup V(\phi)$.
If $\phi(\bp)$ contains zero, then $\bp$ is an element of $V(\phi)$.
\end{proof}

When $\HH$ is not stringent, we can have polynomials $\phi$ such that an evaluation $\phi(\bp)$ may be a set containing positive and negative elements, but not zero.
This ensures $\bp \in \cHS(\phi)$ while not being contained in either $\HS(\phi)$ or $V(\phi)$.
The following example demonstrates exactly this behaviour.

\begin{example} \label{ex:Non-stringent-quo-halfspaces}
Recall $\HH = \RR((t))/S = \{\0, \1, -\1, \bar{t}, -\bar{t}\}$, the non-stringent ordered quotient hyperfield from Example~\ref{ex:finite+ordered}. 
Consider the polynomial $\phi = X_1 \boxplus X_2 \in \HH[X_1,X_2]$; a visualisation of $V(\phi), \HS(\phi)$ and $\cHS(\phi)$ is given in Figure~\ref{Fig:Non-stringent-quo-halspaces}.
In particular, we observe that there exist points in $\cHS(\phi)$ that are not in $\HS(\phi)$ or $V(\phi)$.
To see why, consider one such point $(\bar{t},-\1) \in \HH^2$ along with two points in the preimage $(t,-t^2), (t^3,-t^2) \in \tau^{-1}(\bar{t},-\1)$.
When evaluating these points in the pull-back $\tilde{\phi} = X_1 + X_2$ of $\phi$, we observe that
% Evaluating it gives 
% \[
% \phi(\bar{t},-\1) = \bar{t} \boxplus -\1 = \{\bar{t},-\1 \} \notin \HS(\phi) \sqcup V(\phi) \, .
% \]
\begin{align*}
    &\tilde{\phi}(t,-t^2) = t - t^2 \in \FF^+ & &\tau(\tilde{\phi}(t,-t^2)) = \bar{t} \\
    &\tilde{\phi}(t^3,-t^2) = -t^2 +t^3 \notin \FF^+ & &\tau(\tilde{\phi}(t^3,-t^2)) = -\1 
\end{align*}
The first point is contained in $\HS(\tilde{\phi})$ while the second isn't.
Furthermore, there is no preimage $\bp \in \tau^{-1}(\bar{t},-\1)$ such that $\bp \in V(\tilde{\phi})$.
\begin{figure}
    \centering
    \includegraphics{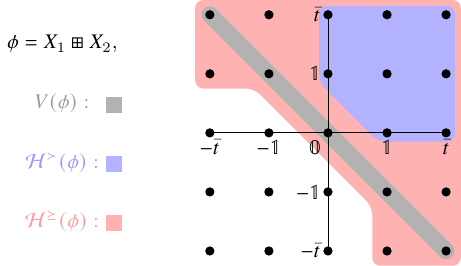}
    \caption{The open halfspace (blue), closed halfspace (red) and the hyperplane (grey) for $\phi = X_1 \boxplus X_2$ as in Example~\ref{ex:Non-stringent-quo-halfspaces}.}
    \label{Fig:Non-stringent-quo-halspaces}
\end{figure}
\end{example}

We end this section by noting that unlike open halfspaces, closed halfspaces are generally not convex, even over stringent hyperfields.

\begin{example}\label{ex:closed+not+convex}
Following Example~\ref{ex:sgn-open-HS-notIncld}, we consider the sign hyperfield as the quotient hyperfield $\SS = \RR/{\RR_{>0}}$, and the polynomial $\phi = X_1 + X_2 \in \SS[X_1,X_2]$.
The closed halfspace decomposes in
\[
\cHS(\phi) = \underbrace{\{(\0,\1), (\1,\1), (\1,\0)\}}_{\HS(\phi)} \sqcup \underbrace{\{(-\1,\1),(\0,\0), (\1,-\1)\}}_{V(\phi)} \subseteq \SS^2
\]
However, this set is not convex: we have seen that the convex hull of $(-\1,\1)$ and $(\1,-\1)$ is $\SS^2$.
Moreover, this also implies the hyperplane $V(\phi)$ is not convex either.
\end{example}

\begin{remark}
Polyhedra in $\RR$ are characterised as the intersection of finitely many closed halfspaces.
Defining polyhedra over a hyperfield this way would imply that polyhedra are generally not convex sets, marking a big departure from classical convexity theory.
Motivated by this, \cite{LS} introduced two separate notions of convexity over the signed tropical hyperfield: tropical open (TO)-convexity and tropical closed (TC)-convexity.
TO-convex sets are the intersection of open halfspaces, and coincides with the definition in this paper, while TC-convex sets are the intersection of closed halfspaces.
It has been shown that TO-convex sets are also TC-convex, and so this is a strictly more general definition.
However, this comes at the expense of some desirable properties that one may want from convex sets: for example, TC-convex sets do not have to be connected.
It would be interesting to study general `hyperfield closed' (HC)-convexity and see how it differs from TC-convexity.
\end{remark}

\section{Hemispace separation and the Kakutani property}\label{sec:kakutani}

As Example~\ref{ex:S-non-open-sep} demonstrates, the Hyperplane Separation Theorem does not hold in general over ordered hyperfields.
In this section we consider a different flavour of separation that occurs more naturally for hyperfields.
Over ordered fields, both open and closed halfspaces are convex sets whose complement is convex, but this is not the case over hyperfields.
We therefore consider a different object for separation, a \emph{hemispace}.

\begin{definition}
    A \emph{hemispace} is a convex set whose complement is also convex.
    Precisely, $X \subseteq \HH^d$ is a hemispace if $X$ and $X^c := \HH^d\setminus X$ are convex. 
\end{definition}

It is easy to see that all halfspaces are hemispaces over $\FF^d$.
However, even over $\RR$ the converse is not true as Example~\ref{ex:U+invariant+hemispace} demonstrates: see~\cite{MLas} for a systematic study of hemispaces over Euclidean vector spaces.
A similar phenomenon occurs in the max-plus semiring setting, where all tropical halfspaces are also tropical hemispaces~\cite{BH}, but the converse is not true: we refer to~\cite{KNS:14,EHN:16} for a complete characterisation of tropical hemispaces.
This is not the case for halfspaces over $\TT\RR$ or other ordered hyperfields; 
as discussed in~\cite{LS} and Section~\ref{sec:halfspace}, closed halfspaces and complements of open halfspaces may not be convex.
As such, hemispaces are the more natural and well-behaved object for convex geometry in this setting. 

\begin{example}
Figure~\ref{fig:three-hemi-examples} gives three examples of hemispaces in $\SS^2$.
 Figure \ref{fig:hemi-half} depicts complementary hemispaces that are also halfspaces: $\HS(-X_2)$ is precisely the blue region.
 Figure \ref{fig:hemi-not-half} presents complementary hemispaces which cannot be constructed as halfspaces.
 Finally, Figure \ref{fig:half-not-hemi} describes a halfspace and its complement which are not complementary hemispaces.
 Explicitly, the green halfspace contains both $(-\1,\1)$ and $(\1,-\1)$. 
 As $(-\1,\1) \boxplus (\1,-\1) = \SS^2$ is not contained in the green halfspace, it cannot be convex.

\begin{figure}[ht]
     \centering
     \begin{subfigure}[b]{0.25\textwidth}
         \centering
         \includegraphics[width=\textwidth]{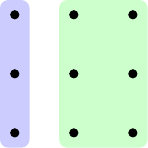}
         \caption{Complementary hemispaces which are halfspaces}
         \label{fig:hemi-half}
     \end{subfigure}
     \hfill
     \begin{subfigure}[b]{0.25\textwidth}
         \centering
         \includegraphics[width=\textwidth]{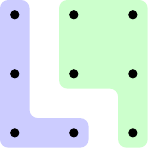}
         \caption{Complementary hemispaces but not halfspaces}
         \label{fig:hemi-not-half}
     \end{subfigure}
     \hfill
     \begin{subfigure}[b]{0.25\textwidth}
         \centering
         \includegraphics[width=\textwidth]{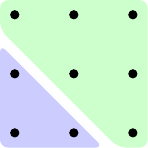}
         \caption{Halfspaces but not hemispaces}
         \label{fig:half-not-hemi}
     \end{subfigure}
        \caption{Examples of hemispaces and halfspaces in $\SS^2$.}
        \label{fig:three-hemi-examples}
\end{figure}
\end{example}

We now give our main theorem of this section: any two disjoint convex sets can be separated by hemispaces over quotient hyperfields. This generalises~\cite[Theorem 3.3]{LS} surrounding hemispace separation in signed tropical convexity.
\begin{theorem}(Kakutani Property)\label{thm:kakutani}
    Let $\HH = \FF/U$ be an ordered quotient hyperfield.
    If $A,B \subseteq \HH^d$ are two disjoint convex sets, then there exists a hemispace $X \subseteq \HH^d$ such that $A \subseteq X$ and $B \subseteq X^c$.
\end{theorem}

\begin{example}
The Kakutani property is depicted in two cases over $\SS^2$ in Figures \ref{fig:kak-ex-1} and \ref{fig:kak-ex-2}. The grey regions indicate the two disjoint convex sets and the green and blue regions are the complementary hemispaces used to separate.
\begin{figure}[ht]
    \centering
    \begin{subfigure}[b]{0.45\textwidth}
         \centering
         \includegraphics[width=0.6\textwidth]{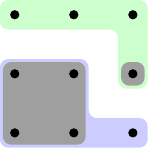}
         \caption{}
         \label{fig:kak-ex-1}
     \end{subfigure}
     \hfill
     \begin{subfigure}[b]{0.45\textwidth}
         \centering
         \includegraphics[width=0.6\textwidth]{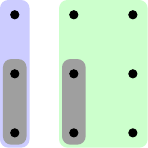}
         \caption{}
         \label{fig:kak-ex-2}
     \end{subfigure}
    \caption{Examples of the Kakutani property over $\SS^2$.}
    \label{fig:kak-exs}
\end{figure}
\end{example}

The proof of Theorem~\ref{thm:kakutani} takes a similar approach to the analogous statement over $\TT\RR$ in~\cite{LS}.
It is reliant on the fact that the Kakutani property is equivalent to the \emph{Pasch property} for particular types of convexity.
We prove this by utilising the Pasch property over ordered fields and pushing it forward to the quotient hyperfield.
%Therefore, this proof will follow later in the section.
%We now present some examples of hemispaces and the Kakutani property.

\begin{proposition}\label{prop:pasch}(Pasch Property)
    Let $\HH = \FF/U$ be an ordered quotient hyperfield.
    For all $\br, \bq_1 , \bq_2 \in \HH^d$ and $\bp_1 \in \conv(\br,\bq_1)$ and $ \bp_2 \in \conv(\br, \bq_2)$, we have
    \[
    \conv(\bq_1 , \bp_2) \cap \conv(\bq_2 , \bp_1) \neq \emptyset.
    \]
\end{proposition}
\begin{proof}
    By~\cite[I. Proposition 4.14.1]{VDV}, the Pasch property holds over all ordered fields.
    As such, we proceed by constructing lifts of the points $\br, \bq_1 , \bq_2, \bp_1, \bp_2$ in $\FF^d$ and utilising the Pasch property there.
    We let $\widetilde{\bx} \in \FF^d$ denote an element of the preimage $\tau^{-1}(\bx)$ for $\bx \in \HH^d$.
    
    By Theorem~\ref{thm:quotient+structure}, there exist elements of $\FF^d$ such that
    \[
    \widetilde{\bp}_1 \in \conv(\widetilde{\br}_1,\widetilde{\bq}_1) \quad , \quad \widetilde{\bp}_2 \in \conv(\widetilde{\br}_2,\widetilde{\bq}_2) \, .
    \]
    In particular, $\widetilde{\br}_1,\widetilde{\br}_2 \in \tau^{-1}(\br)$ but are possibly distinct; to utilise the Pasch property, they must be the same.
    As they are both in the preimage of $\br$, there exists $\bu \in U^d$ such that 
    \[
    \widetilde{\br}_1 = \bu * \widetilde{\br}_2 := (u_1\cdot \widetilde{r}_{21}, \dots, u_d\cdot \widetilde{r}_{2d}) \, .
    \]
    Applying $\bu$ to $\widetilde{\bp}_2$ and $\widetilde{\bq}_2$ implies that $\bu*\widetilde{\bp}_2 \in \conv(\widetilde{\br}_1, \bu*\widetilde{\bp}_2)$.
    Furthermore applying $\bu$ does not change their image under $\tau$, as $\bu*\widetilde{\bp}_2 \in \tau^{-1}(\bp_2)$ and $\bu*\widetilde{\bq}_2 \in \tau^{-1}(\bq_2)$.

    We can now apply the Pasch property over $\FF^d$ to show there exists some $\bw \in \FF^d$ such that
    \[
    \bw \in \conv(\widetilde{\bq}_1,\bu*\widetilde{\bp}_2) \cap \conv(\bu*\widetilde{\bq}_2,\widetilde{\bp}_1) \, .
    \]
    Applying Theorem~\ref{thm:quotient+structure} implies the Pasch property holds over $\HH$:
  \begin{align*}
     \tau(\bw) &\in \tau(\conv(\widetilde{\bq}_1,\bu*\widetilde{\bp}_2)) \cap \tau(\conv(\bu*\widetilde{\bq}_2,\widetilde{\bp}_1)) \\
     &\subseteq \conv(\tau(\widetilde{\bq}_1),\tau(\bu*\widetilde{\bp}_2)) \cap \conv(\tau(\bu*\widetilde{\bq}_2),\tau(\widetilde{\bp}_1))\\
     &= \conv(\bq_1 , \bp_2) \cap \conv(\bq_2 , \bp_1) \, .
\end{align*}
    
%     As $\HH = \FF/U$, Theorem \ref{thm:quotient+structure} implies that lifts of all $\br , \bp_1 , \bp_2 , \bq_1 , \bq_2 \in \HH^d$ can be made to elements of $\FF^d$. (This is done in a similar manner to the proof of Theorem~\ref{thm:caratheodory+hyp}). More concretely, the lifts are $\tilde{\br} , \tilde{\bp}_1 , \tilde{\bp}_2 , \tilde{\bq}_1 , \tilde{\bq}_2 \in \FF^d$, such that $\tilde{x} \in \tau^{-1}(x)$ for $x\in \{\br , \bp_1 , \bp_2 , \bq_1 , \bq_2\}$, where $\tau: \FF \rightarrow \HH = \FF/U$ and 
%     \[
%     \tilde{\bp}_1 \in \conv(\tilde{\br},\tilde{\bq}_1) \quad , \quad \tilde{\bp}_2 \in \conv(\tilde{\br},\tilde{\bq}_2).
%     \]
%     Noting that the lift $\tilde{\br} \in \tau^{-1}(\br)$ is the same in both cases. This is required to utilise the Pasch property over $\FF$ and Theorem \ref{thm:quotient+structure} ensures this lifts can be constructed in this way. By the Pasch property holding over $\FF$, (see \cite[I. Proposition 4.14.1]{VDV}), 
% \begin{align}
%     & \Rightarrow \conv(\tilde{\bq}_1,\tilde{\bp}_2) \cap \conv(\tilde{\bq}_2,\tilde{\bp}_1) \neq \emptyset, \nonumber \\
%     & \Rightarrow \tau(\conv(\tilde{\bq}_1,\tilde{\bp}_2)) \cap \tau(\conv(\tilde{\bq}_2,\tilde{\bp}_1)) \neq \emptyset. \nonumber
% \end{align}
%     Then, by Proposition~\ref{prop:cone+containment},
%     \[
%     \conv(\tau(\tilde{\bq}_1),\tau(\tilde{\bp}_2)) \cap \conv(\tau(\tilde{\bq}_2),\tau(\tilde{\bp}_1)) \neq \emptyset. 
%     \]
%     Therefore, by the construction of the lifts,
%     \[
%     \Rightarrow \conv(\bq_1 , \bp_2) \cap \conv(\bq_2 , \bp_1) \neq \emptyset.
%     \]
\end{proof}

\begin{example}
Figure~\ref{fig:pasch-ex} depicts two examples of the Pasch property.
The first is a schematic of how to view the Pasch property in Euclidean space.

The second is an example of the Pasch property over $\SS^2$.
The dark and light blue regions on the left hand side depict $\conv(\br,\bq_1)$ and $\conv(\br,\bq_2)$ respectively.
On the right hand side the light green and olive regions depict $\conv(\bp_1,\bq_2)$ and $\conv(\bp_2,\bq_1)$ respectively. The point $\bw = (1,0)$ illustrates that $\conv(\bp_1,\bq_2) \cap \conv(\bp_2,\bq_1)$ is non-empty.
\begin{figure}[ht]
    \centering
    \includegraphics[width=0.3\textwidth]{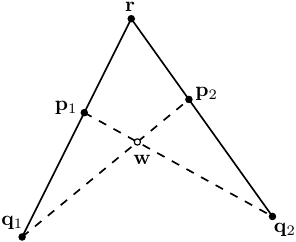} \hfill
    \includegraphics[width=0.55\textwidth]{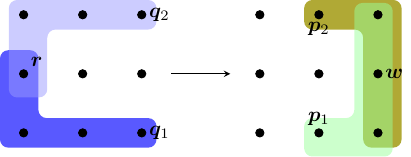}
    \caption{Two examples of the Pasch property.
    The left shows an example in Euclidean space, the right shows an example over $\SS^2$.}
    \label{fig:pasch-ex}
\end{figure}
\end{example}

% \begin{lemma}\label{lem:conv-2-ary}
% The convexity over $\HH$ is a \emph{2-ary} convexity. Specifically, $S \subseteq \HH^d$ is convex if and only if $\conv(\bp,\bq) \subseteq S$ for all $\bp,\bq \in S$. (See \cite{VC} for n-ary definition in general.)
% \end{lemma}
% \begin{proof}
% Say, $\conv(\bp,\bq) \subseteq S$ for all $\bp,\bq \in S$, then by definition $S$ is convex. Conversely, if $S$ is convex then $S =\conv(S)$. By Lemma \ref{lem:conv+comb} $\conv(S)$ is equal to the set of convex combinations of, possibly multisets, of elements in $S$. In particular, $\conv(\bp,\bq) \subseteq S$ for all $\bp,\bq \in S$. 
% \end{proof}

\begin{proof}[Proof of Theorem~\ref{thm:kakutani}]
\added{Let $\cC$ be the collection of all convex subsets of $\HH^d$, including the empty set as the trivial convex set. 
We show that this forms an abstract \emph{convexity} in the sense of~\cite{VanDeVel}, as it satisfies the following three conditions.
Firstly, $\emptyset$ and $\HH^d$ are contained in $\cC$.
Secondly, $\cC$ is closed under arbitrary intersections by Proposition~\ref{prop:intconv-conv}.
Finally, given a chain $S_0 \subseteq S_1 \subseteq S_2 \subseteq \cdots$ of convex sets, the union $S = \bigcup_{i \geq 0} S_i$ is also convex.
This follows as for any $\bp, \bq \in S$, there exists some $k \geq 0$ such that $\bp, \bq \in S_k$, hence $\conv(\bp,\bq) \subseteq S_k \subseteq S$.

We next note that $\cC$ is a 2-ary convexity as defined in~\cite[Section 1.4]{VanDeVel}.
Explicitly, a convexity is 2-ary if $S \in \cC$ whenever $\conv(\bp,\bq) \subseteq S$ for all $\bp,\bq \in S$, which is precisely our definition of a convex set $\HH^d$.}
By \cite[Theorem 5]{Chepoi}, the Pasch property is equivalent to the Kakutani property for 2-ary convexities.
As Proposition~\ref{prop:pasch} holds for all quotient hyperfields, so does the Kakutani property.
\end{proof}

We deduce the following as an immediate corollary of Theorem~\ref{thm:kakutani}.
\begin{corollary}
    Let $\HH = \FF/U$ be a quotient hyperfield.
    Any convex set $S \subseteq \HH^d$ can be written as the intersection of the hemispaces containing it, i.e.
    \[
    S = \bigcap_{\substack{X \supseteq S \\ X \text{hemispace}}} X \, .
    \]
\end{corollary}
\begin{proof}
    A point $\bp \in \HH^d$ is always a convex set.
    Therefore given any point $\bp \notin S$, Theorem~\ref{thm:kakutani} implies there exists a hemispace $X$ such that $S \subseteq X$ and $\bp \notin X$.
\end{proof}

We conclude this section by examining the structure of hemispaces over quotient hyperfields, and connecting them to hemispaces over ordered fields.
Given $U \subseteq \FF^\times$, we say a set $Y \subseteq \FF^d$ is \emph{$U$-invariant} if
\begin{align*}
    \bp \in Y \quad \text{ if and only if } \quad \bu * \bp := (u_1\cdot p_1, \dots, u_d\cdot p_d) \in Y \quad , \quad \text{ for all } \bu \in U^d \, .
\end{align*}
Note that if $Y$ is a $U$-invariant hemispace, its complement $Y^c = \FF^d \setminus Y$ is also $U$-invariant as $\bp \notin Y$ if and only if $\bu * \bp \notin Y$ for all $\bu \in U^d$.

\begin{proposition}
    Let $\HH = \FF/U$ be a quotient hyperfield.
    Then $X \subseteq \HH^d$ is a hemispace if and only if $\tau^{-1}(X) \subseteq \FF^d$ is a $U$-invariant hemispace.
\end{proposition}
\begin{proof}
    Given a hemispace $X \subseteq \HH^d$ and its complement $X^c = \HH^d \setminus X$, Lemma~\ref{lem:pre-conv-conv} implies the preimages $Y = \tau^{-1}(X)$ and $Y^c = \tau^{-1}(X^c)$ are complementary hemispaces of $\FF^d$.
    Moreover, they are $U$-invariant as $\overline{\bp} = \overline{\bu*\bp}$ by the quotient structure.

    Conversely, let $Y$ and $Y^c$ be complementary $U$-invariant hemispaces of $\FF^d$.
    Let $\overline{\bp}, \overline{\bq} \in \tau(Y)$ and $\overline{a}, \overline{b} \in \HH^+$ such that $\1 \in \overline{a} \boxplus \overline{b}$.
    For any $\overline{\br} \in \overline{a} \odot \overline{\bp} \boxplus \overline{b} \odot \overline{\bq}$, there exists $\bu, \bv \in U^d$ such that $\br = a\cdot(\bu * \bp) + b\cdot (\bv * \bq)$.
    As $Y$ is convex and $U$-invariant, to follows that $\br \in Y$.
    As such, $\overline{\br} \in \tau(Y)$ and so $\tau(Y)$ is convex; an identical proof holds for $\tau(Y^c)$.
    Finally, suppose that $\overline{\br} \in \tau(Y) \cap \tau(Y^c)$, then there exists $\br, \br' \in \tau^{-1}(\overline{\br})$ such that $\br \in Y$ and $\br' \in Y^c$.
    As $\br, \br'$ have the same image in $\tau$, there exists $\bu \in U^d$ such that $\br' = \bu * \br$.
    As $Y$ is $U$-invariant, this implies $\br' \in Y \cap Y^c$, a contradiction to $Y$ and $Y^c$ being complementary hemispaces.
    This shows $\tau(Y)$ and $\tau(Y^c)$ are complementary hemispaces.
\end{proof}

\begin{example}\label{ex:U+invariant+hemispace}
    We return to our running example of $\SS^2$ considered as the quotient hyperfield $\SS = \RR/\RR_{>0}$.
    Consider the $\RR_{>0}$-invariant hemispace
    \[
    Y = \{(x_1,x_2) \in \RR^2 \mid x_1 > 0 \text{ or } x_1 = 0 , x_2 \geq 0\} \subseteq \RR^2 \, .
    \]
    Its image under the quotient map is the hemispace $\tau(Y) \subseteq \SS^2$, as depicted in Figure~\ref{fig:pf-hemispaces}.
    \begin{figure}[ht]
        \centering
        \includegraphics[width=0.6\textwidth]{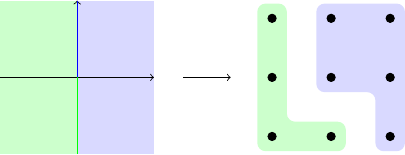}
        \caption{$\RR_{>0}$-invariant hemispaces over $\RR^2$ and their images over $\SS^2$, which are complementary hemispaces. }
        \label{fig:pf-hemispaces}
    \end{figure}

    Conversely, consider the complementary hemispaces
    \[
    Y = \{(x_1,x_2) \in \RR^2 \mid x_2 - x_1 > 0 \text{ or } x_1 = x_2 \geq 0\} \subseteq \RR^2 \quad , \quad Y^c = \RR^2 \setminus Y \, .
    \]
    These are not $\RR_{>0}$-invariant: for example, the point $(1,1) \in Y$ but $(2,1) \notin Y$.
    As a result, their images $X = \tau(Y)$ and $X^c = \tau(Y^c) \subseteq \SS^2$ are not disjoint, as both $X$ and $X^c$ contain $(\1,-\1)$ and $(-\1,\1)$, as seen in Figure \ref{fig:not-pf-hemispaces}.
    Furthermore, neither can be convex as the smallest convex set containing $(\1,-\1)$ and $(-\1,\1)$ is $\SS^2$.
    \begin{figure}[ht]
        \centering
        \includegraphics[width=0.6\textwidth]{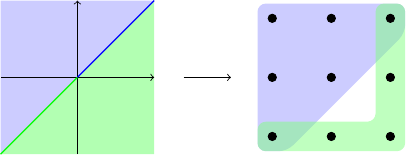}
        \caption{Hemispaces over $\RR^2$ that are not $\RR_{>0}$-invariant, and their images over $\SS^2$, which are not complementary hemispaces.}
        \label{fig:not-pf-hemispaces}
    \end{figure}
\end{example}

\section{Halfspace separation for stringent hyperfields}\label{sec:farkas}

We saw in Section~\ref{sec:halfspace} that the Hyperplane Separation Theorem (Theorem~\ref{thm:hyp+sep+thm}) does not hold in general for ordered hyperfields.
Moreover, in Section~\ref{sec:kakutani} we saw that hemispace separation holds for all quotient hyperfields via the Kakutani Property (Theorem~\ref{thm:kakutani}), implying that hemispaces may be the more natural objects for convex geometry over hyperfields.
However, there are examples of hyperfields for which (open) hyperplane separation does hold, such as $\TT\RR$, and we would like to begin classifying these hyperfields.
In this section, we make progress in this direction by restricting to stringent hyperfields.

We say that an ordering on $\HH$ is \emph{dense} if for all $a,b$ such that $b \boxplus -a \subseteq \HH^+$, there exists $c$ such that $b \boxplus -c \subseteq \HH^+$ and $c \boxplus -a \subseteq \HH^+$.
Note that this definition is equivalent to the definition of dense for order relations $\succ_{\HH^+}$.
It is straightforward to show that any ordering on $\FF\rtimes G$ is dense, and that an ordering on $\SS\rtimes G$ is dense if and only if the total order on $G$ is dense.
%We show the following theorem.

\begin{theorem} \label{thm:separation}
    Let $\HH$ be an ordered stringent hyperfield with a dense ordering.
    Consider a finitely generated convex set $\conv(T) \subseteq \HH^d$ and point $\bp \notin \conv(T)$.
    There exists an open halfspace $\HS(\phi)$ such that $\conv(T) \subseteq \HS(\phi)$ and $\bp \notin \HS(\phi)$ if either
    \begin{itemize}
        \item $\HH = \SS \rtimes G$,
        \item $\HH = \FF \rtimes G$, and $T$ and $\bp$ are sufficiently generic.
    \end{itemize}
\end{theorem}
We will refrain from precisely defining `sufficiently generic' until the proof, but one can consider it as taking sums of elements in $T \cup \{\bp\}$ always results in singleton vectors.

The proof of Theorem~\ref{thm:separation} closely follows and generalises the techniques used in~\cite{LV} to prove halfspace separation over the signed tropical semiring.
Explicitly, we adapt their formulation of Fourier-Motzkin elimination to all densely ordered stringent hyperfields.
Their methods require an adaption for hyperfields of the form $\FF \rtimes G$, where we must allow for inequalities whose coefficients may not be singletons.
This is the content of Section~\ref{sec:FM}.
The remainder of the proof is a variant of the Farkas lemma for hyperfields, given in Section~\ref{sec:farkaslem}.
The result follows for $\SS \rtimes G$, with the signed tropical hyperfield as an example of this.
However, the Farkas lemma only provably holds over $\FF \rtimes G$ when the inputs are sufficiently generic: it remains to be seen whether this genericity condition can be relaxed.

We briefly showcase some examples where densely ordered is necessary, but genericity may not be necessary.

\begin{example}
    Recall that the sign hyperfield does not have a dense ordering, as $\0 \prec \1$ but there is no element between them.
    One could also deduce this from Example~\ref{ex:S-non-open-sep}, where we present an example of a point that cannot be separated from a convex set with an open halfspace.
    We do not know whether one can find counterexamples over other stringent hyperfields without a dense ordering.
\end{example}

\begin{example}\label{ex:RxZ+sep}
    Let $\HH = \RR \rtimes \ZZ$.
    Recall from Example~\ref{ex:RxZ2} the convex set $\conv(\bp, -\bp) \subseteq \HH^2$ where $\bp = [(-1,0), (1,0)] \in \HH^2$.
    The point $\bq = [(1,0), (1,0)]$ is not contained in $\conv(\bp, -\bp)$, as depicted in Figure~\ref{fig:RxZ+sep}.
    We note that $\bp, -\bp$ and $\bq$ are far from generic points, however we can still separate $\bq$ from $\conv(\bp,-\bp)$.
    Consider the polynomial \added{$\phi = -X_1 \boxplus -X_2 \boxplus (k,0) \in \HH[X_1,X_2]$ where $k \in (0,2)$.
    One can check that $\conv(\bp,-\bp) \subseteq \HS(\phi)$ but $\bq \notin \HS(\phi)$.}
    %One can also check that for any other choice of $\bq$, there exists an open halfspace that separates it from $\conv(\bp,-\bp)$.
\end{example}
\begin{figure}
    \centering
    \includegraphics[scale=0.8]{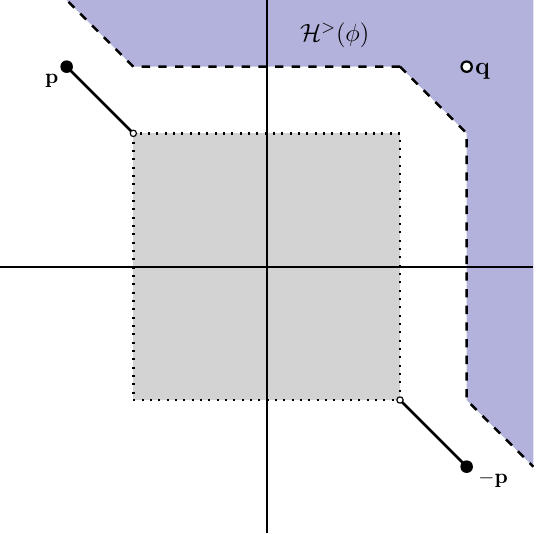}
    \caption{Separating the point $\bq$ from $\conv(\bp,-\bp)$ over $\RR \rtimes \ZZ$ as in Example~\ref{ex:RxZ+sep}.
    The separating open halfspace is $\HS(\phi)$ where \added{$\phi = -X \boxplus -Y \boxplus (k,0)$ for $k \in (0,2)$}.
    }
    \label{fig:RxZ+sep}
\end{figure}

Motivated by Example~\ref{ex:RxZ+sep}, we ask whether separation holds more generally for hyperfields of the form $\FF \rtimes G$.
\begin{question}
    Can we drop the genericity condition from Theorem~\ref{thm:separation}?
\end{question}

\subsection{Fourier-Motzkin elimination for stringent hyperfields} \label{sec:FM}
In the following, we will develop Fourier-Motzkin elimination for stringent hyperfields.
\added{Every hyperfield $\HH$ shall be assumed to be ordered and stringent in this subsection.}
%All hyperfields in this subsection will be ordered stringent hyperfields.
As such, we can now alternate between the ordering $\HH^+$ and the order relation $\succ_{\HH^+}$ as it is a strict total order by Proposition~\ref{prop:partial+order}.

Up until this point, we have treated linear polynomials $\phi \in \HH[\bX]$ as having singleton coefficients.
However, when doing Fourier-Motzkin elimination we will have to add polynomials, resulting in `polynomials' whose coefficients are sets.
In general, not every subset of $\HH$ will be obtainable as a sum of elements.
As such, we say a subset $A \subseteq \HH$ is \emph{realisable} if it can be written as a (finite) linear combination of elements of $\HH$, i.e. $A = \bigboxplus_{i=1}^n a_i$ where $a_i \in \HH$.
We denote the set of realisable sets of $\HH$ as $\cR(\HH) \subseteq \cP(\HH)^*$.

\begin{lemma}{\cite[Lemma 39]{BP}}\label{lem:realisable+sets}
    Let $\HH$ be a stringent hyperfield.
    The realisable sets of $\HH$ are precisely the singletons and $a \boxplus -a$ for $a \in \HH$.
\end{lemma}

We require some technical lemmas that will allow us to eliminate variables from inequality systems.
\begin{lemma} \label{lem: dense+order}
    Let $\HH$ be an ordered stringent hyperfield whose ordering is dense, and let $A,B \in \cR(\HH)$ be realisable sets.
    Then $a \prec b$ for all $a \in A$ and $b \in B$ if and only if there exists $c \in \HH$ such that $a \prec c \prec b$ for all $a \in A, b \in B$.
\end{lemma}
\begin{proof}
    We first observe that at least one of $A, B$ must be a singleton, else $\0 \in A \cap B$ and $\0 \nprec \0$.
    If both are singletons, the order being dense gives the result.
    If one is a singleton, say $B = \{b\}$, then the proof depends on the hyperfield.

    For $\HH = \SS \rtimes G$, without loss of generality we have $A = a \boxplus -a$, and $a = \sup(A)$.
    As $a' \preceq a \prec b$ for all $a' \in A$, denseness of the order gives the existence of some $c$ such that $a \prec c \prec b$.

    For $\HH = \FF \rtimes G$, we have $A = (k,g) \boxplus (-k,g) = \{(k', g') \mid g' < g \, , k' \in \FF\} \cup \0$.
    This implies $b = (\ell, h)$ for some $h \geq g$ and $\ell \in \FF^+$.
    Taking $c = (\ell/2, h)$ suffices.
\end{proof}

\begin{lemma}\label{lem:sign+set+split}
    Let $\HH$ be a stringent ordered hyperfield, and $a, b \in \HH$.
    Then $a \boxplus b \boxplus -b \succ \0$ if and only if there exists $c \in \HH$ such that $a \boxplus c \succ \0 \, , \, a \boxplus -c \succ \0$ and $b \boxplus -b = c\boxplus -c$.
\end{lemma}
\begin{proof}
    We prove the statement separately for $\SS \rtimes G$ and $\FF \rtimes G$.
    
    For $\SS \rtimes G$, we note that $b \boxplus -b = c \boxplus -c$ if and only if $c = \pm b$, so we prove it for $c=b$.
    For the first direction, we note that $a \boxplus b \boxplus -b = \bigcup_{d \in b \boxplus -b} a \boxplus d$. 
    As $\{b, -b\} \subseteq b \boxplus -b$, this immediately implies $a \boxplus b, a\boxplus -b \subseteq \HH^+$ as special cases.
    Conversely, as $a \boxplus a = a$ in $\SS \rtimes G$ and $\HH^+$ is closed under sums, we have
    \[
    (a \boxplus b) \boxplus (a\boxplus -b) = a \boxplus b \boxplus -b \subseteq \HH^+ \, .
    \]
    For $\FF \rtimes G$, we denote $a = (k_a, g_a)$ and $b = (k_b,g_b)$.
    For the first direction,
    \[
    a \boxplus (b \boxplus -b) = (k_a, g_a) \boxplus \left(\left\{(\ell, h) \mid h < g_b \right\}\cup \{\0\}\right) \subseteq \HH^+ 
     \quad \Rightarrow \quad g_a \geq g_b \, , \, k_a \in \FF^+ .
    \]
    Setting $c = (k_a/2, g_b)$ satisfies the conditions that $a \boxplus \pm c \subseteq \HH^+$ and $c \boxplus -c = b \boxplus -b$.
    
    Conversely, suppose $a \boxplus c \subseteq \HH^+$ and $a \boxplus -c \subseteq \HH^+$.
    then $g_a \geq g_c = g_b$.
    If $g_c > g_a$, then $a \boxplus \pm c = \pm c$, which cannot both be positive.
    Hence $g_a \geq g_c = g_b$; this immediately implies that $a \boxplus b \boxplus -b \subseteq \HH^+$.
\end{proof}

\begin{remark}
    In the previous lemma, we can replace the element $a \in \HH$ with the realisable set $A \in \cR(\added{\HH})$ and get precisely the same result.
    This is because $A \boxplus (b \boxplus -b) \subseteq \HH^+$ forces $A$ to be a singleton over stringent hyperfields, as the sum of two non-singleton sets contains $\0$.
\end{remark}
% \begin{remark}
%     The property that $a \boxplus b \boxplus -b \subseteq \HH^+$ implies that $a \boxplus \pm b \subseteq \HH^+$ is not true in $\FF \rtimes G$.
%     For example, consider $a = (k_a, g), b = (k_b, g)$ where $k_a < k_b$ in the ordering on $\FF$.
%     Then $a \boxplus -b \nsubseteq \HH^+$, but we still have $a \boxplus b \boxplus -b \subseteq \HH^+$.
% %    \ben{Figure?}
%     As Fourier-Motzkin elimination only deduces whether a solution exists, and so this additional complication can be circumvented.
% \end{remark}

% It will be convenient for us to consider the space $\cR(\HH)^d$.
% We call an element $\bA \in \cR(\HH)^d$ a \emph{realisable set-vector}, a length $d$ vector whose $i$-th entry is a realisable set $A_i = \bigboxplus_{k=1}^m a_{ik}$.
% Observe that every realisable set-vector can be written as a hypersum of vectors in $\HH^d$: assuming that each $A_i$ can be written as the sum of $m$ elements (inserting extra zeroes if necessary), then $\bA = \bigboxplus_{k=1}^m \ba_k$, where $\ba_k = (a_{1k}, \dots, a_{dk})$. We note that for a realisable set $A = \bigboxplus_{i=1}^n a_i$, its product with a scalar $\lambda \in \HH$ is well defined as by distributivity we have
% \[
% \lambda \odot A = \lambda \odot \left(\bigboxplus_{i=1}^n a_i\right) = \bigboxplus_{i=1}^n \lambda \odot a_i \, .
% \]
% \ben{This may be overstating the vectorspace viewpoint, can trim down}

Consider the system of $n$ strict linear inequalities in $d$ variables:
\begin{align} \label{eq:strict+inequality}
    \bigboxplus_{i=1}^d A_{ij} \odot X_i \succ \0 \quad , \quad \added{A_{ij} \in \cR(\HH)} \, , \, j = 1, \dots, n \, .
\end{align}
While each $\bigboxplus_{i=1}^d A_{ij} \odot X_i$ is not a linear polynomial, we can consider it as a finite sum of linear polynomials as each $A_{ij}$ is realisable.
We let $\cA = (A_{ij}) \in \cR(\HH)^{d\times n}$ denote the $(d \times n)$-matrix of realisable sets encoding~\eqref{eq:strict+inequality}, and let $\cS_\succ(\cA) \subseteq \HH^d$ denote the solution set to~\eqref{eq:strict+inequality}.
We claim that Fourier-Motzkin elimination gives us a system of inequalities for the $i$th coordinate projection of the solution set
\[
\pi_{i}(\cS_\succ(\cA)) = \{(p_1, \dots, p_{i-1},p_{i+1}, \dots, p_d) \in \HH^{d-1} \mid (p_1,\dots,p_d) \in \cS_\succ(\cA)\} \, .
\]
% This will be used in two separate ways.
% In Section~\ref{sec:farkas}, we apply Fourier-Motzkin to inductively prove the Farkas lemma.
% In Section~\ref{sec:exterior}, we apply it to an inequality system in $n+d$ variables to recover a system in $d$ variables.

We show how to find an inequality system for $\pi_d(\cS_\succ(\cA))$: the process for other coordinates is analogous.
We partition $[n]$ into $J^{+}, J^-, J^\bullet$ and $J^\0$ based on whether the $d$th coefficient is positive, negative, a set or zero.
Moreover, we can scale the $j$th inequality so that $A_{dj}\in \{\1, -\1, \1 \boxplus -\1, \0\}$ without affecting the solution set; this means~\eqref{eq:strict+inequality} is equivalent to the inequality system
\begin{align} \label{eq:strict+inequality+mod}
\begin{tabular}{ccccc}
    $\left(\bigboxplus_{i=1}^{d-1} A_{ij} \odot X_i\right)$ & $\boxplus$ & $\1 \odot X_d$ & $\succ \0$ & $, \quad  j \in J^+ \, ,$ \\
    $\left(\bigboxplus_{i=1}^{d-1} A_{ij} \odot X_i\right)$ & $\boxplus$ & $-\1 \odot X_d$ & $\succ \0$ & $, \quad  j \in J^- \, ,$ \\
    $\left(\bigboxplus_{i=1}^{d-1} A_{ij} \odot X_i\right)$ & $\boxplus$ & $(\1 \boxplus -\1) \odot X_d$ & $\succ \0$ & $, \quad  j \in J^\bullet \, ,$ \\
    $\left(\bigboxplus_{i=1}^{d-1} A_{ij} \odot X_i\right)$ &  &  & $\succ \0$ & $, \quad  j \in J^\0 \, ,$ \\
\end{tabular}
\end{align}
To ease length of notation, we denote $J^{+\bullet} = J^+ \cup J^\bullet$ and $J^{-\bullet} = J^- \cup J^\bullet$.

\begin{proposition}[Strict Fourier-Motzkin]\label{prop:strict+fm}
    Let $\HH$ be an ordered stringent hyperfield whose ordering is dense.
    The $d$-th coordinate projection of $\cS_\succ(\cA)$ is
    \[
    \pi_d(\cS_\succ(\cA)) = \cS_\succ(\tilde{\cA})
    \]
    where $\tilde{\cA}$ induces the following system of inequalities in $d-1$ variables:
    \begin{equation}
    \label{eq:strict+inequality+d-1}\begin{split} 
\bigboxplus_{i=1}^{d-1} \tilde{A}_{i,jk} \odot X_i &\succ \0 \quad , \quad (j,k) \in J^{+\bullet} \times J^{-\bullet} \quad , \quad \tilde{A}_{i,jk} = A_{ij} \boxplus A_{ik}  \\
\bigboxplus_{i=1}^{d-1} A_{ij} \odot X_i &\succ \0 \quad , \quad j \in J^\0 \, .
\end{split}
\end{equation}
\end{proposition}
\begin{proof}
We refer to the inequalities in~\eqref{eq:strict+inequality+mod} of type $(+, -, \bullet, \0)$ respectively.
We first observe that we can replace each inequality of type $(\bullet)$ by one of type $(+)$ and one of type $(-)$.
Such an inequality can only have a solution if $\bigboxplus_{i=1}^{d-1} A_{ij} \odot X_i$ is a singleton, else we have the sum of two sets which must contain $\0$.
Using Lemma~\ref{lem:sign+set+split}, we deduce that $\left(\bigboxplus_{i=1}^{d-1} A_{ij} \odot X_i\right) \boxplus (\1 \boxplus-\1) \odot X_d$ has a solution if and only if the inequalities
\[
\begin{cases}
    \left(\bigboxplus_{i=1}^{d-1} A_{ij} \odot X_i\right) \boxplus \1 \odot X_d \succ \0\\
    \left(\bigboxplus_{i=1}^{d-1} A_{ij} \odot X_i\right) \boxplus -\1 \odot X_d \succ \0\\
\end{cases} \, 
\]
have a solution.
Repeating this for all inequalities of type $(\bullet)$ gives an inequality system with a solution if and only if~\eqref{eq:strict+inequality+mod} has a solution; consisting of inequalities of type $(+)$ for all $J^{+\bullet}$, inequalities of type $(-)$ for all $J^{-\bullet}$ and inequalities of type $(\0)$ for all $J^\0$.

Finally we eliminate $X_d$ in the following way.
Lemma~\ref{lem: dense+order} implies that for every pair of type $(+)$ and type $(-)$ inequalities $(j,k) \in J^{+\bullet} \times J^{-\bullet}$, there exists a common solution if and only if there is a solution to the inequality
\[
\bigboxplus_{i=1}^{d-1} \tilde{A}_{i,jk} \odot X_i \succ \0 \quad , \quad \tilde{A}_{i,jk} = A_{ij} \boxplus A_{ik} \, .
\]
Ranging over all $(j,k)$ gives the result.
Note that the upper bound of $n^2$ inequalities occurs when all $n$ inequalities are of type $(\bullet)$.
\end{proof}

\added{
\begin{example} \label{ex:fm}
    Consider the system of inequalities over $\RR \rtimes \ZZ$ defined by the columns of the matrix $\cA$:
    \[
    \begin{aligned}
        &(1): & \: -X_1 \boxplus X_2 \boxplus X_3 &\succ \0 \\
        &(2): & \: X_1 \boxplus -X_2 \boxplus X_3 &\succ \0 \\
        &(3): & \:-X_1 \boxplus -X_2 \boxplus -X_3 &\succ \0
    \end{aligned} \, , \quad\quad
    \cA := 
    \begin{blockarray}{cccc}
     1 & 2 & 3 &\\
    \begin{block}{(ccc)c}
     (-1,0) & (1,0) & (-1,0) & X_1  \\
     (1,0) & (-1,0) & (-1,0) & X_2 \\
     (1,0) & (1,0) & (-1,0) & X_3 \\
    \end{block} 
    \end{blockarray}
    \subseteq (\RR \rtimes \ZZ)^{3 \times 3} \, .
    \]
    We demonstrate how Fourier-Motzkin elimination can be used to compute $\cS_\succ(\cA)$ and solve this system.

    We first eliminate $X_3$ from our system, so our equations are partitioned into $J^+ = \{1,2\}$, $J^- = \{3\}$ and $J^\bullet = J^\0 = \emptyset$.
    By Proposition~\ref{prop:strict+fm}, the projected solution set $\pi_3(\cS_\succ(\cA))$ is equal to the solution set of the two-variable system determined by the matrix
    \begin{equation} \label{eq:first+fm+step}
    \tilde{\cA} := 
    \begin{blockarray}{ccc}
    (1,3) & (2,3) & \\
    \begin{block}{(cc)c}
    (-1,0) & (1,0) \boxplus (-1,0) & X_1 \\
    (1,0) \boxplus (-1,0) & (-1,0) & X_2  \\
    \end{block} 
    \end{blockarray}
    \subseteq \cR(\RR \rtimes \ZZ)^{2 \times 2} \, ,
    \end{equation}
    where $J^{+\bullet} \times J^{-\bullet} = \{(1,3), (2,3)\}$.
    Note that the coefficients of the corresponding inequalities are no longer singletons, but rather realisable sets.

    We next eliminate $X_2$ from our system.
    Keeping the same equation labels as \eqref{eq:first+fm+step}, our equations are partitioned into $J^\bullet = \{(1,3)\}$ and $J^- = \{(2,3)\}$, with both $J^+ = J^\0 = \emptyset$.
    Again applying Proposition~\ref{prop:strict+fm}, we obtain the projected solution set $\pi_2(\pi_3(\cS_\succ(\cA)))$ as the solution of the single variable system determined by the matrix
    \[
    \bar{\cA} := 
    \begin{blockarray}{ccc}
    ((1,3),(2,3)) & ((1,3),(1,3)) & \\
    \begin{block}{(cc)c}
    (-1,0) & (-2,0) & X_1 \\
    \end{block} 
    \end{blockarray}
    \subseteq \cR(\RR \rtimes \ZZ)^{1 \times 2} \, .
    \]

    Write $X_i = (k_i,g_i)$, we now work backwards to describe the solution set of our system.
    The two inequalities described by $\bar{\cA}$ simply imply that $(k_1,g_1) \prec \0$, hence $\pi_{2}(\pi_3(\cS_\succ(\cA))) = \{(k_1, g_1) \mid k_1 < 0\}$.
    Using this, we deduce that the two equations described by $\tilde{\cA}$ reduce to
    \begin{align*}
    (1,3) &\colon (-k_1,g_1) \boxplus (k_2,g_2) \boxplus (-k_2,g_2) = 
    \begin{cases}
    (-k_1, g_1) & g_1 \geq g_2 \\
    \{(k,g) \mid g < g_2, k \in \RR \} & g_1 < g_2 \quad (\#)
    \end{cases}\\
    (2,3) &\colon (k_1,g_1) \boxplus (-k_1,g_1) \boxplus (-k_2,g_2) = 
    \begin{cases}
    (-k_2, g_2) & g_2 \geq g_1 \\
    \{(k,g) \mid g < g_1, k \in \RR \} & g_2 < g_1 \quad (\#)
    \end{cases}
    \end{align*}
    The lines labelled by $(\#)$ contradict that the output must be positive, hence we deduce that $g_1 = g_2 = g$.
    Furthermore, the other case of $(2,3)$ gives us that $k_2 < 0$, hence $\pi_3(\cS_\succ(\cA)) = \{[(k_1, g), (k_2,g)] \mid k_1,k_2 < 0\}$.
    Finally, we deduce that our original system of equations reduces to
    \begin{align*}
        (1): & \: (-k_1,g) \boxplus (k_2,g) \boxplus (k_3,g_3) \succ \0 \\
        (2): & \: (k_1,g) \boxplus (-k_2,g) \boxplus (k_3,g_3) \succ \0 \\
        (3): & \:(-k_1,g) \boxplus (-k_2,g) \boxplus (-k_3,g_3) \succ \0
    \end{align*}
    If $g_3 > g$, equations $(1)$ and $(3)$ respectively imply $k_3 >0$ and $k_3< 0$, a contradiction.
    If $g_3 < g$, equations $(1)$ and $(2)$ respectively imply $k_1 - k_2 > 0$ and $k_2 - k_1 > 0$, a contradiction.
    Hence $g_3 = g$, and we get the description of the solution set
    \begin{equation}
    \label{eq:ex+solution+set}
    \cS_{\succ}(\cA) := \{[(k_1, g), (k_2,g), (k_3,g)] \in (\RR \rtimes \ZZ)^3 \mid k_1 - k_2 + k_3 > 0 \, , \, -k_1 + k_2 + k_3 > 0 \, , \, -k_1 -k_2 - k_3 > 0\} \, .
    \end{equation}
    Note that the previous inequalities $k_1, k_2 < 0$ we deduced are implied by these inequalities.
\end{example}
}

\begin{remark}
    When $\HH = \SS \rtimes G$, we can improve this further by replacing inequalities with set coefficients by two equations with singleton coefficients.
We can replace the inequality $\bigboxplus_{i=1}^d A_{i} \odot X_i \succ \0$ by the two equations
\begin{align*} \label{eq:finite+strict}
    \bigboxplus_{i=1}^{d} a^+_{i} \odot X_i \succ \0 \quad , \quad  \bigboxplus_{i=1}^{d} a^-_{i} \odot X_i \succ \0 
    \quad \text{ where }\begin{cases}
        a_i^+ = a_i^- = a &\text{ if } A_{i} = \{a\} \\
        a_i^+ = -a_i^- = a \in \HH^+ &\text{ if } A_{i} = a \boxplus -a
    \end{cases} \, .
\end{align*}
The proof is analogous to the proof over $\TT\RR$ given in~\cite[Theorem 4.12]{LV}, but the crucial detail is that the set $a \boxplus -a$ has a well-defined supremum and infimum, $a$ and $-a$.

The same is not true over $\FF \rtimes G$, as
\[
a \boxplus -a = (k_a, g_a) \boxplus (-k_a, g_a) = \left\{(\ell, h) \mid h < g_a, \ell \in \FF^+\right\} \, .
\]
If $G$ is dense, there is no largest $h$ such that $h < g_a$.
Even if it is not dense, there is no largest $\ell \in \FF^+$.
As such, there is no finite set of elements we can substitute $a \boxplus -a$ for.
\end{remark}

\subsection{Farkas lemma for stringent hyperfields} \label{sec:farkaslem}
We now return to $\HH$ being any ordered hyperfield.
We shall restrict to stringent hyperfields again for the main theorem of this section: the Farkas lemma.

\begin{definition}
Let $\cA = (A_{ij}) \in \cR(\HH)^{d\times n}$.
    The \emph{(positive) kernel of $\cA$} is the set
    \[
    \ker(\cA) = \SetOf{\blambda \in (\HH^+ \cup \{\0\})^n\setminus (\0, \dots, \0)}{\bigboxplus_{j=1}^n \lambda_j \odot A_{ij} \ni \0 \quad \forall i \in [d]} \, .
    \]
    The \emph{separation set of $\cA$} is the set 
    \[
    \sep(\cA) = \SetOf{\balpha \in \HH^d}{\bigboxplus_{i=1}^d \alpha_i \odot A_{ij} \subseteq \HH^+ \quad \forall j \in [n]} \, .
    \]
It follows from the definitions that one can scale the rows of $\cA$ without altering $\ker(\cA)$, and positively scale the columns of $\cA$ without altering $\sep(\cA)$.
\end{definition}

When $\cA$ is a matrix of singletons, it corresponds to a set of $n$ points in $d$-dimensional space.
In this case, we have the following geometric intuition behind these definitions. 
\begin{proposition} \label{prop:ker+sep+geometry}
    Let $\HH$ be an ordered hyperfield where $T = \{\bp_1, \dots, \bp_n\} \subseteq \HH^d$ and $\bq \in \HH^d$.
    Consider the $(d+1) \times (n+1)$ matrix
    \begin{equation} \label{eq:homog}
    \cA = \begin{pmatrix}
    \bp_1 & \cdots & \bp_n & -\bq \\
    \1 & \cdots & \1 & -\1 \\
    \end{pmatrix} \, .
    \end{equation}
    \begin{enumerate}
        \item $\ker(\cA) \neq \emptyset$ if and only if $\bq \in \conv(T)$,
        \item $\sep(\cA) \neq \emptyset$ then there exists an open halfspace $\HS(\phi)$ such that $\conv(T) \subseteq \HS(\phi)$ and $\bq \notin \HS(\phi)$.
    \end{enumerate}
\end{proposition}
\begin{proof}
    For the first claim, $\ker(\cA) \neq \emptyset$ if and only if there exists $\blambda \in (\HH^+ \cup \{\0\})^{n+1}$ such that
    \begin{align*}
    \0 &\in \left(\bigboxplus_{j=1}^n \lambda_j \odot p_{ij}\right) \boxplus \lambda_{n+1} \odot -q_i  & &\Longleftrightarrow & \lambda_{n+1} \odot q_i &\in \bigboxplus_{j=1}^n \lambda_j \odot p_{ij} \quad , \quad \forall i \in [d] \, \\
    \0 &\in \left(\bigboxplus_{j=1}^n \lambda_j \right) \boxplus -\lambda_{n+1} & &\Longleftrightarrow & \lambda_{n+1} &\in \bigboxplus_{j=1}^n \lambda_j \, .
    \end{align*}
    Note that $\lambda_{n+1} \neq \0$, as this would imply all other $\lambda_j = \0$.
    Scaling the latter equations by $\lambda_{n+1}^{-1}$ precisely gives the condition that $\bq \in \conv(T)$.

    For the second claim, $\sep(\cA) \neq \emptyset$ implies there exists $\balpha \in \HH^{d+1}$ such that
    \begin{align*}
    \left(\bigboxplus_{i=1}^d \alpha_i\odot p_{ij}\right) \boxplus \alpha_{d+1} &\subseteq \HH^+ \quad \forall j \in [n] \\
    \left(\bigboxplus_{i=1}^d \alpha_i\odot q_i\right) \boxplus \alpha_{d+1} &\subseteq \HH^- \, .
    \end{align*}
    Therefore the affine polynomial $\phi = (\bigboxplus_{i=1}^d \alpha_i \odot X_i) \boxplus \alpha_{d+1}$ defines a separating open halfspace.
\end{proof}
\begin{remark}
The second part of Proposition~\ref{prop:ker+sep+geometry} is not if and only if, as we just require a polynomial $\phi$ such that $\phi(\bq) \nsubseteq \HH^+$.
For stringent hyperfields, this includes $\phi$ such that $\0 \in \phi(\bq)$.
For non-stringent hyperfields, this includes many more polynomials as $\phi(\bq)$ could have positive and negative elements, see Example~\ref{ex:Non-stringent-quo-halfspaces}.
\end{remark}

\begin{proposition}[Weak Farkas]\label{prop:weak+duality}
    Let $\HH$ be an ordered hyperfield and $\cA \in \cR(\HH)^{d\times n}$.
    At most one of $\ker(\cA)$ and $\sep(\cA)$ are non-empty.
\end{proposition}
\begin{proof}
    Suppose $\ker(\cA)$ is non-empty, then there exists $\blambda \in (\HH^+ \cup \{\0\})^n$ such that $\0 \in \bigboxplus_{j=1}^n \lambda_j \odot A_{ij}$ for all $j \in [n]$.
    Assume $\sep(\cA)$ is also non-empty and pick $\balpha \in \sep(\cA)$.
    Then for all $j \in [n]$ we have $\bigboxplus_{i=1}^d \alpha_i \odot A_{ij} \subseteq \HH^+$.
    As at least one $\lambda_j \neq \0$ and $\HH^+$ is closed under addition and multiplication, we have
    \begin{align*}
        \HH^+ \supseteq \bigboxplus_{j=1}^n \lambda_j \odot \left(\bigboxplus_{i=1}^d \alpha_i \odot A_{ij}\right)  
        = \bigboxplus_{i=1}^d \bigboxplus_{j=1}^n \alpha_i \odot \lambda_j \odot A_{ij} 
        = \bigboxplus_{i=1}^d \alpha_i \odot \underbrace{\left(\bigboxplus_{j=1}^n \lambda_j \odot A_{ij}\right)}_{\0 \, \in} \ni \0 \, .
    \end{align*}
    This is a contradiction, therefore $\sep(\cA)$ must be empty.
\end{proof}
We emphasise that the previous proposition holds for any ordered hyperfield.
However, to prove the strong version we require much stricter conditions.

\begin{theorem}[Farkas Lemma]\label{thm:strong+duality}
    Let $\HH = \SS \rtimes G$ be a stringent ordered hyperfield with dense ordering \added{and $\cA \in \cR(\HH)^{d\times n}$}.
    Then exactly one of $\ker(\cA)$ and $\sep(\cA)$ are non-empty.
\end{theorem}
\begin{proof}
    We proceed by induction on $d$, the number of variables in the underlying system.
    Note that Proposition~\ref{prop:weak+duality} implies that if $\sep(\cA)$ is non-empty then $\ker(\cA)$ is empty, therefore we assume that $\sep(\cA)$ is empty.
    
    When $d=1$, $\sep(\cA)$ is empty if and only if not all $A_j$ are simultaneously positive or negative.
    If any $A_j$ is a set, then $\0 \in \lambda_j \odot A_j$ for any $\lambda_j \neq \0$, and so $\ker(\cA)$ is non-empty.
    If all $A_j$ are singletons, there exists two elements $A_j, A_k \in \cA$ with opposite signs, in which case 
    \[
    \0 \in A_j\odot A_k \boxplus -A_k \odot A_j \, ,
    \]
    i.e. $\ker(\cA)$ is non-empty.
    
    For the general case, without loss of generality we scale the elements of $\cA$ such that $A_{dj} \in \{\1,-\1,\0,\1 \boxplus-\1\}$ for all $j \in [n]$.
    By Proposition~\ref{prop:strict+fm}, the system of inequalities induced by $\cA$ has no solution if and only if the system of inequalities induced by $\tilde{\cA}$ has no solution.
    By induction, this implies that there exists some $\tilde{\blambda} \in \ker(\tilde{\cA})$ whose entries are indexed by $(J^{+\bullet} \times J^{-\bullet}) \cup J^\0$.
    We define $\blambda \in (\HH^+ \cup \{\0\})^n$ as
    \[
    \lambda_j = \begin{cases}
        \bigboxplus_{k \in J^{-\bullet}} \tilde{\lambda}_{(j,k)} & j \in J^+ \\
        \bigboxplus_{k \in J^{+\bullet}} \tilde{\lambda}_{(k,j)} & j \in J^- \\
        \bigboxplus_{k \in J^{-\bullet}} \tilde{\lambda}_{(j,k)} \boxplus \bigboxplus_{k \in J^{+\bullet}} \tilde{\lambda}_{(k,j)} & j \in J^\bullet \\
        \tilde{\lambda}_j & j \in J^\0 
    \end{cases} \, ,
    \]
    and claim that $\blambda \in \ker(\cA)$.
    We note that for $1 \leq i \leq d-1$, we have
    \begin{align*}
        \0 \in&\bigg(\mspace{-10mu}
        \bigboxplus_{\substack{(j,k) \in \\ J^{+\bullet} \times J^{-\bullet}}} \mspace{-10mu}
        \tilde{\lambda}_{(j,k)} \odot \tilde{A}_{i,jk}\bigg) \boxplus \bigg(\bigboxplus_{j \in J^\0}\tilde{\lambda_j} \odot A_j \bigg) \\
        =&
        \bigg(\mspace{-10mu}\bigboxplus_{\substack{(j,k) \in \\ J^{+\bullet} \times J^{-\bullet}}} \mspace{-10mu}\tilde{\lambda}_{(j,k)} \odot (A_{ij} \boxplus A_{ik})\bigg) \boxplus \bigg(\bigboxplus_{j \in J^\0}\tilde{\lambda_j} \odot A_j\bigg) \\
        =&\bigg(\bigboxplus_{j \in J^+} \Big(\bigboxplus_{k \in J^{-\bullet}} \tilde{\lambda}_{(j,k)}\Big) \odot A_{ij}\bigg) \boxplus
        \bigg(\bigboxplus_{j \in J^-} \Big(\bigboxplus_{k \in J^{+\bullet}} \tilde{\lambda}_{(k,j)}\Big) \odot A_{ij}\bigg) \boxplus \\
        & \mspace{50mu}\bigg(\bigboxplus_{j \in J^\bullet} \Big(\bigboxplus_{k \in J^{-\bullet}} \tilde{\lambda}_{(j,k)} \boxplus \bigboxplus_{k \in J^{+\bullet}} \tilde{\lambda}_{(k,j)}\Big) \odot A_{ij}\bigg) \boxplus \bigg(\bigboxplus_{j \in J^\0}\tilde{\lambda_j} \odot A_j\bigg) \\
        =& \bigboxplus_{j=1}^n \lambda_{j} \odot A_{ij} \, .
    \end{align*}
    Remains to check that $\0 \in \bigboxplus_{j} \lambda_j \odot A_{dj}$.
    \begin{align*}
        \bigboxplus_{j=1}^n \lambda_j \odot A_{dj} &= \bigg(\bigboxplus_{j \in J^+} \lambda_j\bigg) \boxplus \bigg(\bigboxplus_{j \in J^-} -\lambda_j\bigg) \boxplus \bigg(\bigboxplus_{j \in J^\bullet} (\lambda_j\boxplus - \lambda_j)\bigg) \\       
        =&\bigg(\mspace{-10mu}\bigboxplus_{\substack{(j,k) \in  \\ J^+ \times J^{-\bullet}}}\mspace{-10mu} \tilde{\lambda}_{(j,k)}\bigg) \boxplus
        \bigg(\mspace{-10mu}\bigboxplus_{\substack{(j,k) \in \\ J^{+\bullet}\times J^-}}\mspace{-10mu} -\tilde{\lambda}_{(j,k)}\bigg) \boxplus
        \bigg(\mspace{-10mu}\bigboxplus_{\substack{(j,k) \in \\ J^\bullet \times J^{-\bullet}}}\mspace{-10mu} \tilde{\lambda}_{(j,k)}\boxplus - \tilde{\lambda}_{(j,k)}\bigg) \boxplus 
        \bigg(\mspace{-10mu}\bigboxplus_{\substack{(j,k) \in \\ J^{+\bullet} \times J^\bullet}}\mspace{-10mu} \tilde{\lambda}_{(j,k)}\boxplus -\tilde{\lambda}_{(j,k)}\bigg)  \\
        =& \bigboxplus_{\substack{(j,k) \in  \\ J^{+\bullet} \times J^{-\bullet}}}\mspace{-10mu} (\tilde{\lambda}_{(j,k)} \boxplus -\tilde{\lambda}_{(j,k)}) \boxplus \bigboxplus_{\substack{(j,k) \in  \\ J^+ \times J^\bullet}}\mspace{-10mu} \tilde{\lambda}_{(j,k)} \boxplus \bigboxplus_{\substack{(j,k) \in  \\ J^\bullet \times J^-}}\mspace{-10mu} -\tilde{\lambda}_{(j,k)}
    \end{align*}
We finally note that over $\SS \rtimes G$, we have $a \boxplus -a \boxplus a = a \boxplus -a$.
Using this identity, we can kill the second and third summation in the previous equation.
This gives
\begin{equation} \label{eq:d+kernel}
\0 \in \bigboxplus_{\substack{(j,k) \in  \\ J^{+\bullet} \times J^{-\bullet}}}\mspace{-10mu} (\tilde{\lambda}_{(j,k)} \boxplus -\tilde{\lambda}_{(j,k)}) = \bigboxplus_{j=1}^n \lambda_j \odot A_{dj} \, .
\end{equation}
\end{proof}

\begin{proof}[Proof of Theorem~\ref{thm:separation}]
When $\HH = \SS\rtimes G$, combining Theorem~\ref{thm:strong+duality} with Proposition~\ref{prop:ker+sep+geometry} gives the result.

Let $\HH = \FF \rtimes G$; note that an ordering on an ordered field is always dense, and so the ordering on $\HH$ is also dense. 
    Let $\cA$ be as the matrix \eqref{eq:homog}.
    We observe that the proof of Theorem~\ref{thm:strong+duality} holds for $\HH$ up until final paragraph, where we have terms labelled by  $J^+ \times J^\bullet$ and $J^\bullet \times J^-$ that in general will not cancel.
    We say $\cA$ is \emph{sufficiently generic} if there exists an ordering of the rows of $\cA$ such that applying Fourier-Motzkin elimination to the rows in this order results in $J^\bullet = \emptyset$ at every iteration, i.e. we do not get cancellation between terms.
    As $J^\bullet$ remains empty, equation~\eqref{eq:d+kernel} holds for each iteration, and so exactly one of $\ker(\cA)$ and $\sep(\cA)$ holds.
    Applying Proposition~\ref{prop:ker+sep+geometry} gives the desired result.
\end{proof}

\added{
\begin{example}
    Let $\HH := \RR \rtimes \ZZ$ and recall from Example~\ref{ex:RxZ2} the convex set $\conv(\bp, -\bp) \subseteq \HH^2$ where $\bp = [(-1,0), (1,0)] \in \HH^2$, and the point $\bq = [(1,0), (1,0)]$.
    As we saw in Example~\ref{ex:RxZ+sep}, $\bq$ is not contained in $\conv(\bp, -\bp)$ and can be separated by an open halfspace.
    We can find this open halfspace by Fourier-Motzkin elimination.
    Recall the system of inequalities over $\RR \rtimes \ZZ$ from Example~\ref{ex:fm}.
    Its defining matrix $\cA$ is precisely the matrix \eqref{eq:homog}, hence each solution to the inequalities gives a separating halfspace by Proposition~\ref{prop:ker+sep+geometry}.
    We calculated in Example~\ref{ex:fm} that this system has the non-empty solution set $\cS_\succ(\cA)$ described in \eqref{eq:ex+solution+set}.
    In particular, the point $[(-1,0),(-1,0),(k,0)]$ is contained in $\cS_\succ(\cA)$ for $k \in (0,2)$, which corresponds to the separating halfspace in Example~\ref{ex:RxZ2}.

    We note that this matrix $\cA$ is not sufficiently generic, as there is a step within the Fourier-Motzkin elimination algorithm in Example~\ref{ex:fm} where $J^\bullet \neq \emptyset$.
    It is also easy to check that any reordering of the rows would still result in $J^\bullet \neq \emptyset$ at some step.
    However, if we perturb $\bq$ to $\bq' = [(1+\alpha, 0), (1 + \beta, 0)]$ for $\alpha, \beta \in \RR \setminus \{-2,-1,0\}$, the resulting matrix
    \[
    \cA' =  \begin{pmatrix}
    \bp_1 & \cdots & \bp_n & -\bq' \\
    \1 & \cdots & \1 & -\1 \\
    \end{pmatrix}
    = \begin{pmatrix}
    (-1,0) & (1,0) & (-1-\alpha,0)   \\
     (1,0) & (-1,0) & (-1-\beta,0)  \\
     (1,0) & (1,0) & (-1,0)  \\
    \end{pmatrix} \, .
    \]
    is sufficiently generic.
    Recomputing Example~\ref{ex:fm} with this matrix, it is easy to check that the Fourier-Motzkin algorithm has $J^\bullet = \emptyset$ at every step.
\end{example}
}

%\ben{Is there a way to prove this for non-generic $\FF \rtimes G$ by showing there is only finitely many bad points, and dealing with them in some other way?}

\added{
\begin{remark}
    An ordered field $\FF$ can be written as the stringent hyperfield $\FF \cong \FF \rtimes G$ where $G$ is the trivial group.
    Furthermore, any matrix $\cA$ over this ordered field satisfies our definition of sufficiently generic: $J^\bullet = \emptyset$ at every step of Fourier-Motzkin elimination, as additive inverses will genuinely cancel to zero rather than leaving a set.
    As such, Theorem~\ref{thm:separation} does recover the Hyperplane Separation Theorem for arbitrary ordered fields.
    Similarly, the observation in its proof that Theorem~\ref{thm:strong+duality} holds for sufficiently generic matrices over $\FF \rtimes G$ allows us to also recover the Farkas Lemma for arbitrary ordered fields.
\end{remark}
}

\section{Applications and further questions}
We end with some avenues for further research and applications of hyperfield convexity.

\subsection{Vectors of matroids over ordered hyperfields}

Much of the recent interest in hyperfields has been motivated by the introduction of \emph{matroids over hyperfields} as pioneered by~\cite{BB}.
Key examples include $\SS$-matroids which are equivalent to \emph{oriented matroids}, and $\TT\RR$-matroids which are equivalent to \emph{oriented valuated matroids}.
We will not formally define them, but instead give informal intuition and motivation where required; see~\cite{BB, LA} for details.

Given a hyperfield $\HH$, a $\HH$-matroid $\cM$ can be defined in terms of its \emph{cocircuits} $\cC^*(\cM) \subseteq \HH^d$, a minimal collection of $d$-tuples satisfying an elimination axiom.
The set of \emph{vectors} $\cV(\cM) \subseteq \HH^d$ of $\cM$ is the set of elements orthogonal to the cocircuits of $\cM$, i.e.
\[
X \in \cV(\cM) \, \Leftrightarrow  \, \0 \in \bigboxplus_{i=1}^d X_i \odot Y_i \quad \forall Y \in \cC^*(\cM) \, .
\]
This definition of vectors is entirely dependent on the cocircuits of $\cM$.
An independent axiom system for vectors was given in~\cite{LA} in terms of reduced echelon forms of support bases that holds for all hyperfields $\HH$.
However, it is very distinct from existing vector axiom systems in the literature, particularly for vectors of oriented matroids.
We recall two examples of vector axioms \added{for} $\SS$ and $\TT\RR$, before considering how one may generalise this to matroids over all ordered hyperfields.

\begin{example}\label{ex:om+composition}
The sign hyperfield has the following \emph{composition operation} $\circ$ on it defined as follows:
\[
\circ \colon \SS \times \SS \rightarrow \SS \, , \quad a \circ b = \begin{cases}
a & a \neq \0 \\
b & \text{otherwise}
\end{cases} \, .
\]
Composition extends component-wise to $\SS^d$.
Moreover, it allows us to characterise vectors of oriented matroids~\cite[Theorem 3.7.5]{BLSWZ}: a subset $\cV \subseteq \SS^d$ is the set of vectors of an $\SS$-matroid if and only if
\begin{enumerate}[label=(V\arabic*)]
    \item $\0 \in \cV$, \label{cond:om+zero}
    \item If $X \in \cV$ then $aX \in \cV$ for all $a \in \SS$, \label{cond:om+scalar}
    \item If $X, Y \in \cV$ then $X \circ Y \in \cV$, \label{cond:om+composition}
    \item If $X, Y \in \cV$ with $X_i = -Y_i$, there exists $Z \in \cV$ with $Z \in X \boxplus Y$ and $Z_i = \0$. \label{cond:om+elimination}
\end{enumerate}
Let us give some geometric intuition behind these axioms in terms of real linear spaces.
Let $\cL \subseteq \RR^d$ be a linear space, the image of this linear space in the sign map is the set of vectors $\cV(M) = \sgn(\cL)$ of some oriented matroid $\cM$.
Axioms~\ref{cond:om+zero} and~\ref{cond:om+scalar} follow immediately from linear spaces being closed under scalar multiplication, while axioms~\ref{cond:om+composition} and~\ref{cond:om+elimination} can be reframed in terms of convexity of real linear spaces.

Consider $\bx, \by \in \cL$ and consider the convex line segment connecting them.
Axiom~\ref{cond:om+composition} says composition describes how the sign data changes if we take an infinitesimal step from $\bx$ to $\by$ along this line segment.
Explicitly, for some sufficiently small $\epsilon > 0$, the point $(1-\epsilon)\bx + \epsilon\by$ has sign
\[
\sgn((1-\epsilon)\bx + \epsilon\by) = \sgn(\bx) \circ \sgn(\by) \, .
\]
As such, we can consider $X \circ Y$ as the perturbation of $X$ in the direction of $Y$, and that axiom~\ref{cond:om+composition} states this must stay within $\cV$.
Axiom~\ref{cond:om+elimination} states that if a coordinate hyperplane separates $\bx$ and $\by$, the line segment between them must intersect that hyperplane.
Explicitly, if $\sgn(x_i) = -\sgn(y_i)$, then setting $\lambda = |x_i|/(|x_i| + |y_i|) \in \RR_{> 0}$, we have $\bz = (1-\lambda)\bx + \lambda \by \in \cL$ with $z_i = 0$ and $\sgn(z_j) \in \sgn(x_j) \boxplus \sgn(y_j)$ for all other coordinates.
\end{example}

\begin{example}\label{ex:trop+composition}
Define the operation
\[
    \lhd \colon \TT\RR \times \TT\RR \rightarrow \TT\RR \quad , \quad a \lhd b = \begin{cases}
        a & |a| = |b| \\ a \boxplus b & \text{otherwise}
    \end{cases}
\]
and extend componentwise to a composition operation on $\TT\RR^d$.
\cite{LS} recently showed that this leads to a vector axiom system for $\TT\RR$-matroids analogous to $\SS$-matroids, where $\lhd$ plays the role of the composition.
The intuition for this can again be viewed as perturbing vectors over a field, namely $\RR[[t^\RR]]$ (or Hahn series over any ordered field).
Let $\bx, \by \in \RR[[t^\RR]]^d$ be vectors in some linear space $\cL \subseteq \RR[[t^\RR]]^d$, we again consider a `small' step $(1-\epsilon)\bx + \epsilon \by$ along the line segment between them.
However, $\RR[[t^\RR]]$ is a non-Archimedean field and so which $\epsilon$ to perturb by is slightly more subtle.
We let $\epsilon = 1/N > 0$ for some arbitrarily large $N \in \NN$: while this is an arbitrarily small real number, it is still larger than any Hahn series with negative leading power.
As such, it is a straightforward check to see that
\[
\sval((1- \epsilon)\bx + \epsilon\by) = \sval(\bx) \lhd \sval(\by) \, .
\]
\end{example}

%Note that a notion of axiom~\ref{cond:om+elimination} can be defined for arbitrary hyperfields, however its known that this does not hold in general; see~\cite{LA}.
It is not known how to generalise the composition axiom~\ref{cond:om+composition} to arbitrary hyperfields.
As a partial step towards solving this problem,~\cite{LA} defines a composition operation on a general hyperfield $\HH$ as a hyperoperation $\circ$ on $\HH^d$ such that
    \begin{enumerate}
        \item For every $X, Y \subseteq \HH^d$ and $Z \in X \circ Y$, we have $\underline{Z} = \underline{X} \cup \underline{Y}$,
        \item If $X, Y \in V(\phi)$ for some linear polynomial $\phi$, then $X \circ Y \subseteq V(\phi)$.
    \end{enumerate}
Both of the operations seen in Examples~\ref{ex:om+composition} and \ref{ex:trop+composition} are composition operations, and their definition is underpinned by the convexity of their corresponding linear spaces.
Motivated by this, we ask whether it is possible to define composition operations for matroids over ordered hyperfields using the language of convexity.

\begin{question}
    Can we define a composition operation $\circ_\HH$ for any ordered hyperfield $\HH$ where $\bx \circ_\HH \by$ is a collection of points on the convex line segment between $\bx$ and $\by$?
\end{question}

Note that while composition was a single-valued operation for both $\SS$ and $\TT\RR$, we will likely need to move to multi-valued operations for non-stringent hyperfields.

\subsection{Closed halfspace separation for stringent polyhedra}\label{sec:exterior}
Ordinary polyhedra can be characterised as the intersection of finitely many closed halfspaces.
This is the first step to many of the nice properties that ordinary polyhedra enjoy, such as faces and face lattices.
As such, it is a natural property to want. 

Over $\TT\RR$,~\cite{LV} showed that signed tropical polyhedra are the intersection of finitely many closed halfspaces.
However, their proof uses a combination of Fourier-Motzkin elimination of non-strict inequalities and technical lemmas that rely on $G = \RR$ being a complete ordered abelian group i.e. every bounded subset of $G$ has a well-defined supremum and infimum.
Many technical lemmas that work for strict Fourier-Motzkin elimination break down for non-strict over hyperfields of the form $\FF \rtimes G$.
Moreover, the only complete ordered abelian groups are $\ZZ$ and $\RR$, and so their proof does not extend to other hyperfields of the form $\SS \rtimes G$.

An alternative approach is to prove separation results over $\HH = \FF/U$ by lifting to the ordered field $\FF$ and separating there.
As a brief demonstration of this method, we show one can separate convex sets over the sign hyperfield with closed halfspaces; recall that this was not possible with open halfspaces.

\begin{proposition}
    Let $\conv(T) \subseteq \SS^d$ be a (finitely generated) convex set and $\bp \notin \conv(T)$.
    There exists a closed halfspace $\cHS(\phi) \subseteq \SS^d$ such that $\conv(T) \subseteq \cHS(\phi)$ and $\bp \notin \cHS(\phi)$.

    Therefore, every (finitely generated) convex set in $\SS^d$ is the intersection of (finitely many) closed halfspaces.
\end{proposition}
\begin{proof}
    Note that there only finitely many points in $\SS^d$ and so all convex sets over $\SS^d$ are finitely generated.
    As such, we can assume that $T = \conv(T)$.
    
    Consider the quotient map $\tau \colon \RR \rightarrow\SS$.
    Let $S = \conv(\Tilde{T}) \subseteq \RR^d$, where $\Tilde{T}$ is a choice of lifts $\Tilde{\bq} \in \tau^{-1}(\bq)$ for each $\bq \in T$; observe that $S$ is a convex polyhedral set.
    Moreover, by the assumption that $T = \conv(T)$ coupled with Proposition~\ref{prop:cone+containment}, we have that $\tau(S) = \conv(T)$.
    Let $P = \tau^{-1}(\bp) \subseteq \RR^d$; by Lemma~\ref{lem:pre-conv-conv} this is a convex set.
    Moreover, it is relatively open: it is full-dimensional in the linear space $L = \SetOf{\bx \in \RR^d}{x_i = 0 \text{ if } p_i = \0}$ and cut out by open halfspaces $\SetOf{\bx \in \RR^d}{\pm x_j > 0}$ where $p_j = \pm \1$.
    By~\cite[Theorem 20.2]{Rockafellar}, we can always strictly separate a polyhedral and relatively open convex set over $\RR$.
    This implies there exists $\cHS(\psi) \subseteq \RR^d$ where $\psi \in \RR[\bX]$ such that $S \subseteq \cHS(\psi)$ but $P \cap \cHS(\psi) = \emptyset$.

    Let $\psi = \tilde{c}_0 + \tilde{c}_1X_1 + \cdots +\tilde{c}_dX_d$ with  $\phi$ its pushforward $\phi = \tau_*(\psi) = c_0 \boxplus c_1\odot X_1 \boxplus \cdots \boxplus c_d\odot X_d$.
    We first note that using Theorem~\ref{clsd_hs_eql}, having $S \subseteq \cHS(\psi)$ implies that
    \[
    \conv(T) = \tau(S) \subseteq \tau(\cHS(\psi)) = \cHS(\phi) \, .
    \]
    It remains to show that $\bp \notin \cHS(\phi)$.

    Suppose there exists $i \in [d]$ such that $c_i \odot p_i = \1$.
    We construct an element of the preimage $\Tilde{\bp} \in P$ as follows:
    \[
    \Tilde{\bp}_i = \frac{(d+1)|\Tilde{c}_0|}{|\Tilde{c}_i|} \quad , \quad \Tilde{\bp}_j = \sgn(p_j)\frac{|\Tilde{c}_0|}{|\Tilde{c}_j|} \, , \, \forall j \neq i \, .
    \]
    Then evaluating $\psi$ at $\tilde{\bp}$ gives
    \begin{align*}
        \psi(\tilde{\bp}) = (d+1)|\tilde{c}_0| + \sum_{j \neq i} \sgn(c_j)\sgn(p_j)|\tilde{c}_0| + \tilde{c}_0 \geq |\tilde{c}_0| \geq 0 \, ,
    \end{align*}
    implying that $\tilde{\bp} \in \cHS(\psi)$.
    This is a contradiction, and so $c_i \odot p_i \neq \1$ for all $i$.
    This implies $\phi(\bp)$ is a sum of $\0$ and $-\1$, at least one of which is $-\1$, and so $\bp \notin \cHS(\phi)$.
\end{proof}

Generalising this method to other stringent hyperfields is currently out of reach.
While the lifted sets are convex, they do not appear to have any nice structure e.g. semialgebraic.
It is not possible to separate arbitrary convex sets over $\FF$ (see~\cite{RR:91}), and so it remains unknown whether these sets are separable.

\subsection{Applications of convexity over \texorpdfstring{$\RR \rtimes \RR$}{R x R}}
As discussed in the introduction, one motivation for tropical convexity has been its applications to complexity of linear programming.
Explicitly, one can study families of linear programs over (generalized) real Puiseux series $\RR[[t^{\RR}]]$ by applying the valuation map $\val$ and studying a single linear program over $\TT$.
As discussed in Example~\ref{ex:tropical+quotients}, one can enrich this valuation map to recall additional data, leading to the signed valuation $\sval$ and the fine valuation $\fval$.
These maps are related by the following commutative diagram
\[
\begin{tikzcd}[column sep=huge]
    \RR[[t^\RR]] \arrow[r, "\fval"] \arrow[rd, "\sval"] \arrow[rdd, "\val"'] & \RR \rtimes \RR \ar[d]\\
    & \TT\RR \ar[d]\\
    & \TT
\end{tikzcd}
\]
where the homomorphism $\RR \rtimes \RR \rightarrow \TT\RR$ sends $(c, g)$ to $(\sgn(c),g)$, and
the homomorphism $\TT\RR \rightarrow \TT$ sends $(\sgn(c),g)$ to $g$.

With these refined valuation maps, this opens the door to studying more refined families of linear programs over $\RR[[t^\RR]]$ by studying a single linear program over $\TT\RR$ or $\RR \rtimes \RR$.
In particular, these hyperfields give greater control over cancellation of terms over the Puiseux series.
This was a key motivation to the study of convexity over $\TT\RR$~\cite{LV,LS}, but there has been no consideration of $\RR \rtimes \RR$ until now.
We believe a systematic study of convexity over this hyperfield may help lead to developments of convexity over $\TT$ and $\TT\RR$, as well as possible applications to convexity and linear programming complexity over $\RR[[t^\RR]]$.

\appendix
\section{Order relations on hyperfields} \label{sec:order+relations} 
%We include a discussion on the relationship between orderings and order relations on hyperfields.
Recall from Example~\ref{ex:field+ordering} that over a field $\FF$, its orderings $\FF^+$ are in bijection with strict total order relations $\prec$ that are compatible with the field operations.
The same does not hold for hyperfields, and we use this section to detail the ways in which this correspondence breaks down.
The message one should take away is that orderings are the correct objects to work with on hyperfields, not order relations.
For a general reference for results in order theory, see~\cite{BS:16}.

Given an ordering $\HH^+$, we define the associated relation $\prec_{\HH^+}$ by 
\[
a \prec_{\HH^+} b \, \Leftrightarrow \, b \boxplus -a \subseteq \HH^+ \, .
\]
This is a strict partial order on $\HH$ but is only a total order when the hyperfield is stringent.

\begin{proposition} \label{prop:partial+order}
    Let $\HH$ be an ordered hyperfield with ordering $\HH^+$, then $\prec_{\HH^+}$ is a strict partial order.
    Moreover, if $\HH$ is stringent then $\prec_{\HH^+}$ is a strict total order.
\end{proposition}
\begin{proof}
    \item (Irreflexive) For any $a \in \HH$, we have $a \nprec a$ as $\0 \in a \boxplus -a \nsubseteq \HH^+$.
    \item (Asymmetric) If $a \prec b$ then $a \boxplus -b = -(b \boxplus -a) \subseteq -\HH^+ = \HH^-$, implying that $b \nprec a$.
    \item (Transitive) If $a \prec b$ and $b \prec c$, then $c\boxplus -a \subseteq (c\boxplus -b) \boxplus (b \boxplus -a) \subseteq \HH^+$, implying that $a \prec c$.
    \item (Total) Let $a,b \in \HH$ where $\HH$ is stringent.
    If both $a \nprec b$ and $b \nprec a$, then $b \boxplus -a$ must be a set overlapping both $\HH^+$ and $\HH^-$.
    As the only sets that occur in a stringent hyperfield are $c \boxplus -c$, this implies $a = b$.
\end{proof}

\begin{example}\label{ex:signed+tropical+order}
Consider the signed tropical hyperfield $\TT\RR$ with its unique ordering $\TT\RR^+ = \{(1,b) \mid b \in \RR\}$.
The order relation $\prec_{\TT\RR^+}$ is a strict total order on $\TT\RR$, equivalently defined as
\begin{equation}\label{eq:signed+tropical+order}
(a_1,b_1) \prec_{\TT\RR^+} (a_2,b_2) \, \Leftrightarrow \, \begin{cases}
    a_1 = -1 \text{ and } a_2 = 1 \\
    b_1 < b_2 \text{ and } a_1 = a_2 = 1 \\
    b_1 > b_2 \text{ and } a_1 = a_2 = -1 \\
\end{cases} \, .
\end{equation}
This is exactly the order relation on the signed tropical semiring~\cite{AGG:14,LV}.
%This is a special case of the order relation on $\SS \rtimes G$ where $G = \RR$.
\end{example}

Even if $\HH$ is not stringent, we always have $a \prec_{\HH_+} b$ for $a \in \HH^-$ and $b \in \HH^+$.
However, the following example shows that these may be the only comparable elements.

\begin{example}\label{ex:rational+quotient}
    Consider the field $\QQ$ and the quotient hyperfield $\HH = \QQ/(\QQ^\times)^2$ where $(\QQ^\times)^2$ is non-zero squares of $\QQ$.
    This hyperfield has a unique ordering, namely $\HH^+ = \QQ^+/(\QQ^\times)^2$.
    For any $a, b \in \QQ^+$, there exists squares $m, n \in (\QQ^\times)^2$ such that $m > a/b$ and $n > b/a$.
    As a result, these elements are incomparable under $\prec_{\HH^+}$ as
    \begin{align*}
        a\cdot 1 + (-b)\cdot m < a - \frac{a}{b}\cdot b = 0 \quad &\Rightarrow \quad b\cdot(\QQ^\times)^2 \nprec_{\HH_+} a\cdot(\QQ^\times)^2 \, , \\
        b\cdot 1 + (-a)\cdot n < b - \frac{b}{a}\cdot a = 0 \quad &\Rightarrow \quad a\cdot(\QQ^\times)^2 \nprec_{\HH_+} b\cdot(\QQ^\times)^2 \, .
    \end{align*}
    A similar argument works for $a, b \in \QQ^-$.
    As a result, the only comparable pairs of elements of $\HH$ are $\bar{a} \in \HH^-$ and $\bar{b} \in \HH^+$.
\end{example}

We note that $\prec_{\HH^+}$ is not compatible with addition even in the stringent case, as distinct elements can become equal after addition with another element.
For example, over $\TT\RR$
\[
(1,0) \prec_{\HH^+} (1,1) \quad \nRightarrow \quad (1,0) \boxplus (1,2) \prec_{\HH^+} (1,1) \boxplus (1,2) \, .
\]
Moreover, $a\boxplus c$ and $b \boxplus c$ may be sets and so one must extend $\prec$ to sets.
We attempt to rectify this in two ways.
Firstly, we can extend $\prec_{\HH^+}$ to a non-strict partial order by defining
\begin{align*}
a \preceq_{\HH^+} b &\Leftrightarrow a = b \text{ or } a \prec_{\HH^+} b \, , \\
&\Leftrightarrow \0 \in b \boxplus -a \text{ or } b \boxplus -a \subseteq \HH^+ \, .
\end{align*}
Secondly, we extend $\preceq_{\HH^+}$ to sets by saying that $A \preceq_{\HH^+} B$ if and only if $a \preceq_{\HH^+} b$ for all $a \in A$ and $b \in B$, or $A = B$.
The following result shows that restricted to stringent hyperfields, this gives a total order that is also compatible with addition.
\begin{proposition}
    Let $\HH$ be an ordered hyperfield with ordering $\HH^+$, then $\preceq_{\HH^+}$ is a non-strict partial order.
    Moreover, if $\HH$ is stringent then $\preceq_{\HH^+}$ is a total order that is compatible with addition.
\end{proposition}
\begin{proof}
All the properties of Proposition~\ref{prop:partial+order} carry over, just remains to show that $\preceq_{\HH^+}$ is compatible with addition when $\HH$ is stringent.
To ease notation, we drop the subscript from $\preceq$.

If $a = b$ then we immediately have that $a \boxplus c = b \boxplus c$.
Therefore we assume that $a \prec b$, in particular $a \neq b$.
When $c \neq -a, -b$, then $a \boxplus c$ and $b \boxplus c$ are singletons, as such
\[
b \boxplus -a \subseteq (b \boxplus c) \boxplus (-a \boxplus -c) \in \HH^+ \, .
\]
We show the case when $c = -a$, the proof for $c = -b$ is very similar.

When $\HH = \FF \rtimes G$, we use the identity $x \boxplus -x \boxplus x = x$: this is straightforward to verify over this hyperfield.
Then for any $k \in a \boxplus -a$, we have
\[
(b \boxplus -a) \boxplus -k \subseteq b \boxplus -a \boxplus a \boxplus -a = b \boxplus -a \subseteq \HH^+ \, .
\]
This implies that $k \prec b \boxplus -a$ for all $k \in a \boxplus -a$, and so $a \boxplus -a \prec b \boxplus -a$.

When $\HH = \SS \rtimes G$, we split into two cases: $-a \prec b$ and $-a \succ b$.
In the former case, we observe that $b \boxplus -a$ and $b \boxplus a$ are contained in $\HH^+$.
Using this along with the identity $x \boxplus x = x$ in $\SS \rtimes G$, for any $k \in a \boxplus -a$ we have
\[
(b \boxplus -a) \boxplus -k \subseteq (b \boxplus -a) \boxplus (-a \boxplus a) = b \boxplus b \boxplus -a \boxplus a \subseteq \HH^+ \, .
\]
As such, $a \boxplus -a \prec b \boxplus -a$.

In the latter case, we have that $a \prec b \prec -a$.
This implies that $-a \boxplus b$ and $-a \boxplus -b$ are both contained in $\HH^+$.
As this remains invariant regardless of the sign of $b$, addition on $\SS \rtimes G$ implies that
\[
-a \boxplus b = -a \boxplus -b = -a \in \HH^+ \, .
\]
Moreover, for any $k \in a \boxplus -a$, we have
\[
-a \boxplus b \boxplus -k = -a \boxplus -k = \begin{cases}
    -a \in \HH^+ & k \neq -a \\
    a \boxplus -a & k = -a \\
\end{cases} \, ,
\]
i.e. $k \preceq b \boxplus -a$ for all $k \in a \boxplus -a$, and therefore $a \boxplus -a \preceq b \boxplus -a$.

% Let $a \preceq b$, then for all $c \in \HH$ we have
% \[
% b \boxplus -a \subseteq b \boxplus c \boxplus -c \boxplus -a \, .
% \]
% If $a = b$ then clearly $a \boxplus c = b \boxplus c$.
% If $a \prec b$, then $b \boxplus c \boxplus -c \boxplus -a$ contains at least one element of $\HH^+$.
% If it is a singleton then $a \boxplus c \prec_{\HH^+} b \boxplus c$, if it is a set then it must contain $\0$ by stringency, and so $a \boxplus c = b \boxplus c$.
% \james{I dont see how this final step works? I dont think that it implies they have to be equal but just have to have non-empty intersection?}
\end{proof}
If $\HH^+$ is implicit, we shall drop the subscript from $\prec$ and $\preceq$ from now on.

\begin{example}\label{ex:balanced+order}
We return to Example~\ref{ex:signed+tropical+order} and the ordering on the signed tropical hyperfield $\TT\RR$.
The strict order $\prec_{\TT\RR^+}$ defined in~\eqref{eq:signed+tropical+order} can extended to a non-strict order $\preceq_{\TT\RR^+}$ via $a \preceq b$ if and only if $a \prec b$ or $a =b$.
Moreover, we can now partially order sets in $\TT\RR$ that arise as sums of elements and their inverses:
\begin{align*}
(1,b) \boxplus (-1,b) \preceq_{\TT\RR^+} (a,c) \quad &\Leftrightarrow \quad a = 1 \, , \, b \leq c \, , \\
(a,b) \preceq_{\TT\RR^+} (1,c) \boxplus (-1,c) \quad &\Leftrightarrow \quad a = -1 \, , \, b \geq c \, .
\end{align*}
Note that no two distinct sets are comparable under this ordering.

This gives rise to the partial order relation on the symmetrized tropical semiring~\cite{AGG:14,LV}.
This is the semiring obtained by adjoining a copy of the `balanced' tropical numbers to the signed tropical semiring: these balanced tropical numbers precisely correspond to the sets of the form $a \boxplus -a$.
%This is a special case of the order relation on $\SS \rtimes G$ where $G = \RR$.
\end{example}

\begin{remark}
    Compatibility with addition has much more complexity over hyperfields, due to issues with extending $\preceq$ to sets.
    As we have defined it, $\preceq$ forms an partial order on the level of sets but where most sets are incomparable.
    Other extensions to sets that allow for more comparability have been investigated, but may not give an order relation; see~\cite[Section 2]{LinziStoj} for example.
    We also sidestep some additional complexity by only considering compatibility with addition on stringent hyperfields, as we only ever compare sets when there is some cancellation. 
    % As $a \boxplus c$ and $b \boxplus c$ are only ever both sets when $a = b = -c$, we clearly have $a \boxplus c = b \boxplus c$.
    % One cannot compare general sets in $\HH$ with $\preceq$, else we get
    % \[
    % \0 \in (a \boxplus -a) \boxplus (b \boxplus -b) \quad \Rightarrow \quad a \boxplus -a = b \boxplus -b \quad \forall a,b \in \HH \, ,
   %  \]
    %which is clearly not true.
\end{remark}

\begin{remark}
%\ben{Add citations to the following. Note that~\cite{LinziStoj} deliberately doesn't cite the original paper that suggests this ordering, so I'm a bit hesitant to do so, but perhaps there are other places?}
    There are other natural ways one can define an order relation on $\HH$ from $\HH^+$, but each has its drawbacks.
    %For example,~\cite{} proposed the relation
    For example, one can consider the relation
    \[
    a \leq b \Leftrightarrow (b \boxplus -a) \cap (\HH^+ \cup \{\0\}) \neq \emptyset \, .
    \]
    When $\HH$ is stringent, this defines exactly the same order relation as $\preceq$.
    When $\HH$ is not stringent, it was noted in~\cite{LinziStoj} that this relation does not satisfy anti-symmetry as the following example demonstrates.
    
    Recall the hyperfield $\HH = \QQ^\times/(\QQ^\times)^2$ introduced in Example~\ref{ex:rational+quotient}.
    We observe that $\bar{2} \leq \bar{3}$ and $\bar{3} \leq \bar{2}$:
    \begin{align*}
        1 = 3\cdot 1 + (-2)\cdot 1 \in 3\cdot(\QQ^\times)^2 + (-2)\cdot(\QQ^\times)^2 \Rightarrow 2\cdot(\QQ^\times)^2 \leq 3\cdot(\QQ^\times)^2 \\
        5 = 2\cdot 4 + (-3)\cdot 1 \in 2\cdot(\QQ^\times)^2 + (-3)\cdot(\QQ^\times)^2 \Rightarrow 3\cdot(\QQ^\times)^2 \leq 2\cdot(\QQ^\times)^2
    \end{align*}
    However $2\cdot(\QQ^\times)^2 \neq 3\cdot(\QQ^\times)^2$: this contradicts anti-symmetry.
\end{remark}

We end this section by relating order preserving homomorphisms and order relations on hyperfields.
In the following, let $\preceq_i$ denote the order relation on $\HH_i$ induced by $\HH^+_i$.
\begin{proposition}\label{prop:ord-f-rel-pf}
Let $f: \HH_1 \rightarrow \HH_2$ be a homomorphism where $\HH_2$ is stringent.
Then $f$ is order preserving if and only if
\[
a \preceq_1 b \, \Rightarrow \, f(a) \preceq_2 f(b) \quad \forall a,b \in \HH_1 \, .
\]
\end{proposition}
\begin{proof}
First assume that $f$ is order preserving.
If $a = b$ then $f(a) = f(b)$ immediately.
By definition $a \prec_1 b$ if and only if $b \boxplus -a \subseteq \HH_1^+$. 
As $f$ is order preserving 
\[
f(b \boxplus -a) \subseteq f(b) \boxplus -f(a) \, \Rightarrow \, \left(f(b) \boxplus -f(a)\right) \cap \HH_2^+ \neq \emptyset \, .
\]
As $\HH_2$ is stringent, either $f(b) \boxplus -f(a) \subseteq \HH^+_2$ or $\0 \in f(b) \boxplus -f(a)$, i.e. $f(a) \preceq_2 f(b)$.

To see the converse, let $a \in \HH_1^+$.
Then $\0_1 \prec_1 a$, implying that $\0_2 \preceq_2 f(a)$.
As $f$ is a hyperfield homomorphism, the only element that maps to $\0_2$ is $\0_1$ and so $\0_2 \prec_2 f(a)$, implying that $f(a) \in \HH_2^+$.
\end{proof}

\begin{remark}
    The condition that $\HH_2$ is stringent cannot be loosened, as the following example shows.
    Consider the hyperfield $\HH = \QQ/(\QQ^\times)^2$ from Example~\ref{ex:rational+quotient}.
    The quotient map $\tau\colon\QQ \rightarrow\HH$ is an order-preserving homomorphism, as $(\QQ^\times)^2 \subseteq \QQ^+$.
    However, the order relation is not preserved: Example~\ref{ex:rational+quotient} showed that $2\cdot(\QQ^\times)^2$ and $3\cdot(\QQ^\times)^2$ are incomparable by $\prec$.
    Moreover, $2\cdot(\QQ^\times)^2 \neq 3\cdot(\QQ^\times)^2$, and so two elements are incomparable in $\HH$, despite $2 \leq 3$ over $\QQ$.
\end{remark}

\bibliographystyle{plain}
\bibliography{References}

\end{document}